\documentclass[11pt]{amsart}
\usepackage{cgeometry}
\usepackage{amsthm}
\usepackage{caption}
\usepackage{multirow}
\usepackage{subcaption}

\usepackage[final]{microtype}
\newcommand\blfootnote[1]{%
  \begingroup
  \renewcommand\thefootnote{}\footnote{#1}%
  \addtocounter{footnote}{-1}%
  \endgroup
}

\usepackage{xurl}
\hypersetup{breaklinks=true}

\hypersetup{
	linkcolor=magenta,
	citecolor=teal!50!green,
	pdftitle = {Template},
	pdfauthor = {Y Zhang}
}

\bibliography{reference}
\title[]{A structure-preserving parametric approximation for anisotropic geometric flows via an $\alpha$-surface energy matrix}

\author[W. Bao]{Weizhu Bao}\thanks{W. Bao (\href{mailto:matbaowz@nus.edu.sg}{\texttt{matbaowz@nus.edu.sg}}), \textsc{Department of Mathematics, National University of Singapore, 119076, Singapore.}}

\author[Y. Li]{Yifei Li}\thanks{Y. Li (\href{mailto:yifei.li@mnf.uni-tuebingen.de}{\texttt{yifei.li@mnf.uni-tuebingen.de}}), \textsc{Mathematisches Institut, Universit\"{a}t T\"{u}bingen, Auf der Morgenstelle 10, 72076, T\"{u}bingen, Germany.}}

\author[W. Ying]{Wenjun Ying}\thanks{W. Ying (\href{mailto:wying@sjtu.edu.cn}{\texttt{wying@sjtu.edu.cn}}), \textsc{School of Mathematical Sciences and Institute of Natural Sciences, Shanghai Jiao Tong University, Shanghai 200240, China.}}

\author[Y. Zhang]{Yulin Zhang}\thanks{Y. Zhang (\href{mailto:yulin.zhang@sjtu.edu.cn}{\texttt{yulin.zhang@sjtu.edu.cn}}), \textsc{School of Mathematical Sciences and Institute of Natural Sciences, Shanghai Jiao Tong University, Shanghai 200240, China.}}

\begin{document}

\begin{abstract}
	In this paper, we propose a structure-preserving parametric approximation for curvature flows with general anisotropic effects. By introducing a hyperparameter $\alpha$, we construct a surface energy matrix $\hat{\boldsymbol{G}}_k^\alpha(\theta)$, which encompasses all existing and potential formulations into a unified form. A fully discrete parametric approximation for anisotropic curvature flows, which exactly preserves the area decay rate, is proposed based on this unified construction. A local energy estimate-based analytical framework is adopted to provide a comprehensive proof of energy stability for all variants of the fully discrete schemes, and it is shown that $\alpha=-1$ is the unique choice that achieves the optimal energy stability condition $3\hat{\gamma}(\theta)\geq\hat{\gamma}(\theta-\pi)$. A novel perspective is proposed for general anisotropic curvature flow by interpreting the normal velocity as a mapping dependent on both geometric quantities and the underlying curve. This viewpoint gives rise to a natural and unified discretization framework in which energy stability is consistently ensured. 
\end{abstract}

\blfootnote{
\textbf{\textit{Keywords---}}geometric flows, parametric finite element method, anisotropy surface energy, structure-preserving, area conservation, energy-stable.
}

\blfootnote{\textbf{\textit{MSC}}:	65M60, 65M12, 35K55, 53C44
}

 \normalsize 

\maketitle

% \tableofcontents

\section{Introduction} 

Curvature-driven evolution of curves and surfaces is fundamental to applications in image processing \cite{alvarez1993axioms,clarenz2000anisotropic,sapiro1994affine}, materials science \cite{fonseca2014shapes,einstein2015equilibrium,gurtin2002interface,randolph2007controlling,taylor1994linking} and solid-state physics \cite{jiang2012phase,thompson2012solid,jiang2016solid,jiang2018solid,ye2010mechanisms}. In crystalline materials, the underlying lattice structure naturally leads to the direction-dependent surface energy density. Such anisotropic effects are especially important when studying the evolution of crystal shapes or thin films, where the interfacial dynamics are strongly influenced by the material's internal symmetry. Understanding and characterizing such anisotropic effects on the evolution of curves and surfaces is therefore crucial for both theoretical analysis and practical applications.

\begin{figure}[htbp]
  \centering
  \includegraphics[width=0.618\textwidth]{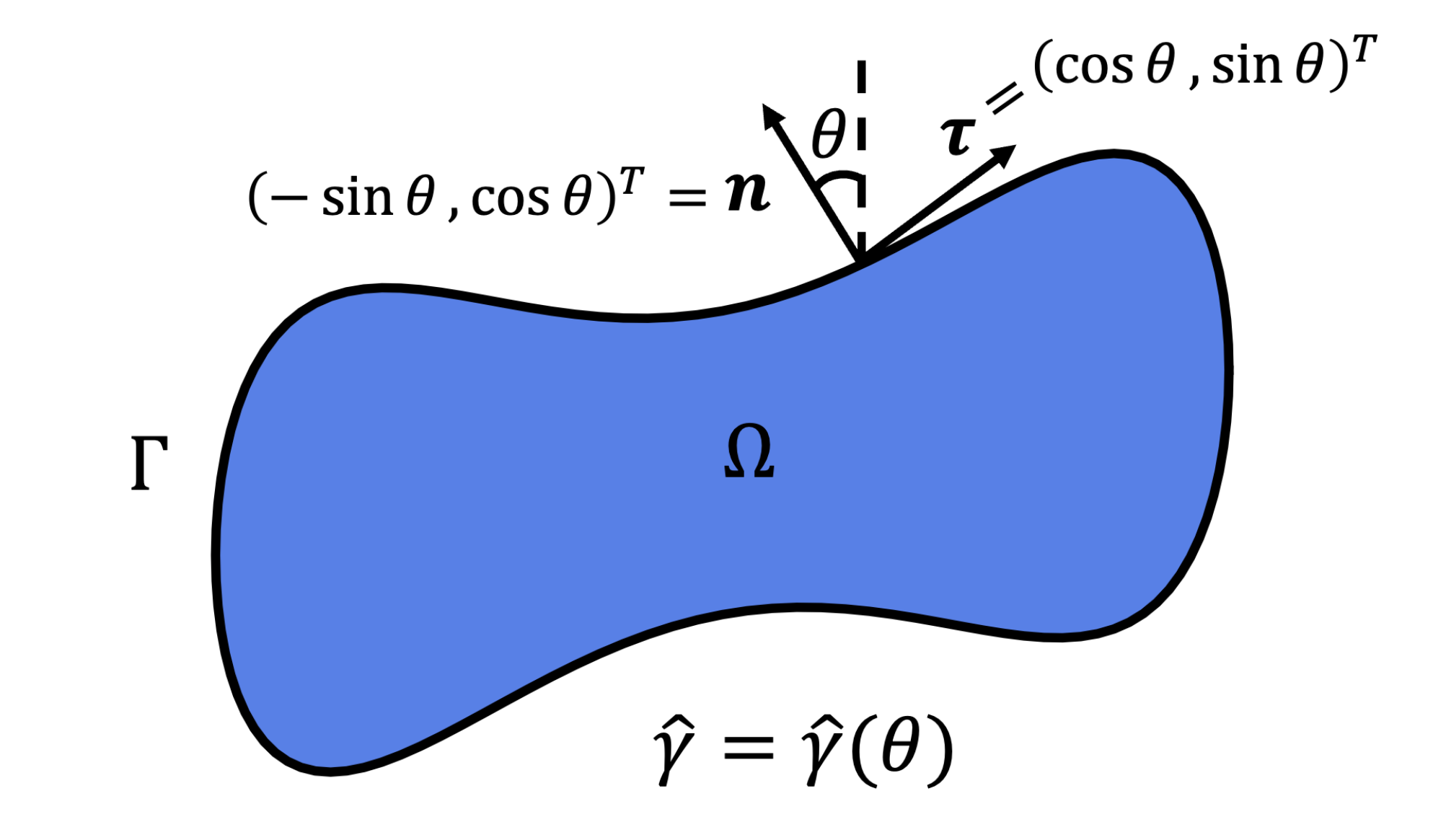}
  \caption{An illustration of an evolving closed curve with an anisotropic surface energy density $\hat{\gamma}(\theta)$. Here, $\theta$ denotes the angle between the $y$-axis and the unit outward normal vector $\boldsymbol{n}=\boldsymbol{n}(\theta)\coloneqq(-\sin\theta,\cos\theta)^T$. The unit tangent vector is $\boldsymbol{\tau}=\boldsymbol{\tau}(\theta)\coloneqq(\cos\theta,\sin\theta)^T$.}
  \label{fig:evolvingcurve}
\end{figure}

As illustrated in Figure~\ref{fig:evolvingcurve}, suppose $\Gamma\coloneqq\Gamma(t)\subset\mathbb{R}^2$ is an evolving closed two-dimensional (2D) curve associated with a given anisotropic surface energy density $\hat{\gamma}(\theta)>0$, where $\theta\in 2\pi\mathbb{T}\coloneqq\mathbb{R}/2\pi\mathbb{Z}$ is the angle between the $y$-axis and the unit outward normal vector $\boldsymbol{n}=\boldsymbol{n}(\theta)\coloneqq(-\sin\theta,\cos\theta)^T$. The evolution is driven by the weighted curvature $\mu\coloneqq\kappa_{\gamma}=\left[\hat{\gamma}(\theta)+\hat{\gamma}^{\prime\prime}(\theta)\right]\kappa$ given by \cite{taylor1992ii}, where $\kappa$ is the classical curvature. This weighted curvature can also be viewed as the first variation of the total free energy $W(\Gamma)$ defined by
\begin{equation}
  W(\Gamma)\coloneqq\int_{\Gamma}\hat{\gamma}(\theta)\,\mathrm{d}s,
\end{equation}
that is,
\begin{equation}\label{eqn:weighted curvature}
  \mu=\frac{\delta W(\Gamma)}{\delta\Gamma}=\lim_{\varepsilon\to0}\frac{W(\Gamma^\varepsilon)-W(\Gamma)}{\varepsilon},
\end{equation}
where $\Gamma^\varepsilon$ is a small perturbation of $\Gamma$. When there is no anisotropic effect, i.e. $\hat{\gamma}(\theta)\equiv\text{const}$, $\mu$ reduces to $\mu=\kappa$.

Consider the anisotropic geometric evolution of closed curves in $\mathbb{R}^2$ with normal velocity $V_n$. Several well-known anisotropic geometric flows, including the anisotropic curvature flow, area-conserved anisotropic curvature flow, and anisotropic surface diffusion, are given by: \begin{equation}\label{eqn:geo flows}
  V_n=\left\{\begin{array}{ll}
    -\mu, & \text{anisotropic curvature flow},\\
    -\mu+\lambda, & \text{area-conserved anisotropic curvature flow},\\
    \partial_{ss}\mu, & \text{anisotropic surface diffusion},
  \end{array}\right.
\end{equation} where $\lambda\coloneqq\int_{\Gamma}\mu\,\mathrm{d}s/|\Gamma|$ is the Lagrange multiplier ensuring that the area of the region enclosed by $\Gamma$ is conserved. The anisotropic geometric flows are related to gradient flows of anisotropic energy functionals such as $W(\Gamma)$, and therefore typically exhibit geometric properties such as energy dissipation and area conservation. Given the wide and profound applications of anisotropic geometric flows, developing a systematic framework of structure-preserving numerical schemes becomes particularly important.

Various numerical methods for curvature-driven problems have been conducted in the past few decades. For example, the level-set method \cite{burger2007level,osher2001level,maxwell2025level}, the phase-field method \cite{du2020phase,brassel2011modified}, the marker particle method \cite{du2010tangent,wang2016modeling}, the finite element method \cite{deckelnick2005computation,deckelnick2005fully}, the evolving surface finite element method (ESFEM) \cite{kovacs2019convergent,kovacs2021convergent,hu2022evolving}, and the parametric finite element method (PFEM) \cite{dziuk1990algorithm,barrett2007variational,li2021energy,bao2017parametric,dziuk1994convergence}. Among these approaches, the energy-stable PFEM (ES-PFEM) proposed by Barrett, Garcke, and N{\"u}rnberg \cite{barrett2007variational,barrett2008parametric}, commonly referred to as the BGN method, has gained significant attention owing to its unconditional energy stability and favorable mesh quality. The BGN method was successfully applied to a wide variety of isotropic curvature-driven problems, including the mean curvature flow \cite{barrett2008parametric}, the surface diffusion \cite{barrett2007parametric}, the multiphase flow \cite{garcke2023structure,garcke2025variational}, the Stefan problem \cite{eto2025parametric}, the Mullins-Sekerka problem \cite{barrett2010stable}, and the evolution of open curves in solid-state dewetting \cite{zhao2021energy,bao2023energy}, consistently demonstrating robust computational performance. A key factor behind the success of BGN-type methods is, instead of approximating the curvature $\kappa$ itself, they work with the {\textit{curvature vector}} $\kappa\boldsymbol{n}$ through the geometric identity
\begin{equation}\label{eqn:curvature identity}
  \kappa\boldsymbol{n}=-\partial_{ss}\boldsymbol{X} = -\partial_s\Bigl(\boldsymbol{I}_d\partial_s\boldsymbol{X}\Bigr),
\end{equation}
here $\boldsymbol{I}_d$ is the $d$-dimensional identity matrix. This approach provides a natural framework for achieving energy stability in isotropic curvature-driven problems.  For more detailed discussions of the BGN-type methods, we refer the reader to the comprehensive review \cite{barrett2020parametricbook} by Barrett et al. 

% Another important geometric property commonly associated with curvature-driven problems is the area/mass conservation, which refers to the invariance of the area enclosed by the evolving curve throughout the evolution. This property appears in models such as SALK, surface diffusion, the Mullins-Sekerka problem and many other curvature-driven phenomena, highlighting the need for numerical methods that rigorously preserve area at the fully discretized level. In \cite{bao2021structure}, Bao and Zhao incorporated an averaged normal vector between two successive time steps into the BGN method and successfully developed a structure-preserving PFEM (SP-PFEM) that achieves volume conservation while maintaining unconditional energy stability at the fully discrete level. Their technique has been widely adopted in the numerical simulation of various volume-conserving curvature-driven problems \cite{bao2022volume,bao2021structure,li2023symmetrized} and has offered valuable insights into addressing volume preservation in the numerical approximation of geometric evolution equations.

There have been numerous attempts to extend the ES-PFEM for isotropic geometric flows to the anisotropic setting. A common feature of these approaches is introducing a suitable surface energy matrix in place of $\boldsymbol{I}_d$ in \eqref{eqn:curvature identity}, yielding an analogous identity for the weighted curvature vector $\mu\boldsymbol{n}$. Barrett et al.\ first achieved this for Riemannian-like surface energies $\hat{\gamma}(\theta)=\left(\boldsymbol{n}(\theta)\cdot G\boldsymbol{n}(\theta)\right)^{1/2}$ using the matrix $\frac{1}{\hat{\gamma}(\theta)}\text{det}(G)\,G^{-1}$ \cite{barrett2008numerical,barrett2010parametric}. In \cite{li2021energy}, Li and Bao constructed a surface energy matrix $\hat{\boldsymbol{G}}(\theta)=\hat{\gamma}(\theta)I_2-\boldsymbol{n}(\theta)\boldsymbol{\xi}(\theta)^T+\boldsymbol{\xi}(\theta)\boldsymbol{n}(\theta)^T$ based on the Cahn-Hoffman $\boldsymbol{\xi}$-vector, achieving the first extension to general surface energies, though with restrictive conditions on $\hat{\gamma}(\theta)$. In \cite{bao2023symmetrized2D}, Bao, Jiang, and Li introduced a stabilizing function $k(\theta)$ and a symmetric surface energy matrix $\hat{\boldsymbol{Z}}_k(\theta)=\hat{\gamma}(\theta)I_2-\boldsymbol{n}\boldsymbol{\xi}^T-\boldsymbol{\xi}\boldsymbol{n}^T+k(\theta)\boldsymbol{n}\boldsymbol{n}^T$, proving unconditional energy stability under the symmetry condition $\hat{\gamma}(\theta)=\hat{\gamma}(\theta-\pi)$. Through refined analysis \cite{bao2025structure,li2025structure}, this condition was later relaxed to $3\hat{\gamma}(\theta)\geq\hat{\gamma}(\theta-\pi)$. Similar stabilization techniques were also applied to $\hat{\boldsymbol{G}}(\theta)$, significantly improving its original stability conditions \cite{bao2024structure,zhang2025stabilized}.

The above studies reveal that different formulations of the surface energy matrix lead to different analyses. Moreover, the energy stability conditions for the resulting discrete schemes vary significantly and remain to be further refined. The main objective of this paper is to provide a unified analysis for the energy stability of SP-PFEMs that apply to all possible formulations of the surface energy matrix, through which we derive optimal energy stability conditions, and to conduct a systematic comparison of their computational performance. Our main contributions are as follows:
\begin{itemize}
  \item \textbf{Unified surface energy matrix.} We introduce a hyperparameter $\alpha\in\mathbb{R}$ and construct the unified $\alpha$-surface energy matrix
  \begin{equation}
    \hat{\boldsymbol{G}}_k^\alpha(\theta)\coloneqq\hat{\gamma}(\theta)I_2-\boldsymbol{n}\boldsymbol{\xi}^T+\alpha\boldsymbol{\xi}\boldsymbol{n}^T+k(\theta)\boldsymbol{n}\boldsymbol{n}^T,
  \end{equation}
  which encompasses all existing formulations \cite{bao2025unified,zhang2025stabilized,bao2023symmetrized2D,li2025structure} as special cases and exhausts all potential constructions (see Remark~\ref{eqn:unity of surf mat}).
  
  \item \textbf{Optimal energy stability conditions.} We establish that the symmetric choice $\alpha=-1$ is the only formulation achieving unconditional energy stability under the necessary and sufficient condition (see Remark~\ref{rmk:optimality of lee}).
  \begin{equation}\label{eqn:common cond}
   3\hat{\gamma}(\theta)-\hat{\gamma}(\theta-\pi)\geq 0,\qquad\forall\theta\in 2\pi\mathbb{T}.
  \end{equation}
  All other formulations require the strictly stronger condition $3\hat{\gamma}(\theta)-\hat{\gamma}(\theta-\pi)>0$ unless additional constraints (e.g., $\hat{\gamma}^\prime(\theta^*)=0$) are imposed.
  
  \item \textbf{Unified velocity discretization.} We extend the SP-PFEM framework to general anisotropic curvature-driven problems with normal velocity of the form
  \begin{equation}
    V_n=\mathfrak{F}(\mu),
  \end{equation}
  where $\mathfrak{F}$ is a mapping depending on the weighted curvature $\mu$. This formulation naturally encompasses all velocities in \eqref{eqn:geo flows} and yields a unified discretization framework that ensures energy stability.
\end{itemize}
In addition, we conduct extensive numerical experiments to demonstrate the computational efficiency of the proposed method and investigate the effect of the parameter $\alpha$. The results demonstrate that the method effectively captures anisotropic curve evolution with robustness across different values of $\alpha$. Moreover, the experiments reveal several interesting phenomena in anisotropic geometric flows.

The structure of this paper is as follows: In section~\ref{sec:geo evo eqn}, we introduce a hyperparameter $\alpha$ to establish a unified construction for all possible surface energy matrices and derive a conservative variational formulation for the anisotropic curvature flow. A full discretization by SP-PFEM is proposed in section~\ref{sec:SP-PFEM}. Concurrently, we state the structure-preserving property of the method. Section~\ref{sec:proof of energy stability} offers a proof of the energy stability of SP-PFEM. Extensions to other anisotropic curvature-driven problems are discussed in section~\ref{sec:generalization}. We report extensive numerical experiments in section~\ref{sec:numer} to validate the accuracy, efficiency, structure-preserving property and robustness of the proposed SP-PFEM. Finally, we conclude the paper in section~\ref{sec:conclusion}.

\section{Anisotropic curvature flow and its variational formulation}\label{sec:geo evo eqn} 

\subsection{The geometric PDE}

Suppose the evolving curve $\Gamma(t)$ is parameterized as $\Gamma(t):\boldsymbol{X}(s,t)=(x(s,t),y(s,t))^T\in\mathbb{R}^2$, where $s$ is the time-dependent arc-length parameter. Then the geomertic evolution equation of the anisotropic curvature flow in \eqref{eqn:geo flows} can be described as follows: \begin{subequations}\label{eqn:original geo evo eqn}
  \begin{align}
    &\partial_t\boldsymbol{X}=-\mu\boldsymbol{n},\qquad 0<s<L(t),\qquad 0\leq t\leq T,\\
    &\mu=\left[\hat{\gamma}(\theta)+\hat{\gamma}^{\prime\prime}(\theta)\right]\kappa.
  \end{align}
\end{subequations} Here, $L(t)=|\Gamma(t)|$ denotes the length of $\Gamma(t)$, $T$ represents the maximum existing time. 

It is noted that during the curve evolution, the velocity component in the tangential direction only affects the parameterization of the curve, without altering its geometric shape. Consequently, it suffices to prescribe the normal velocity $V_n=\boldsymbol{n}\cdot\partial_t\boldsymbol{X}$ in the normal direction. Building on this observation, by allowing tangential motion, an equivalent formulation of the anisotropic curvature flow to that in \eqref{eqn:original geo evo eqn} can be stated as the following geometric PDE: \begin{subequations}\label{eqn:geo PDE}
  \begin{align}
    &\boldsymbol{n}\cdot\partial_t\boldsymbol{X}=-\mu,\qquad 0<s<L(t)\qquad 0\leq t\leq T,\label{eqn:aniso curvature flow}\\
    &\mu=\left[\hat{\gamma}(\theta)+\hat{\gamma}^{\prime\prime}(\theta)\right]\kappa.\label{eqn:repres of weighted curvature}
  \end{align}
\end{subequations} 

\subsection{A unified $\alpha$-surface energy matrix}

To derive a conservative formulation for the anisotropic curvature flow \eqref{eqn:geo PDE}, the following unified $\alpha$-surface energy matrix is introduced: \begin{equation}\label{eqn:surf energy mat}
  \hat{\boldsymbol{G}}_k^\alpha(\theta)\coloneqq\hat{\gamma}(\theta)I_2-\boldsymbol{n}(\theta)\boldsymbol{\xi}(\theta)^T+\alpha\boldsymbol{\xi}(\theta)\boldsymbol{n}(\theta)^T+k(\theta)\boldsymbol{n}(\theta)\boldsymbol{n}(\theta)^T,
\end{equation} with $\boldsymbol{\xi}(\theta)=\hat{\gamma}(\theta)\boldsymbol{n}(\theta)-\hat{\gamma}^\prime(\theta)\boldsymbol{\tau}(\theta)$ being the Cahn-Hoffman $\boldsymbol{\xi}$-vector, $\alpha\in\mathbb{R}$ and $k\colon 2\pi\mathbb{T}\to\mathbb{R}$ is a pre-determined stablizing function.

\begin{lemma}
  For the weighted curvature $\mu$ defined in \eqref{eqn:weighted curvature}, the following geomertic identity holds: \begin{equation}\label{eqn:weighted curvature identity}
    \mu\boldsymbol{n}+\partial_s\Bigl(\hat{\boldsymbol{G}}_k^\alpha(\theta)\partial_s\boldsymbol{X}\Bigr)=\boldsymbol{0},\qquad\forall \alpha\in\mathbb{R}.
  \end{equation}
\end{lemma}

\begin{proof}
  From \cite{jiang2019sharp} or \cite[Theorem 2.1]{zhang2025stabilized}, it is known that \begin{equation}\label{eqn:weighted curvature vec}
    \mu\boldsymbol{n}=-\partial_s\Bigl(\hat{\gamma}(\theta)\partial_s\boldsymbol{X}+\hat{\gamma}^\prime(\theta)\boldsymbol{n}\Bigr).
  \end{equation} Thus it remains to prove that \begin{equation}
    \hat{\boldsymbol{G}}_k^\alpha(\theta)\partial_s\boldsymbol{X}=\hat{\gamma}(\theta)\partial_s\boldsymbol{X}+\hat{\gamma}^\prime(\theta)\boldsymbol{n}.
  \end{equation} Noting that $\boldsymbol{n}=(-\sin\theta,\cos\theta)^T,\boldsymbol{\tau}=(\cos\theta,\sin\theta)^T$, we have \begin{equation}\label{eqn:Gk}
    \hat{\boldsymbol{G}}_k^\alpha(\theta)=\hat{\gamma}(\theta)I_2+\hat{\gamma}^\prime(\theta)(\boldsymbol{n}\boldsymbol{\tau}^T-\alpha\boldsymbol{\tau}\boldsymbol{n}^T)+\left(k(\theta)+(\alpha-1)\hat{\gamma}(\theta)\right)\boldsymbol{n}\boldsymbol{n}^T.
  \end{equation}
  
  Combining \eqref{eqn:Gk} with the facts $\partial_s\boldsymbol{X}=\boldsymbol{\tau}$ and $\boldsymbol{n}^T\partial_s\boldsymbol{X}\equiv 0$ yields \begin{equation}\label{eqn:surf mat act on grad X}
    \begin{aligned}
      \hat{\boldsymbol{G}}^\alpha_k(\theta)\partial_s\boldsymbol{X}&=\hat{\gamma}(\theta)\partial_s\boldsymbol{X}+\hat{\gamma}^\prime(\theta)\boldsymbol{n}+\left(-\alpha\hat{\gamma}^\prime(\theta)\boldsymbol{\tau}\boldsymbol{n}^T+\left(k(\theta)+(\alpha-1)\hat{\gamma}(\theta)\right)\boldsymbol{n}\boldsymbol{n}^T\right)\partial_s\boldsymbol{X}\\
      &=\hat{\gamma}(\theta)\partial_s\boldsymbol{X}+\hat{\gamma}^\prime(\theta)\boldsymbol{n}.
    \end{aligned}
  \end{equation} This proves the lemma.
\end{proof}

Applying the identity \eqref{eqn:weighted curvature identity}, a \textit{strong formulation} for the geometric PDE \eqref{eqn:aniso curvature flow}--\eqref{eqn:repres of weighted curvature} is expressed as follows: \begin{subequations}\label{eqn:strong form}
  \begin{align}
    &\boldsymbol{n}\cdot\partial_t\boldsymbol{X}+\mu=0,\qquad 0<s<L(t),\qquad\forall 0\leq t\leq T,\label{eqn:strong form a}\\
    &\mu\boldsymbol{n}+\partial_s\Bigl(\hat{\boldsymbol{G}}_k^\alpha(\theta)\partial_s\boldsymbol{X}\Bigr)=\boldsymbol{0},\label{eqn:strong form b}
  \end{align}
\end{subequations} where $L(t)$ is the length of the evolving curve $\Gamma(t)$.

\begin{remark}
  When $\hat{\gamma}(\theta)\equiv 1$, the weighted curvature $\mu$ reduces to the classical curvature $\kappa$, and by taking $k(\theta)\equiv 1-\alpha$, we have the surface energy matrix $\hat{\boldsymbol{G}}_k^\alpha(\theta)\equiv I_2$. Thus \eqref{eqn:strong form} will reduce to the standard formulation by BGN method for mean curvature flow \cite{barrett2007variational}.
\end{remark}

\begin{remark}
  By selecting different parameter $\alpha$, the strong form \eqref{eqn:strong form} will generate different formulations for the weighted curvature $\mu$. For example, when $\alpha=-1$, it offers the symmetrized formulations in \cite{bao2023symmetrized2D,li2025structure}; by setting $\alpha=0,k(\theta)\equiv 0$, we will obtain the formulation proposed in \cite{barrett2008variational}; and it will lead to the formulations in \cite{li2021energy,bao2024structure,zhang2025stabilized} by choosing $\alpha=1$.
\end{remark}

\begin{remark}\label{eqn:unity of surf mat}
  For any surface energy matrix $\hat{\boldsymbol{G}}(\theta)$ satisfying $\mu\boldsymbol{n}=-\partial_s\left(\hat{\boldsymbol{G}}(\theta)\partial_s\boldsymbol{X}\right)$, it can be obtained by $\hat{\boldsymbol{G}}_k^\alpha(\theta)$ in \eqref{eqn:surf energy mat}. To see this, consider $K(\boldsymbol{\tau})\coloneqq\left\{\boldsymbol{A}\in\mathbb{R}^{2\times 2}\mid \boldsymbol{A}\boldsymbol{\tau}=\boldsymbol{0}\right\}$ with $\text{dim}\,K(\boldsymbol{\tau})=2$ and $\boldsymbol{\xi}\boldsymbol{n}^T,\boldsymbol{n}\boldsymbol{n}^T \in K(\boldsymbol{\tau})$. For the anisotropic case, as $\boldsymbol{\xi}\nparallel\boldsymbol{n}$, $K(\boldsymbol{\tau})=\{\alpha\boldsymbol{\xi}\boldsymbol{n}^T+k\boldsymbol{n}\boldsymbol{n}^T\mid \alpha,k\in\mathbb{R}\}$. Since $\mu \boldsymbol{n}=-\partial_s\left(\hat{\boldsymbol{G}}(\theta)\partial_s\boldsymbol{X}\right) = -\partial_s\left(\hat{\boldsymbol{G}}_0^0(\theta)\partial_s\boldsymbol{X}\right)$, we know that $\hat{\boldsymbol{G}}(\theta)-\hat{\boldsymbol{G}}_0^0(\theta)\in K(\boldsymbol{\tau})$. Therefore, we can deduce that $\hat{\boldsymbol{G}}(\theta)\in \hat{\boldsymbol{G}}_0^0(\theta)+K(\boldsymbol{\tau})=\{\hat{\boldsymbol{G}}_k^\alpha(\theta)\mid \alpha,k\in\mathbb{R}\}$. This shows that all possible surface energy matrix can be expressed in the form of \eqref{eqn:surf energy mat}. 
\end{remark}

\subsection{Variational formulation}

To obtain a variational formulation based on the strong form \eqref{eqn:strong form}, we suppose the evolving $\Gamma(t)$ is parametrized by a time-independent parameter $\rho$ over a fixed domain $\mathbb{I}\coloneqq[0,1]$, i.e. \begin{equation}
  \Gamma(t)\colon \boldsymbol{X}(\rho,t)=(x(\rho,t),y(\rho,t))^T,\qquad \forall \rho\in\mathbb{I},\,\,t\in[0,T].
\end{equation} Thus the arc-length parameterization can be computed as $s(\rho,t)=\int_0^\rho|\partial_r\boldsymbol{X}(r,t)|\,\mathrm{d}r$. In this paper, we make no distinction between $\boldsymbol{X}(s,t)$ and $\boldsymbol{X}(\rho,t)$ and assume the parametrization by $\rho$ is always regular, i.e. $\frac{1}{C}\leq |\partial_rs(\rho,t)|\leq C,\,\,\forall\rho\in\mathbb{I}$ for a constant $C>1$.

For an evolving curve $\Gamma(t)$, the $L^2$-space with respect to $\Gamma(t)$ is defined as follows: \begin{equation}
  L^2(\mathbb{I})\coloneqq\left\{u\colon\mathbb{I}\to\mathbb{R}\mid\int_{\Gamma(t)}|u(s)|^2\,\mathrm{d}s=\int_{\mathbb{I}}|u(s(\rho,t))|^2\partial_\rho s\,\mathrm{d}\rho<+\infty\right\},
\end{equation} equipped with the inner product \begin{equation}
  \Bigl(u,v\Bigr)_{\Gamma(t)}\coloneqq\int_{\Gamma(t)}u(s)v(s)\,\mathrm{d}s=\int_{\mathbb{I}}u(s(\rho,t))v(s(\rho,t))\partial_\rho s\,\mathrm{d}\rho,\qquad\forall u,v\in L^2(\mathbb{I}).
\end{equation} And the corresponding Sobolev spaces are given as \begin{subequations}
  \begin{align}
    &H^1(\mathbb{I})\coloneqq\left\{u\in L^2(\mathbb{I})\mid \partial_{\rho}u\in L^2(\mathbb{I})\right\},\\
    &H^1_p(\mathbb{I})\coloneqq \left\{u\in H^1(\mathbb{I})\mid u(0)=u(1)\right\}.
  \end{align}
\end{subequations}

Multiplying test functions $\varphi\in H^1_p(\mathbb{I})$ to \eqref{eqn:strong form a} and $\boldsymbol{\omega}=(\omega_1,\omega_2)^T\in [H^1_p(\mathbb{I})]^2$ to \eqref{eqn:strong form b}, respectively. Then integrating over $\Gamma(t)$ and applying integration by parts, we obtain the \textit{variational formulation} for the strong form \eqref{eqn:strong form} as follows: Suppose the initial closed curve $\Gamma(0)\coloneqq\boldsymbol{X}(\cdot,0)=(x(\cdot,0),y(\cdot,0))^T\in[H^1_p(\mathbb{I})]^2$ and the initial weighted curvature $\mu(\cdot,0)\coloneqq\mu_0(\cdot)\in H_p^1(\mathbb{I})$, for any $t>0$, find the solution $\left(\boldsymbol{X}(\cdot,t)=(x(\cdot,t),y(\cdot,t))^T,\mu(\cdot,t)\right)\in[H^1_p(\mathbb{I})]^2\times H^1_p(\mathbb{I})$ such that \begin{subequations}\label{eqn:var form}
  \begin{align}
    &\Bigl(\boldsymbol{n}\cdot\partial_t\boldsymbol{X},\varphi\Bigr)_{\Gamma(t)}+\Bigl(\mu,\varphi\Bigr)_{\Gamma(t)}=0,\qquad\forall\varphi\in H^1_p(\mathbb{I}),\label{eqn:var form a}\\
    &\Bigl(\mu\boldsymbol{n},\boldsymbol{\omega}\Bigr)_{\Gamma(t)}-\Bigl(\hat{\boldsymbol{G}}_k^\alpha(\theta)\partial_s\boldsymbol{X},\partial_s\boldsymbol{\omega}\Bigr)_{\Gamma(t)}=0,\qquad\forall\boldsymbol{\omega}\in [H^1_p(\mathbb{I})]^2.\label{eqn:var form b}
  \end{align}
\end{subequations}

\subsection{Properties of the variational formulation}

Denote $A(t)$ as the total area enclosed by the evolving curve $\Gamma(t)$, and $W(t)$ as the total interfacial energy, which are formally defined as \begin{equation}
  A(t)\coloneqq\int_{\Gamma(t)}y(s,t)\partial_sx(s,t)\,\mathrm{d}s,\qquad W(t)\coloneqq\int_{\Gamma(t)}\hat{\gamma}(\theta)\,\mathrm{d}s.
\end{equation}

To derive the area decay rate and the energy dissipation rate, we need the time derivative of the inclination angle $\partial_t\theta$ as well as the transport lemma.

\begin{lemma}
  For the time derivative of the inclination angle $\theta$, the following geometric identity holds: \begin{equation}\label{eqn:inclination angle identity}
    \partial_t\theta=\partial_s(\partial_t\boldsymbol{X})\cdot\boldsymbol{n}.
  \end{equation}
\end{lemma}

\begin{proof}
  Firstly, consider the time derivative of $|\partial_\rho\boldsymbol{X}|=\sqrt{(\partial_\rho x)^2+(\partial_\rho y)^2}$. Then \begin{equation}\label{eqn:time derivative of |grad X|}
    \begin{aligned}
      \partial_t|\partial_\rho\boldsymbol{X}|&=\frac{\partial_{\rho} x\partial_t(\partial_{\rho} x)+\partial_{\rho} y\partial_t(\partial_{\rho} y)}{\sqrt{(\partial_\rho x)^2+(\partial_\rho y)^2}}=\frac{\partial_{\rho}\boldsymbol{X}}{|\partial_\rho\boldsymbol{X}|}\cdot\frac{\partial_\rho(\partial_t\boldsymbol{X})}{|\partial_\rho\boldsymbol{X}|}|\partial_\rho\boldsymbol{X}|\\
      &=\partial_s\boldsymbol{X}\cdot\partial_s(\partial_t\boldsymbol{X})|\partial_{\rho}\boldsymbol{X}|.
    \end{aligned}
  \end{equation} Therefore, by \eqref{eqn:time derivative of |grad X|}, \begin{equation}
    \begin{aligned}
      \partial_s(\partial_t\boldsymbol{X})&=\frac{1}{|\partial_\rho\boldsymbol{X}|}\partial_\rho(\partial_t\boldsymbol{X})=\frac{1}{|\partial_\rho\boldsymbol{X}|}\partial_t\left(|\partial_\rho\boldsymbol{X}|(\cos\theta,\sin\theta)^T\right)\\
      &=\partial_s\boldsymbol{X}\cdot\partial_s(\partial_t\boldsymbol{X})(\cos\theta,\sin\theta)^T+(-\sin\theta,\cos\theta)^T\partial_t\theta.
    \end{aligned}
  \end{equation} Combining with the fact that $\boldsymbol{n}=(-\sin\theta,\cos\theta)^T$ gives the desired result.
\end{proof}

\begin{lemma}[Transport lemma, \cite{zhang2025stabilized}]\label{lma:transport lma}
  Suppose $\Gamma(t)$ is a two-dimensional piecewise $C^1$ curve parameterized by $\boldsymbol{X}(\rho,t)$, function $f:\Gamma(t)\times\mathbb{R}^+\to\mathbb{R}$ is differentiable. Then \begin{equation}
    \frac{\mathrm{d}}{\mathrm{d}t}\int_{\Gamma(t)}f\,\mathrm{d}s=\int_{\Gamma(t)}\partial_tf+f\partial_s(\partial_t\boldsymbol{X})\cdot\partial_s\boldsymbol{X}\,\mathrm{d}s.
  \end{equation}
\end{lemma}

\begin{proposition}[Area decay rate and energy dissipation]
  Let $(\boldsymbol{X}(\cdot,t),\mu(\cdot,t))$ be the solution to the variational formulation \eqref{eqn:var form}. Then the total area $A(t)$ obeys the following decay rate and the total interficial energy $W(t)$ is dissipative, i.e., \begin{equation}
    \frac{\mathrm{d}A}{\mathrm{d}t}=-\Bigl(\mu,1\Bigr)_{\Gamma(t)},\qquad W(t)\leq W(t^\prime)\leq W(0),\qquad\forall t\geq t^\prime\geq 0.
  \end{equation}
\end{proposition}

\begin{proof}
  Denote $\Omega(t)$ the region enclosed by $\Gamma(t)$. Applying the Reynolds' transport theorem \cite{reynolds1983papers} and taking $\varphi=1$ in \eqref{eqn:var form a}, \begin{equation}
    \begin{aligned}
      \frac{\mathrm{d}}{\mathrm{d}t}A(t)&=\frac{\mathrm{d}}{\mathrm{d}t}\int_{\Omega(t)}\mathrm{d}x\,\mathrm{d}y=\int_{\Gamma(t)}\boldsymbol{n}\cdot\partial_t\boldsymbol{X}\,\mathrm{d}s\\
      &=\Bigl(\boldsymbol{n}\cdot\partial_t\boldsymbol{X},1\Bigr)_{\Gamma(t)}=-\Bigl(\mu,1\Bigr)_{\Gamma(t)}.
    \end{aligned}
  \end{equation}

  For the energy dissipation, differenting $W(t)$ with respect to $t$ by Lemma~\ref{lma:transport lma}, \begin{equation}
    \begin{aligned}
      \frac{\mathrm{d}}{\mathrm{d}t}W(t)=\int_{\Gamma(t)}\hat{\gamma}^\prime(\theta)\partial_t\theta+\hat{\gamma}(\theta)\partial_s(\partial_t\boldsymbol{X})\cdot\partial_s\boldsymbol{X}\,\mathrm{d}s.
    \end{aligned}
  \end{equation} Combining with the geometric identity \eqref{eqn:inclination angle identity} of $\theta$, we have \begin{equation}\label{eqn:energy time derivative}
    \begin{aligned}
      \frac{\mathrm{d}}{\mathrm{d}t}W(t)&=\int_{\Gamma(t)}\left[\hat{\gamma}(\theta)\partial_s\boldsymbol{X}+\hat{\gamma}^\prime(\theta)\boldsymbol{n}\right]\cdot\partial_s(\partial_t\boldsymbol{X})\,\mathrm{d}s\\
      &=\int_{\Gamma(t)}\hat{\boldsymbol{G}}_k^\alpha(\theta)\partial_s\boldsymbol{X}\cdot\partial_s(\partial_t\boldsymbol{X})\,\mathrm{d}s.
    \end{aligned}
  \end{equation} By setting $\varphi=\mu$ in \eqref{eqn:var form a} and $\boldsymbol{\omega}=\partial_t\boldsymbol{X}$ in \eqref{eqn:var form b}, and considering \eqref{eqn:energy time derivative}, then \begin{equation}
    \begin{aligned}
      \frac{\mathrm{d}}{\mathrm{d}t}W(t)&=\Bigl(\hat{\boldsymbol{G}}_k^\alpha(\theta)\partial_s\boldsymbol{X},\partial_s(\partial_t\boldsymbol{X})\Bigr)_{\Gamma(t)}=\Bigl(\mu\boldsymbol{n},\partial_t\boldsymbol{X}\Bigr)_{\Gamma(t)}=-\Bigl(\mu,\mu\Bigr)_{\Gamma(t)}\leq 0.
    \end{aligned}
  \end{equation}
\end{proof}

\section{A structure-preserving parametric finite element approximation}\label{sec:SP-PFEM}

In this section, a parametric finite element full discretization is proposed based on the variational form \eqref{eqn:var form}, which preserves both the area decay rate and the energy dissipation.

Let $N>2$ be a positive integer and define the mesh size as $h=1/N$. Consider a uniform partition of the interval as $\mathbb{I}=[0,1]\coloneqq\cup_{j=1}^NI_j$ with $I_j=[\rho_{j-1},\rho_j],\rho_j\coloneqq jh$ for $j=0,1,\ldots,N$. The piecewise linear finite element spaces are defined as \begin{subequations}
  \begin{align}
    &\mathbb{K}^h\coloneqq\left\{u^h\in C(\mathbb{I})\mid u^h|_{I_j}\in\mathcal{P}^1(I_j),\,\forall j=1,2,\ldots,N\right\}\subseteq H^1(\mathbb{I}),\\
    &\mathbb{K}_p^h\coloneqq\left\{u^h\in\mathbb{K}^h\mid u^h(0)=u^h(1)\right\}\subseteq H^1_p(\mathbb{I}),
  \end{align}
\end{subequations} where $\mathcal{P}^1(I_j)$ represents the space of all polynomials on the interval $I_j$ with degree at most $1$. Additionally, we define $0=t_0<t_1<\cdots<t_M=T$ be a uniform partition of $[0,T]$ with time steps $t_m=m\tau,\,\,\tau\coloneqq T/M$.

Let $\Gamma^m\coloneqq\boldsymbol{X}^m(\cdot)\in[\mathbb{K}_p^h]^2$ be an approximation of $\Gamma(t_m)=\boldsymbol{X}(\cdot,t_m=m\tau)$, satisfies the following non-degeneracy condition: 
\begin{equation}
  \min_{1\leq j\leq N}|\boldsymbol{h}_j^m|>0,\qquad\forall m=0,1,\ldots,M,
\end{equation} 
where $\boldsymbol{h}_j^m\coloneqq\boldsymbol{X}^m(\rho_j)-\boldsymbol{X}^m(\rho_{j-1})$ and $\displaystyle u(\rho_j^\pm)=\lim_{\rho\to\rho^\pm_j}u(\rho)$. Similarly, $\mu^m\in\mathbb{K}_p^h$ denotes an approximation of $\mu(\cdot,t_m)$. 

The mass-lumped inner product for $u,v\in\mathbb{K}_p^h$ is defined as follows: \begin{equation}
  \Bigl(u,v\Bigr)^h_{\Gamma^m}\coloneqq \sum_{j=1}^N\frac{|\boldsymbol{h}_j^m|}{2}\left(u(\rho_j^-)v(\rho_j^-)+u(\rho_{j-1}^+)v(\rho_{j-1}^+)\right).
\end{equation}
And the discretized derivative $\partial_{s^m}$ on $\Gamma^m$ is defined as \begin{equation}
  \partial_{s^m} f|_{I_j}\coloneqq \frac{f(\rho_j)-f(\rho_{j-1})}{|\boldsymbol{h}_j^m|},\qquad\forall j=1,2,\cdots,N.
\end{equation} 
The above definitions can be directly extended to vector-valued functions. 

Discrete geometric quantities such as the unit tangential vector $\boldsymbol{\tau}^m$, the unit normal vector $\boldsymbol{n}^m$ and the inclination angle $\theta^m$ of the polygonal curve $\Gamma^m$ can be computed as \begin{equation}
  \boldsymbol{\tau}^m|_{I_j}=\frac{\boldsymbol{h}_j^m}{|\boldsymbol{h}_j^m|}\coloneqq\boldsymbol{\tau}_j^m,\qquad\boldsymbol{n}^m|_{I_j}=-\frac{(\boldsymbol{h}_j^m)^\perp}{|\boldsymbol{h}_j^m|}\coloneqq \boldsymbol{n}_j^m,
\end{equation} and \begin{equation}
  \theta^m|_{I_j}=\theta_j^m,\qquad\text{where}\,\,\theta_j^m\,\,\text{satisfying}\,\,(\cos\theta_j^m,\sin\theta_j^m)^T=\boldsymbol{\tau}_j^m.
\end{equation}

Now we are ready to present a structure-preserving parametric finite element approximation for the anisotropic curvature flow \eqref{eqn:aniso curvature flow}--\eqref{eqn:repres of weighted curvature}: 

Suppose $\Gamma^0\in[\mathbb{K}_p^h]^2$ be the initial approximation given by $\boldsymbol{X}^0(\rho_j)=\boldsymbol{X}(\rho_j,0),\,\,j=1,2,\cdots,N$. Find the solution $\left(\boldsymbol{X}^{m+1}(\cdot)=(x^m(\cdot),y^m(\cdot))^T,\mu^{m+1}(\cdot)\right)\in[\mathbb{K}_p^h]^2\times\mathbb{K}_p^h,\,\,m=0,1,\cdots,M-1$, such that \begin{subequations}\label{eqn:SP-PFEM}
  \begin{align}
    &\Bigl(\boldsymbol{n}^{m+\frac{1}{2}}\cdot\frac{\boldsymbol{X}^{m+1}-\boldsymbol{X}^m}{\tau},\varphi^h\Bigr)^h_{\Gamma^m}+\Bigl(\mu^{m+1},\varphi^h\Bigr)^h_{\Gamma^m}=0,\qquad\forall\varphi^h\in\mathbb{K}_p^h,\label{eqn:SP-PFEM a}\\
    &\Bigl(\mu^{m+1}\boldsymbol{n}^{m+\frac{1}{2}},\boldsymbol{\omega}^h\Bigr)^h_{\Gamma^m}-\Bigl(\hat{\boldsymbol{G}}_k^\alpha(\theta^m)\partial_{s^m}\boldsymbol{X}^{m+1},\partial_{s^m}\boldsymbol{\omega}^h\Bigr)^h_{\Gamma^m}=0,\qquad\forall\boldsymbol{\omega}^h\in[\mathbb{K}_p^h]^2,\label{eqn:SP-PFEM b}
  \end{align}
\end{subequations} where \begin{equation}
  \boldsymbol{n}^{m+\frac{1}{2}}\coloneqq-\frac{1}{2}\left(\partial_{s^m}\boldsymbol{X}^{m}+\partial_{s^m}\boldsymbol{X}^{m+1}\right)^\perp=-\frac{1}{2|\partial_\rho\boldsymbol{X}^m|}\left(\partial_\rho\boldsymbol{X}^m+\partial_\rho\boldsymbol{X}^{m+1}\right)^\perp.
\end{equation} 

% \begin{remark}
%   By selecting different parameter $a$, the proposed SP-PFEM \eqref{eqn:SP-PFEM} will generate different numerical schemes. For example, when $a=-1$, it will give us the SP-PFEMs in \cite{barrett2008numerical},\cite{bao2023symmetrized2D} and \cite{li2025structure}; when $a=1$, our method will reduce to the SP-PFEMs in \cite{bao2024structure} and \cite{zhang2025stabilized}.
% \end{remark}

\begin{remark}
  The choice of $\boldsymbol{n}^{m+\frac{1}{2}}$ is motivated by the area-preserving PFEM proposed by Bao and Zhao for surface diffusion \cite{bao2021structure}. It rigorously characterizes the area difference of a evolving polygonal curve between two discrete time levels, making it crucial for maintaining the area decay rate.
\end{remark}

\begin{remark}
    Although it does not seem easy to prove the unique solvability of the scheme \eqref{eqn:SP-PFEM}, if we replace $\boldsymbol{n}^{m+\frac{1}{2}}$ with $\boldsymbol{n}^m$, the scheme becomes linear. In this case, through an argument similar to that in \cite{barrett2007parametric,barrett2020parametricbook}, we can prove that the linear system admits a unique solution under relatively weak assumptions on $\boldsymbol{n}^m$.
\end{remark}

\begin{remark}
An extension of the proposed SP-PFEM to higher-order finite element spaces is an important topic for future research. Based on the recent strictly structure-preserving isoparametric finite element method for isotropic curvature flow developed in \cite{garcke2025isoparametric}, we expect that such an extension is feasible. In particular, if one adopts a similar discretization of the normal vector $\boldsymbol{n}$ and establishes the corresponding local energy estimates, then the present method may be generalized to higher-order elements while still preserving key structural properties such as area conservation and energy dissipation.
\end{remark}

\subsection{Area decay rate preserving and energy dissipation properties of the SP-PFEM}

Let $A^m$ be the area enclosed by the polygonal curve $\Gamma^m$, and $W^m$ be the total interfacial energy, which are given by \begin{equation}
  A^m\coloneqq\frac{1}{2}\sum_{j=1}^N\left(x^m(\rho_j)-x^m(\rho_{j-1})\right)\left(y^m(\rho_j)+y^m(\rho_{j-1})\right),\qquad W^m\coloneqq \sum_{j=1}^N\hat{\gamma}(\theta_j^m)|\boldsymbol{h}_j^m|.
\end{equation}

Our main result is stated as follows:

\begin{theorem}[structure-preserving]\label{thm:main result}
  The SP-PFEM \eqref{eqn:SP-PFEM} is area decay rate preserving, i.e., \begin{equation}
    \frac{A^{m+1}-A^m}{\tau}=-\Bigl(\mu^{m+1},1\Bigr)^h_{\Gamma^m},\qquad\forall 0\leq m\leq M-1.
  \end{equation}
  
  Moreover, if $\alpha=-1$ and $\hat{\gamma}(\theta)\in C^2(2\pi\mathbb{T})$ satisfies \begin{equation}\label{eqn:energy cond rlx}
    3\hat{\gamma}(\theta)-\hat{\gamma}(\theta-\pi)\geq 0,\qquad\forall\theta\in 2\pi\mathbb{T}.
  \end{equation} 
  Then the SP-PFEM \eqref{eqn:SP-PFEM} is unconditionally energy stable with sufficiently large $k(\theta)$, i.e., \begin{equation}
    W^{m+1}\leq W^m\leq\cdots\leq W^0,\qquad\forall 0\leq m\leq M-1.
  \end{equation} 
  Otherwise, for $\alpha\neq-1$, the energy stability condition \eqref{eqn:energy cond rlx} needs to be strengthened to \begin{equation}\label{eqn:energy cond}
    3\hat{\gamma}(\theta)-\hat{\gamma}(\theta-\pi)>0,\qquad\forall\theta\in2\pi\mathbb{T}.
  \end{equation}
\end{theorem}

We only provide a proof for area decay rate preserving property here, and leave the energy dissipation part to the next section.

\begin{proof}
  Similar to derivations of \cite[Theorem~2.1]{bao2021structure}, we have
  \begin{equation}\label{eqn:area decay rate formula}
    A^{m+1}-A^m=\Bigl(\boldsymbol{n}^{m+\frac{1}{2}}\cdot(\boldsymbol{X}^{m+1}-\boldsymbol{X}^m),1\Bigr)^h_{\Gamma^m}.
  \end{equation} Thus by taking $\varphi^h\equiv 1$ in \eqref{eqn:SP-PFEM a}, we obtain \begin{equation}
    A^{m+1}-A^m=-\tau\Bigl(\mu^{m+1},1\Bigr)^h_{\Gamma^m}.
  \end{equation} This establishes the area decay rate formula.
\end{proof}

\section{Proof of the unconditional energy stability}\label{sec:proof of energy stability}

\subsection{Minimal stabilizing function and unconditional energy stability}\label{subsec:proof of energy stability}

To prove the unconditional energy dissipation property of the SP-PFEM \eqref{eqn:SP-PFEM}, we first introduce a minimal stabilizing function $k^\alpha_\text{min}(\theta)$, which is defined as \begin{equation}\label{eqn:def of min stab func}
  k^\alpha_\text{min}(\theta)\coloneqq k_0^\alpha(\theta)-(\alpha-1)\hat{\gamma}(\theta),\qquad\forall\theta\in 2\pi\mathbb{T},
\end{equation} where $k_0^\alpha(\theta)$ is given by \begin{equation}\label{eqn:def of min stab func b}
  k_0^\alpha(\theta)\coloneqq \inf\left\{a\geq 0\mid 4\hat{\gamma}(\theta)P_{\alpha, a}(\phi,\theta)\geq Q^2_\alpha(\phi,\theta),\,\,\forall\phi\in 2\pi\mathbb{T}\right\}.
\end{equation} Here, $P_{\alpha, a},Q_\alpha$ are two auxiliary functions defined as \begin{subequations}\label{eqn:aux func}
  \begin{align}
    &P_{\alpha, a}(\phi,\theta)\coloneqq \hat{\gamma}(\theta)+\frac{\alpha-1}{2}\hat{\gamma}^\prime(\theta)\sin 2\phi+a\sin^2\phi,\label{eqn:aux func a}\\
    &Q_\alpha(\phi,\theta)\coloneqq \hat{\gamma}(\theta-\phi)+\hat{\gamma}(\theta)\cos\phi+\alpha\hat{\gamma}^\prime(\theta)\sin\phi.\label{eqn:aux func b}
  \end{align}
\end{subequations}

\begin{lemma}\label{asp:stab func well-defined}
  Suppose that $\hat{\gamma}(\theta)\in C^2(2\pi\mathbb{T})$ satisfies the condition \eqref{eqn:energy cond} for $\alpha=-1$, and condition \eqref{eqn:energy cond rlx} otherwise. Then the minimal stabilizing function $k_\text{min}^\alpha(\theta)$ exists, i.e. \begin{equation}\label{eqn:stab func well-defined}
    k_\text{min}^\alpha(\theta)<+\infty,\qquad\forall\theta\in 2\pi\mathbb{T}.
  \end{equation}
\end{lemma}

The proof of Lemma~\ref{asp:stab func well-defined} will be provided in Section~\ref{sec:existence of min stab func}. For the moment, we assume its validity. The following local energy estimate follows from Lemma~\ref{asp:stab func well-defined}: \begin{lemma}[local energy estimate]\label{lma:loc energy est}
  Assume that \eqref{eqn:stab func well-defined} holds true. For any $\boldsymbol{p},\boldsymbol{q}\in\mathbb{R}^2\backslash\{\boldsymbol{0}\}$, let $\boldsymbol{p}=|\boldsymbol{p}|(\cos\varphi,\sin\varphi)^T,\boldsymbol{q}=|\boldsymbol{q}|(\cos\theta,\sin\theta)^T$. Then for sufficiently large $k(\theta)$, \begin{equation}\label{eqn:loc energy est}
    \frac{1}{|\boldsymbol{q}|}\Bigl(\hat{\boldsymbol{G}}_k^\alpha(\theta)\boldsymbol{p}\Bigr)\cdot(\boldsymbol{p}-\boldsymbol{q})\geq \hat{\gamma}(\varphi)|\boldsymbol{p}|-\hat{\gamma}(\theta)|\boldsymbol{q}|.
  \end{equation}
\end{lemma}

\begin{proof}
  Recall the definition \eqref{eqn:surf energy mat} of $\hat{\boldsymbol{G}}_k^\alpha(\theta)$, then \begin{equation}
    \begin{aligned}
      &\frac{1}{|\boldsymbol{q}|}\Bigl(\hat{\boldsymbol{G}}_k^\alpha(\theta)\boldsymbol{p}\Bigr)\cdot\boldsymbol{p}\\
      &=\frac{|\boldsymbol{p}|^2}{|\boldsymbol{q}|}\left(\hat{\boldsymbol{G}}_k^\alpha(\theta)\left(\begin{array}{c}
        \cos\varphi\\
        \sin\varphi
      \end{array}\right)\right)\cdot\left(\begin{array}{c}
        \cos\varphi\\
        \sin\varphi
      \end{array}\right)\\
      &=\frac{|\boldsymbol{p}|^2}{|\boldsymbol{q}|}\left(\hat{\gamma}(\theta)+(\alpha-1)\hat{\gamma}^\prime(\theta)\sin(\theta-\varphi)\cos(\theta-\varphi)+\left(k(\theta)+(\alpha-1)\hat{\gamma}(\theta)\right)\sin^2(\theta-\varphi)\right)\\
      &=\frac{|\boldsymbol{p}|^2}{|\boldsymbol{q}|}\left(\hat{\gamma}(\theta)+\frac{\alpha-1}{2}\hat{\gamma}^\prime(\theta)\sin\left(2(\theta-\varphi)\right)+\left(k(\theta)+(\alpha-1)\hat{\gamma}(\theta)\right)\sin^2(\theta-\varphi)\right)\\
      &=\frac{|\boldsymbol{p}|^2}{|\boldsymbol{q}|}P_{\alpha, k + (\alpha - 1)\hat{\gamma}}(\theta-\varphi,\theta),
    \end{aligned}
  \end{equation} and \begin{equation}
    \begin{aligned}
      \frac{1}{|\boldsymbol{q}|}\Bigl(\hat{\boldsymbol{G}}_k^\alpha(\theta)\boldsymbol{p}\Bigr)\cdot\boldsymbol{q}&=|\boldsymbol{p}|\left(\hat{\boldsymbol{G}}_k^\alpha(\theta)\left(\begin{array}{c}
        \cos\varphi\\
        \sin\varphi
      \end{array}\right)\right)\cdot\left(\begin{array}{c}
        \cos\theta\\
        \sin\theta
      \end{array}\right)\\
      &=|\boldsymbol{p}|\left(\hat{\gamma}(\theta)\cos(\theta-\varphi)+\alpha\hat{\gamma}^\prime(\theta)\sin(\theta-\varphi)\right)\\
      &=|\boldsymbol{p}|\left(Q_\alpha(\theta-\varphi,\theta)-\hat{\gamma}(\varphi)\right).
    \end{aligned}
  \end{equation} Suppose that \eqref{eqn:stab func well-defined} is satisfied. Then for sufficiently large $k(\theta)\geq k_\text{min}^\alpha(\theta)$, we know $k(\theta)+(\alpha-1)\hat{\gamma}(\theta)\geq k_\text{min}^\alpha(\theta) + (\alpha-1)\hat{\gamma}(\theta) = k_0^\alpha(\theta)$. Thus, from \eqref{eqn:def of min stab func b}, we have 
  \begin{equation}
    \begin{aligned}
    4\hat{\gamma}(\theta)P_{\alpha, k + (\alpha - 1)\hat{\gamma}}(\theta-\varphi,\theta) &= 4\hat{\gamma}(\theta)P_{\alpha, k_0^\alpha}(\theta-\varphi,\theta) + 4\hat{\gamma}(\theta)\left(k(\theta) - k_\text{min}^\alpha(\theta)\right)\sin^2(\theta-\varphi) \\
    &\geq Q_\alpha^2(\theta-\varphi,\theta),\qquad\forall\varphi\in 2\pi\mathbb{T}.
    \end{aligned}
  \end{equation} Combining with the fact that $\frac{1}{4c}x^2-x\geq-c,\,\,\forall x\in\mathbb{R},\,\,c>0$ gives 
  \begin{equation}
    \begin{aligned}
      &\frac{1}{|\boldsymbol{q}|}\Bigl(\hat{\boldsymbol{G}}_k^\alpha(\theta)\boldsymbol{p}\Bigr)\cdot(\boldsymbol{p}-\boldsymbol{q})\\
      & = \frac{|\boldsymbol{p}|^2}{|\boldsymbol{q}|}P_{\alpha, k + (\alpha - 1)\hat{\gamma}}(\theta-\varphi,\theta) - |\boldsymbol{p}|\left(Q_\alpha(\theta-\varphi,\theta)-\hat{\gamma}(\varphi)\right) \\
      &\geq \left(\frac{1}{4\hat{\gamma}(\theta)|\boldsymbol{q}|}\left(|\boldsymbol{p}|Q_\alpha(\theta-\varphi,\theta)\right)^2-|\boldsymbol{p}|Q_\alpha(\theta-\varphi,\theta)\right)+\hat{\gamma}(\varphi)|\boldsymbol{p}|\\
      &\geq \hat{\gamma}(\varphi)|\boldsymbol{p}|-\hat{\gamma}(\theta)|\boldsymbol{q}|.
    \end{aligned}
  \end{equation}
\end{proof}

\begin{remark}\label{rmk:optimality of lee}
  Following the analysis in \cite{bao2024structure,zhang2025stabilized,li2025structure}, taking $\boldsymbol{p}=-\boldsymbol{q}$ (i.e., $\varphi=\theta-\pi$) in \eqref{eqn:loc energy est} yields the necessary condition $3\hat{\gamma}(\theta)\geq \hat{\gamma}(\theta-\pi)$ for the local energy estimate. Comparing this with condition \eqref{eqn:energy cond rlx} for $\alpha=-1$, we conclude that \eqref{eqn:energy cond rlx} is both necessary and sufficient for the local energy estimate when $\alpha=-1$, and is therefore optimal.
\end{remark}

\begin{remark}
 For the isotropic case, i.e. $\hat{\gamma}(\theta)\equiv 1$, it is easy to see that $k_0^\alpha(\theta)\equiv 0$ for any $\alpha\in\mathbb{R}$. This indicates that the minimal stabilizing function $k_\text{min}^\alpha(\theta)\equiv 1-\alpha$.
\end{remark}

With the local energy estimate \eqref{eqn:loc energy est}, the unconditional energy dissipation property of SP-PFEM \eqref{eqn:SP-PFEM} can be proven under Assumption~\ref{asp:stab func well-defined}: \begin{proof}
  Suppose $k(\theta)$ is sufficiently large such that $k(\theta)\geq k_\text{min}^\alpha(\theta)$. 
  
  Then for $\forall 0\leq m\leq M-1$, \begin{equation}
    \begin{aligned}
      &\Bigl(\hat{\boldsymbol{G}}_k^\alpha(\theta^m)\partial_{s^m}\boldsymbol{X}^{m+1},\partial_{s^m}(\boldsymbol{X}^{m+1}-\boldsymbol{X}^m)\Bigr)^h_{\Gamma^m}\\
      &=\sum_{j=1}^N\left[|\boldsymbol{h}_j^m|\Bigl(\hat{\boldsymbol{G}}_k^\alpha(\theta^m)\frac{\boldsymbol{h}_j^{m+1}}{|\boldsymbol{h}_j^m|}\Bigr)\cdot\frac{\boldsymbol{h}_j^{m+1}-\boldsymbol{h}_j^m}{|\boldsymbol{h}_j^m|}\right]\\
      &=\sum_{j=1}^N\left[\frac{1}{|\boldsymbol{h}_j^m|}\Bigl(\hat{\boldsymbol{G}}_k^\alpha(\theta^m)\boldsymbol{h}_j^{m+1}\Bigr)\cdot(\boldsymbol{h}_j^{m+1}-\boldsymbol{h}_j^m)\right].
    \end{aligned}
  \end{equation} Applying the local energy estimate \eqref{eqn:loc energy est} by letting $\boldsymbol{p}=\boldsymbol{h}_j^{m+1},\boldsymbol{q}=\boldsymbol{h}_j^m$, \begin{equation}\label{eqn:energy difference}
    \begin{aligned}
      &\Bigl(\hat{\boldsymbol{G}}_k^\alpha(\theta^m)\partial_{s^m}\boldsymbol{X}^{m+1},\partial_{s^m}(\boldsymbol{X}^{m+1}-\boldsymbol{X}^m)\Bigr)^h_{\Gamma^m}\\
      &\geq \sum_{j=1}^N\left[\hat{\gamma}(\theta_j^{m+1})|\boldsymbol{h}_j^{m+1}|-\hat{\gamma}(\theta_j^m)|\boldsymbol{h}_j^m|\right]\\
      &=\sum_{j=1}^N\hat{\gamma}(\theta_j^{m+1})|\boldsymbol{h}_j^{m+1}|-\sum_{j=1}^N\hat{\gamma}(\theta_j^m)|\boldsymbol{h}_j^m|\\
      &=W^{m+1}-W^m.
    \end{aligned}
  \end{equation} Therefore, by taking $\varphi^h=\mu^{m+1}$ and $\boldsymbol{\omega}^h=\boldsymbol{X}^{m+1}-\boldsymbol{X}^m$ in \eqref{eqn:SP-PFEM}, we conclude that \begin{equation}
    \begin{aligned}
      W^{m+1}-W^m&\leq \Bigl(\hat{\boldsymbol{G}}_k^\alpha(\theta^m)\partial_{s^m}\boldsymbol{X}^{m+1},\partial_{s^m}(\boldsymbol{X}^{m+1}-\boldsymbol{X}^m)\Bigr)^h_{\Gamma^m}\\
      &=\Bigl(\mu^{m+1}\boldsymbol{n}^{m+\frac{1}{2}},\boldsymbol{X}^{m+1}-\boldsymbol{X}^m\Bigr)^h_{\Gamma^m}\\
      &=-\tau\Bigl(\mu^{m+1},\mu^{m+1}\Bigr)^h_{\Gamma^m}\leq 0.
    \end{aligned}
  \end{equation} This implies the energy dissipation property of the SP-PFEM \eqref{eqn:SP-PFEM} as claimed.
\end{proof}

\subsection{Existence of the minimal stabilizing function}\label{sec:existence of min stab func}

In this section, we analyze the existence of the minimal stabilizing function $k_\text{min}^\alpha(\theta)$. 

Assume that $\hat{\gamma}(\theta)\in C^2(2\pi\mathbb{T})$ satisfies 
\begin{equation}
  3\hat{\gamma}(\theta)-\hat{\gamma}(\theta-\pi)\geq 0,\qquad\forall\theta\in 2\pi\mathbb{T},
\end{equation}
and let $c\coloneqq\inf\limits_{\theta\in 2\pi\mathbb{T}}\bigl[3\hat{\gamma}(\theta)-\hat{\gamma}(\theta-\pi)\bigr]$. Since $k_\text{min}^\alpha(\theta) = k_0^\alpha(\theta) - (\alpha - 1)\hat{\gamma}(\theta)$, the existence of $k_\text{min}^\alpha(\theta)$ is equivalent to the boundedness of $k_0^\alpha(\theta)$. We therefore focus on the latter. Specifically, we establish a necessary and sufficient condition for the boundedness of $k_0^\alpha(\theta)$ in the critical case where $c=0$, and derive an upper bound estimate for $k_0^\alpha(\theta)$ when $c>0$.

\subsubsection{Boundedness of $k_0^\alpha(\theta)$}

\begin{theorem}\label{thm:existence of k0 c=0}
  If $c=0$, then $k_0^\alpha(\theta)$ exists if and only if one of the following conditions holds:
  \begin{enumerate}
    \item $\alpha=-1$;\label{eqn:cond 1}
    \item $\hat{\gamma}^\prime(\theta^*)=0$ whenever $3\hat{\gamma}(\theta^*)=\hat{\gamma}(\theta^*-\pi)$.\label{eqn:cond 2}
  \end{enumerate}
\end{theorem}

\begin{proof}
  \textit{Sufficiency.} The proofs for conditions \eqref{eqn:cond 1} and \eqref{eqn:cond 2} follow arguments similar to those in \cite{li2025structure} and \cite{zhang2025stabilized}, respectively, and are omitted for brevity.
  
  \textit{Necessity.} Suppose $3\hat{\gamma}(\theta^*)=\hat{\gamma}(\theta^*-\pi)$ for some $\theta^*\in 2\pi\mathbb{T}$. Then $\hat{\gamma}(\theta-\pi)/\hat{\gamma}(\theta)$ attains its maximum at $\theta=\theta^*$, which implies
  \begin{equation}
    \left.\frac{d}{d\theta}\left(\frac{\hat{\gamma}(\theta-\pi)}{\hat{\gamma}(\theta)}\right)\right|_{\theta=\theta^*}
    =\frac{\hat{\gamma}^\prime(\theta^*-\pi)\hat{\gamma}(\theta^*)-\hat{\gamma}(\theta^*-\pi)\hat{\gamma}^\prime(\theta^*)}{\hat{\gamma}^2(\theta^*)}
    =\frac{\hat{\gamma}^\prime(\theta^*-\pi)-3\hat{\gamma}^\prime(\theta^*)}{\hat{\gamma}(\theta^*)}=0.
  \end{equation}
  Thus
  \begin{equation}\label{eqn:der formula}
    3\hat{\gamma}^\prime(\theta^*)=\hat{\gamma}^\prime(\theta^*-\pi)\quad\text{whenever}\quad 3\hat{\gamma}(\theta^*)=\hat{\gamma}(\theta^*-\pi).
  \end{equation}
  
  Define $F_{\alpha}(\phi,\theta)\coloneqq4\hat{\gamma}(\theta)P_{\alpha, k_\text{min}^\alpha}(\phi,\theta)-Q^2_\alpha(\phi,\theta)$, we have
  \begin{equation}\label{eqn:F geq 0}
    F_{\alpha}(\phi,\theta)\geq 0,\qquad\forall\phi, \theta\in 2\pi\mathbb{T}.
  \end{equation}
  Applying the mean value theorem to $F_{\alpha}(\phi,\theta^*)$ at $\phi=\pi$, we have
  \begin{equation}\label{eqn:expansion of F}
    \begin{aligned}
      F_{\alpha}(\phi,\theta^*)&=\left(3\hat{\gamma}(\theta^*)-\hat{\gamma}(\theta^*-\pi)\right)\left(\hat{\gamma}(\theta^*)+\hat{\gamma}(\theta^*-\pi)\right)\\
      &\quad+\left[4(\alpha-1)\hat{\gamma}(\theta^*)\hat{\gamma}^\prime(\theta^*)+2\left(\hat{\gamma}(\theta^*-\pi)-\hat{\gamma}(\theta^*)\right)\left(\hat{\gamma}^\prime(\theta^*-\pi)+\alpha\hat{\gamma}^\prime(\theta^*)\right)\right](\phi-\pi)\\
      &\quad+O((\phi-\pi)^2).
    \end{aligned}
  \end{equation}
  Substituting $3\hat{\gamma}(\theta^*)=\hat{\gamma}(\theta^*-\pi)$ and \eqref{eqn:der formula} into \eqref{eqn:expansion of F}, we obtain
  \begin{equation}
    \begin{aligned}
      F_{\alpha}(\phi,\theta^*)&=\left[4(\alpha - 1)\hat{\gamma}(\theta^*)\hat{\gamma}^\prime(\theta^*)+2\left(3\hat{\gamma}(\theta^*) - \hat{\gamma}(\theta^*)\right) \left(3\hat{\gamma}^\prime(\theta^*) + \alpha\hat{\gamma}^\prime(\theta^*)\right)\right](\phi-\pi)\\
      & \qquad  + O((\phi-\pi)^2) \\
      &=8(\alpha+1)\hat{\gamma}(\theta^*)\hat{\gamma}^\prime(\theta^*)(\phi-\pi)++ O((\phi-\pi)^2).
    \end{aligned}
  \end{equation}
  
  Therefore, $(\alpha+1)\hat{\gamma}^\prime(\theta^*)=0$, yielding either $\alpha=-1$ or $\hat{\gamma}^\prime(\theta^*)=0$.
\end{proof}

Theorem~\ref{thm:existence of k0 c=0} unifies the main results from \cite{zhang2025stabilized} and \cite{li2025structure}, and reveals that $\alpha=-1$ is the optimal choice among all surface energy matrices. Indeed, as noted in Remark~\ref{rmk:optimality of lee}, condition \eqref{eqn:energy cond rlx} $3\hat{\gamma}(\theta)-\hat{\gamma}(\theta-\pi)\geq 0$ is both necessary and sufficient for the local energy estimate when $\alpha=-1$. In contrast, for $\alpha \neq -1$, ensuring energy stability of SP-PFEM under condition \eqref{eqn:energy cond rlx} requires either additional assumptions such as $\hat{\gamma}^\prime(\theta^*)=0$, or strengthening the condition to \eqref{eqn:energy cond} $3\hat{\gamma}(\theta)-\hat{\gamma}(\theta-\pi)> 0$.

% By Theorem~\ref{thm:existence of k0 c>0}, we know that for any $a\in\mathbb{R}$, the minimal stabilizing function $k_\text{min}^a(\theta)$ always exists when $\hat{\gamma}(\theta)$ satisfies \eqref{eqn:energy cond}. Therefore, in this case, by combining Lemma~\ref{lma:loc energy est} and the proof in section \ref{subsec:proof of energy stability}, we conclude that: for sufficiently large $k(\theta)\geq k_\text{min}^a(\theta)$, the SP-PFEM \eqref{eqn:SP-PFEM} is unconditionally energy stable. 

\subsubsection{A global upper bound of $k_0^\alpha(\theta)$}

To establish a global upper bound of $k_0^\alpha(\theta)$, we introduce the following lemmas:

\begin{lemma}[\cite{li2025structure}]\label{lma:derivative est}
  Let $f$ be a non-negative $C^2$ function on $2\pi\mathbb{T}$. Then for any positive constant $C\geq \sup\limits_{2\pi\mathbb{T}}|f^{\prime\prime}|$, we have \begin{equation}\label{eqn:derivative est}
    -f(x) \leq f^\prime(x)y +\frac{C}{2}y^2,\qquad\forall x,y\in 2\pi\mathbb{T}.
  \end{equation}
\end{lemma}

\begin{lemma}[Estimation of $Q_\alpha(\phi,\theta)$]\label{lma:est of Q}
  Suppose $c>0$. Then for $Q_\alpha(\phi,\theta)$ defined in \eqref{eqn:aux func b}, we have \begin{equation}\label{eqn:est of Q}
    |Q_\alpha(\phi,\theta)|\leq |P_{\alpha, A}(\phi,\theta)+\hat{\gamma}(\theta)|,\qquad\forall\phi\in 2\pi\mathbb{T},
  \end{equation} where $A(\theta)$ is defined as \begin{subequations}
    \begin{align}
      &A(\theta)\coloneqq \frac{\pi^2}{8}\left(C_\alpha\sup_{2\pi\mathbb{T}}|\hat{\gamma}^{\prime\prime}|+(3|\alpha|+2)|\hat{\gamma}^\prime(\theta)|+\hat{\gamma}(\theta)+B(\theta)\right),\\
      &B(\theta)\coloneqq \frac{2}{c}(\alpha+1)^2|\hat{\gamma}^\prime(\theta)|^2,\qquad C_\alpha=\max\{5,\frac{4}{\pi^2}(2|\alpha|+1)^2\}.
    \end{align}
 \end{subequations} 
\end{lemma}

\begin{proof}
We begin by proving the lower bound of $Q_a(\phi,\theta)$. Applying Lemma~\ref{lma:derivative est}, we obtain \begin{equation}
    \begin{aligned}
      Q_\alpha(\phi,\theta)+P_{\alpha, 0}(\phi,\theta)+\hat{\gamma}(\theta)&=\hat{\gamma}(\theta-\phi)+\hat{\gamma}(\theta)(2+\cos\phi)+\hat{\gamma}^\prime(\theta)(\alpha\sin\phi+\frac{\alpha-1}{2}\sin 2\phi)\\
      &\geq \hat{\gamma}(\theta)+\hat{\gamma}^\prime(\theta)(\alpha+(\alpha-1)\cos\phi)\sin\phi\\
      &\geq \hat{\gamma}(\theta)-(2|\alpha|+1)|\hat{\gamma}^\prime(\theta)||\sin\phi|\\
      &\geq-\frac{(2|\alpha|+1)^2}{2}\sup_{2\pi\mathbb{T}}|\hat{\gamma}^{\prime\prime}|\sin^2\phi.\\
    \end{aligned}
  \end{equation} Combining with the facts $A(\theta)\geq\frac{(2|\alpha|+1)^2}{2}\sup\limits_{2\pi\mathbb{T}}|\hat{\gamma}^{\prime\prime}|$ and noting $P_{\alpha, A}(\phi,\theta)=P_{\alpha, 0}(\phi,\theta)+A(\theta)\sin^2\phi$, we deduce \begin{equation}
    Q_\alpha(\phi,\theta)\geq -P_{\alpha, 0}(\phi, \theta) - \hat{\gamma}(\theta) - A(\theta)\sin^2\phi = -P_{\alpha, A}(\phi,\theta)-\hat{\gamma}(\theta),\qquad\forall\phi\in 2\pi\mathbb{T}.
  \end{equation}

  The remaining task is to determine the upper bound of $Q_\alpha(\phi,\theta)$. We divide the proof into two cases: 

  Case 1: For $|\phi|\leq\frac{\pi}{2}$. By adopting the mean value theorem to $Q_\alpha(\cdot,\theta)-P_{\alpha, 0}(\cdot,\theta)-\hat{\gamma}(\theta)$ on $[0,\phi]$, then there exists a $\xi\in[0,\phi]$ such that \begin{equation}
    \begin{aligned}
      &Q_\alpha(\phi,\theta)-P_{\alpha, 0}(\phi,\theta)-\hat{\gamma}(\theta)\\
      &=\frac{1}{2}\left(\hat{\gamma}^{\prime\prime}(\theta-\xi)-\hat{\gamma}(\theta)\cos\xi-\alpha\hat{\gamma}^\prime(\theta)\sin\xi+2(\alpha-1)\hat{\gamma}^\prime(\theta)\sin 2\xi\right)\phi^2\\
      &\leq \frac{1}{2}\left(\sup_{2\pi\mathbb{T}}|\hat{\gamma}^{\prime\prime}|+(3|\alpha|+2)|\hat{\gamma}^\prime(\theta)|+\hat{\gamma}(\theta)\right)\left(\frac{\pi^2}{4}\sin^2\phi\right)\\
      &\leq A(\theta)\sin^2\phi.
    \end{aligned}
  \end{equation} The penultimate inequality comes from the fact that $|\phi|\leq \frac{\pi}{2}|\sin\phi|,\,\,\forall|\phi|\leq\frac{\pi}{2}$.

  Case 2: For $|\phi-\pi|<\frac{\pi}{2}$. Again, we apply the mean value theorem to $Q_\alpha(\cdot,\theta)-P_{\alpha, 0}(\cdot,\theta)-\hat{\gamma}(\theta)$ on $[\phi,\pi]$. Then there exists a $\xi\in[\phi,\pi]$ such that \begin{equation}
    \begin{aligned}
      &Q_\alpha(\phi,\theta)-P_{\alpha, 0}(\phi,\theta)-\hat{\gamma}(\theta)\\
      &=\left(\hat{\gamma}(\theta-\pi)-3\hat{\gamma}(\theta)\right)-\left(\hat{\gamma}^\prime(\theta-\pi)+(2\alpha-1)\hat{\gamma}^\prime(\theta)\right)(\phi-\pi)\\
      &\quad+\frac{1}{2}\left(\hat{\gamma}^{\prime\prime}(\theta-\xi)-\hat{\gamma}(\theta)\cos\xi-\alpha\hat{\gamma}^\prime(\theta)\sin\xi+2(\alpha-1)\hat{\gamma}^\prime(\theta)\sin 2\xi\right) (\phi-\pi)^2.
    \end{aligned}
  \end{equation} 

   By the definition of $c$, we have $3\hat{\gamma}(\theta)-\hat{\gamma}(\theta-\pi)\geq c$. Applying Lemma~\ref{lma:derivative est} to $3\hat{\gamma}(\theta)-\hat{\gamma}(\theta-\pi)-c$, we obtain
   \begin{equation}
    \begin{aligned}
      &\hat{\gamma}(\theta-\pi)-3\hat{\gamma}(\theta) \\
      & = -c -\left(3\hat{\gamma}(\theta)-\hat{\gamma}(\theta-\pi)-c\right)\\
      &\leq -c + \left(\hat{\gamma}^\prime(\theta-\pi)-3\hat{\gamma}^\prime(\theta)\right)(\phi-\pi)+\frac{1}{2}\sup\limits_{2\pi\mathbb{T}}|3\hat{\gamma}^{\prime\prime}(\theta)-\hat{\gamma}^{\prime\prime}(\theta-\pi)|(\phi-\pi)^2  \\
      & \leq -c +\left(\hat{\gamma}^\prime(\theta-\pi)-3\hat{\gamma}^\prime(\theta)\right)(\phi-\pi)+2\sup\limits_{2\pi\mathbb{T}}|\hat{\gamma}^{\prime\prime}|(\phi-\pi)^2.
    \end{aligned}
   \end{equation}
   Therefore, \begin{equation}\label{eqn:Q est at pi}
    \begin{aligned}
      &Q_\alpha(\phi,\theta)-P_{\alpha, 0}(\phi,\theta)-\hat{\gamma}(\theta)\\
      &\leq -c +\left(\hat{\gamma}^\prime(\theta-\pi)-3\hat{\gamma}^\prime(\theta)\right)(\phi-\pi)+2\sup\limits_{2\pi\mathbb{T}}|\hat{\gamma}^{\prime\prime}|(\phi-\pi)^2 \\
      & \quad -\left(\hat{\gamma}^\prime(\theta-\pi)+(2\alpha-1)\hat{\gamma}^\prime(\theta)\right)(\phi-\pi)\\
      &\quad+\frac{1}{2}\left(\hat{\gamma}^{\prime\prime}(\theta-\xi)-\hat{\gamma}(\theta)\cos\xi-\alpha\hat{\gamma}^\prime(\theta)\sin\xi+2(\alpha-1)\hat{\gamma}^\prime(\theta)\sin 2\xi\right) (\phi-\pi)^2 \\
      &\leq -c-2(\alpha+1)\hat{\gamma}^\prime(\theta)(\phi-\pi)\\
      &\quad+\frac{1}{2}\left(5\sup_{2\pi\mathbb{T}}|\hat{\gamma}^{\prime\prime}|+(3|\alpha|+2)|\hat{\gamma}^\prime(\theta)|+\hat{\gamma}(\theta)\right)(\phi-\pi)^2.
    \end{aligned}
  \end{equation}

  Since $c>0$, the lower order term in the right-hand side of \eqref{eqn:Q est at pi} can be controlled by a quadratic term, i.e. \begin{equation}
    -c-2(\alpha+1)\hat{\gamma}^\prime(\theta)(\phi-\pi)\leq \frac{1}{c}(\alpha+1)^2|\hat{\gamma}^\prime(\theta)|^2(\phi-\pi)^2.
  \end{equation} Combining with the fact that $|\phi-\pi|\leq \frac{\pi}{2}|\sin\phi|,\,\,\forall|\phi-\pi|\leq\frac{\pi}{2}$, we obtain the desired inequality \begin{equation}\label{eqn:Q upper bound}
    Q_\alpha(\phi,\theta)\leq P_{\alpha, A}(\phi,\theta)+\hat{\gamma}(\theta),\qquad\forall \phi\in 2\pi\mathbb{T}.
  \end{equation} 
\end{proof}

\begin{theorem}\label{thm:existence of k0 c>0}
  If $c>0$. For any $\alpha\in\mathbb{R}$, $k_0^\alpha(\theta)$ given in \eqref{eqn:def of min stab func} admits the following upper bound: \begin{equation}\label{eqn:bound of k0}
    k_0^\alpha(\theta)\leq\frac{1}{4\hat{\gamma}(\theta)}\left[A^2(\theta)+4\hat{\gamma}(\theta)A(\theta)+(\alpha-1)^2|\hat{\gamma}^\prime(\theta)|^2\right]<+\infty,\qquad\forall\theta\in 2\pi\mathbb{T}.
  \end{equation}
\end{theorem}

\begin{proof}
  By Lemma~\ref{lma:est of Q}, we have \begin{equation}
    Q_\alpha^2(\phi,\theta)\leq \left(P_{\alpha, A}(\phi,\theta)+\hat{\gamma}(\theta)\right)^2,\qquad\forall\phi\in 2\pi\mathbb{T}.
  \end{equation}

  Recall the definition of $P_{\alpha, a}(\phi,\theta)$ in \eqref{eqn:aux func a}, we know that $P_{\alpha, a}(\phi,\theta)=P_{\alpha, 0}(\phi,\theta)+a(\theta)\sin^2\alpha$. Thus, \begin{equation}
    \begin{aligned}
      &4\hat{\gamma}(\theta)P_{\alpha, a}(\phi,\theta)-Q_\alpha^2(\phi,\theta)\\
      &\geq 4\hat{\gamma}(\theta)(a(\theta)-A(\theta))\sin^2\phi+4\hat{\gamma}(\theta)P_{\alpha, A}(\phi,\theta)-\left(P_{\alpha, A}(\phi,\theta)+\hat{\gamma}(\theta)\right)^2\\
      &= 4\hat{\gamma}(\theta)(a(\theta)-A(\theta))\sin^2\phi-\left(P_{\alpha, A}(\phi,\theta)-\hat{\gamma}(\theta)\right)^2\\
      &=\left(4\hat{\gamma}(\theta)(a(\theta)-A(\theta))-((\alpha-1)\hat{\gamma}^\prime(\theta)\cos\phi+A(\theta)\sin\phi)^2\right)\sin^2\phi\\
      &\geq \left(4\hat{\gamma}(\theta)a(\theta)-A^2(\theta)-(a-1)^2|\hat{\gamma}^\prime(\theta)|^2-4\hat{\gamma}(\theta)A(\theta)\right)\sin^2\phi.
    \end{aligned}
  \end{equation} The last inequality is a direct consequence of the bound $|a\cos\phi+b\sin\phi|\leq\sqrt{a^2+b^2},\,\,\forall\phi\in 2\pi\mathbb{T}$.

  Therfore, for any $a(\theta)\geq\frac{1}{4\hat{\gamma}(\theta)}\left[A^2(\theta)+4\hat{\gamma}(\theta)A(\theta)+(\alpha-1)^2|\hat{\gamma}^\prime(\theta)|^2\right]$, we have \begin{equation}
    4\hat{\gamma}(\theta)P_{\alpha, a}(\phi,\theta)-Q_\alpha^2(\phi,\theta)\geq 0,\qquad\forall\phi\in 2\pi\mathbb{T}.
  \end{equation} This gives \eqref{eqn:bound of k0} by the definition of $k_0^a(\theta)$ in \eqref{eqn:def of min stab func b}, and the proof is complete.
\end{proof}

\section{Generalizations to other geometric flows}\label{sec:generalization}

In this section, we consider two specific geometric flows, namely the area-conserved anisotropic curvature flow and the anisotropic surface diffusion, and present their corresponding SP-PFEMs. At the end of this section, based on the proposed analytical framework, we will discuss how to design structure-preserving algorithms for general normal velocity laws.

\subsection{Area-conserved anisotropic curvature flow}

Similar to \eqref{eqn:strong form}, for the area-conserved anisotropic curvature flow in \eqref{eqn:geo flows}, we have the following conservative strong formulation: \begin{subequations}\label{eqn:strong form acmcf}
  \begin{align}
    &\boldsymbol{n}\cdot\partial_t\boldsymbol{X}+\mu-\lambda(t)=0,\\
    &\mu\boldsymbol{n}+\partial_s\Bigl(\hat{\boldsymbol{G}}_k^\alpha(\theta)\partial_s\boldsymbol{X}\Bigr)=\boldsymbol{0},
  \end{align}
\end{subequations} where $\lambda(t)\coloneqq\int_{\Gamma(t)}\mu\,\mathrm{d}s/|\Gamma(t)|$ is the Lagrange multiplier that ensures the enclosed area remains constant over time.

Suppose the initial closed curve $\Gamma(0)=\boldsymbol{X}(\cdot,0)\in[H_p^1(\mathbb{I})]^2$ and the initial weighted curvature $\mu(\cdot,0)=\mu_0(\cdot)\in H_p^1(\mathbb{I})$ is given. Then a variational formulation based on \eqref{eqn:strong form acmcf} is stated as follows: For any $t>0$, find the solution $\left(\boldsymbol{X}(\cdot,t),\mu(\cdot,t)\right)\in[H_p^1(\mathbb{I})]^2\times H_p^1(\mathbb{I})$ such that \begin{subequations}\label{eqn:weak form acmcf}
  \begin{align}
    &\Bigl(\boldsymbol{n}\cdot\partial_t\boldsymbol{X},\varphi\Bigr)_{\Gamma(t)}+\Bigl(\mu-\lambda(t),\varphi\Bigr)_{\Gamma(t)}=0,\qquad\forall\varphi\in H_p^1(\mathbb{I}),\\
    &\Bigl(\mu\boldsymbol{n},\boldsymbol{\omega}\Bigr)_{\Gamma(t)}-\Bigl(\hat{\boldsymbol{G}}_k^\alpha(\theta)\partial_s\boldsymbol{X},\partial_s\boldsymbol{\omega}\Bigr)_{\Gamma(t)}=0,\qquad\forall\boldsymbol{\omega}\in[H_p^1(\mathbb{I})]^2.
  \end{align}
\end{subequations}

% To design a structure-preserving scheme based on the above variational formulation, we follow ideas in \cite{pei2023structure,bao2024structure} and discretize $\lambda(t)$ with respect to $\Gamma^m$ as: \begin{equation}
%   \lambda^{m,*}\coloneqq\frac{\Bigl(\mu^{m+1},1\Bigr)^h_{\Gamma^m}}{\Bigl(1,1\Bigr)^h_{\Gamma^m}}.
% \end{equation}

Then the SP-PFEM for area-conserved anisotropic curvature flow in \eqref{eqn:geo flows} is as follows: Suppose the initial curve $\Gamma^0$ is given by $\boldsymbol{X}^0(\rho_j)=\boldsymbol{X}(\rho_j,0),\,\,j=1,2,\cdots,N$. For any $m\geq 0$, find the solution $\left(\boldsymbol{X}^{m+1}(\cdot),\mu^{m+1}(\cdot)\right)\in[\mathbb{K}_p^h]^2\times \mathbb{K}_p^h$ such that \begin{subequations}\label{eqn:SP-PFEM acmcf}
  \begin{align}
    &\Bigl(\boldsymbol{n}^{m+\frac{1}{2}}\cdot\frac{\boldsymbol{X}^{m+1}-\boldsymbol{X}^m}{\tau},\varphi^h\Bigr)_{\Gamma^m}^h+\Bigl(\mu^{m+1}-\lambda^{m,*},\varphi^h\Bigr)_{\Gamma^m}^h=0,\qquad\forall\varphi^h\in\mathbb{K}_p^h,\label{eqn:SP-PFEM acmcf a}\\
    &\Bigl(\mu^{m+1}\boldsymbol{n}^{m+\frac{1}{2}},\boldsymbol{\omega}^h\Bigr)_{\Gamma^m}^h-\Bigl(\hat{\boldsymbol{G}}_k^\alpha(\theta^m)\partial_{s^m}\boldsymbol{X}^{m+1},\partial_{s^m}\boldsymbol{\omega}^h\Bigr)_{\Gamma^m}^h=0,\qquad\forall\boldsymbol{\omega}^h\in[\mathbb{K}_p^h]^2,\label{eqn:SP-PFEM acmcf b}
  \end{align}
\end{subequations} where $\lambda^{m,*}\coloneqq\left(\mu^{m+1},1\right)^h_{\Gamma^m}/|\Gamma^m|$.

For the SP-PFEM \eqref{eqn:SP-PFEM acmcf}, we have the following structure-preserving property: \begin{theorem}\label{thm:SP-PFEM acmcf}
  Suppose $\hat{\gamma}(\theta)$ satisfies \eqref{eqn:energy cond} or \eqref{eqn:energy cond rlx} when $a=-1$. Then the SP-PFEM \eqref{eqn:SP-PFEM acmcf} is structure-preserving with sufficiently large $k(\theta)$, i.e. \begin{equation}
    A^{m+1}=A^m=\cdots=A^0,\qquad W^{m+1}\leq W^m\leq\cdots \leq W^0,\qquad\forall m\geq 0.
  \end{equation} 
\end{theorem}

The proof is similar to that in \cite[Theorem 4.2]{bao2024structure} and is therefore omitted.

\subsection{Anisotropic surface diffusion}

Similarly, for anisotropic surface diffusion in \eqref{eqn:geo flows}, the conservative strong form is given as \begin{subequations}\label{eqn:strong form sd}
  \begin{align}
    &\boldsymbol{n}\cdot\partial_t\boldsymbol{X}-\partial_{ss}\mu=0,\\
    &\mu\boldsymbol{n}+\partial_s\Bigl(\hat{\boldsymbol{G}}_k^\alpha(\theta)\partial_s\boldsymbol{X}\Bigr)=\boldsymbol{0}. 
  \end{align}
\end{subequations} And the corresponding variational formulation follows by a similar derivation.

Suppose the initial closed curve $\Gamma^0$ is given by $\boldsymbol{X}^0(\rho_j)=\boldsymbol{X}(\rho_j,0),\,\,j=1,2,\cdots,N$. Then the SP-PFEM for anisotropic surface diffusion in \eqref{eqn:geo flows} can be stated as follows: For $m\geq 0$, find the solution $\left(\boldsymbol{X}^{m+1}(\cdot),\mu^{m+1}(\cdot)\right)\in[\mathbb{K}_p^h]^2\times \mathbb{K}_p^h$ satisfying \begin{subequations}\label{eqn:SP-PFEM sd}
  \begin{align}
    &\Bigl(\boldsymbol{n}^{m+\frac{1}{2}}\cdot\frac{\boldsymbol{X}^{m+1}-\boldsymbol{X}^m}{\tau},\varphi^h\Bigr)_{\Gamma^m}^h+\Bigl(\partial_{s^m}\mu^{m+1},\partial_{s^m}\varphi^h\Bigr)_{\Gamma^m}^h=0,\qquad\forall\varphi^h\in\mathbb{K}_p^h,\label{eqn:SP-PFEM sd a}\\
    &\Bigl(\mu^{m+1}\boldsymbol{n}^{m+\frac{1}{2}},\boldsymbol{\omega}^h\Bigr)_{\Gamma^m}^h-\Bigl(\hat{\boldsymbol{G}}_k^\alpha(\theta^m)\partial_{s^m}\boldsymbol{X}^{m+1},\partial_{s^m}\boldsymbol{\omega}^h\Bigr)_{\Gamma^m}^h=0,\qquad\forall\boldsymbol{\omega}^h\in[\mathbb{K}_p^h]^2.\label{eqn:SP-PFEM sd b}
  \end{align}
\end{subequations}

For the SP-PFEM \eqref{eqn:SP-PFEM sd}, the following structure-preserving property holds: \begin{theorem}\label{thm:SP-PFEM sd}
  Suppose $\hat{\gamma}(\theta)$ satisfies \eqref{eqn:energy cond} or \eqref{eqn:energy cond rlx} when $a=-1$. Then the SP-PFEM \eqref{eqn:SP-PFEM sd} is structure-preserving with sufficiently large $k(\theta)$, i.e. \begin{equation}
    A^{m+1}=A^m=\cdots= A^0,\qquad W^{m+1}\leq W^m\leq\cdots \leq W^0,\qquad\forall m\geq 0.
  \end{equation}
\end{theorem}

For the proof we refer the reader to \cite{li2025structure}. Details are omitted here for brevity.

\subsection{General normal velocity laws}

% The geometric identity \eqref{eqn:weighted curvature identity} provides an effective approach for discretizing the weighted curvature $\mu$. As a result, the parametric finite element approximation \eqref{eqn:SP-PFEM}, like the original BGN method, can be extended to more general anisotropic curvature-driven problems.

% Different geometric flow problems involve different normal velocities. Designing energy-stable numerical schemes for such problems often requires a delicate discretization of the velocity. This is because, in previous literature, the normal velocity is typically treated as a function depending only on geometric quantities, such as $\mathfrak{F}(\mu)$. We observe that by treating the normal velocity as a mapping depending on both geometric quantities and the underlying curve, a unified discretization framework can be naturally derived. Moreover, the resulting schemes under this framework consistently ensure energy dissipation.

The introduction of surface energy matrices enables us to address anisotropic problems analogously to the way isotropic cases are treated in the BGN-type method. %Departing from the conventional view of the velocity as a function solely of geometric quantities, we regard it as a mapping that depends jointly on geometric quantities and the underlying curve.

Here we consider flows of the form \begin{equation}\label{eqn:general normal velocity law}
  V_n=\mathfrak{F}(\mu),\qquad\text{on}\,\,\Gamma(t),
\end{equation} where $\mathfrak{F}(\cdot)$ is a mapping that maps functions on closed evolving curve $\Gamma(t)$ to functions on $\Gamma(t)$. 

If $\mathfrak{F}$ satisfies \begin{equation}\label{eqn:F property}
  \Bigl(\mathfrak{F}(\mu),\mu\Bigr)_{\Gamma(t)}\leq 0,
\end{equation} then the evolution equation \eqref{eqn:general normal velocity law} exhibits the property of area decay rate and energy dissipation as \begin{equation}
  \frac{\mathrm{d}}{\mathrm{dt}}A(t)=\Bigl(\mathfrak{F}(\mu),1\Bigr)_{\Gamma(t)},\qquad\frac{\mathrm{d}}{\mathrm{d}t}W(t)=\Bigl(V_n,\mu\Bigr)_{\Gamma(t)}=\Bigl(\mathfrak{F}(\mu),\mu\Bigr)_{\Gamma(t)}\leq 0.
\end{equation} 
All the evolution laws discussed above can be incorporated into this equation. For instance, for the anisotropic surface diffusion $V_n=\partial_{ss}\mu$, by taking integration by parts, we have \begin{equation}
  \Bigl(\mathfrak{F}(\mu),\mu\Bigr)_{\Gamma(t)}=\Bigl(\partial_{ss}\mu,\mu\Bigr)_{\Gamma(t)}=-\Bigl(\partial_s\mu,\partial_s\mu\Bigr)_{\Gamma(t)}\leq 0.
\end{equation}

As we have seen earlier, to obtain a structure-preserving numerical scheme for the evolution equation \eqref{eqn:general normal velocity law}, we should formally consider the following strong form: \begin{subequations}\label{eqn:strong form general}
  \begin{align}
    &\boldsymbol{n}\cdot\partial_t\boldsymbol{X}-\mathfrak{F}(\mu)=0,\\
    &\mu\boldsymbol{n}+\partial_s\Bigl(\hat{\boldsymbol{G}}_k^\alpha(\theta)\partial_s\boldsymbol{X}\Bigr)=\boldsymbol{0}.
  \end{align}
\end{subequations} 

While the notation $\mathfrak{F}(\mu)$ is conventional, constructing structure-preserving fully discrete schemes requires careful treatment of the underlying curve $\Gamma(t)$. To make this dependence explicit, we adopt the notation $\mathfrak{F}(\mu,\Gamma) = \mathfrak{F}(\mu)$. This perspective naturally suggests a semi-implicit discretization: treating $\mu$ implicitly at time level $m+1$ while treating $\Gamma$ explicitly at level $m$. As shown below, this strategy yields structure-preserving fully discrete schemes in a unified manner:

Suppose the initial data $\Gamma^0$ is given. For $m\geq 0$, find the solution $\left(\boldsymbol{X}^{m+1}(\cdot),\mu^{m+1}(\cdot)\right)\in[\mathbb{K}_p^h]^2\times \mathbb{K}_p^h$ such that \begin{subequations}\label{eqn:SP-PFEM general}
  \begin{align}
    &\Bigl(\boldsymbol{n}^{m+\frac{1}{2}}\cdot\frac{\boldsymbol{X}^{m+1}-\boldsymbol{X}^m}{\tau},\varphi^h\Bigr)_{\Gamma^m}^h-\Bigl(\mathfrak{F}(\mu^{m+1},\Gamma^m),\varphi^h\Bigr)_{\Gamma^m}^h=0,\qquad\forall\varphi^h\in\mathbb{K}_p^h,\\
    &\Bigl(\mu^{m+1}\boldsymbol{n}^{m+\frac{1}{2}},\boldsymbol{\omega}^h\Bigr)_{\Gamma^m}^h-\Bigl(\hat{\boldsymbol{G}}_k^\alpha(\theta^m)\partial_{s^m}\boldsymbol{X}^{m+1},\partial_{s^m}\boldsymbol{\omega}^h\Bigr)_{\Gamma^m}^h=0,\qquad\forall \boldsymbol{\omega}^h\in[\mathbb{K}_p^h]^2.
  \end{align}
\end{subequations}

For \eqref{eqn:SP-PFEM general}, the following structure-preserving property holds: \begin{theorem}\label{thm:SP-PFEM general}
  Suppose $\hat{\gamma}(\theta)$ satifies \eqref{eqn:energy cond} or \eqref{eqn:energy cond rlx} when $a=-1$. Then the SP-PFEM \eqref{eqn:SP-PFEM general} is structure-preserving with sufficiently large $k(\theta)$, i.e. \begin{equation}
    \frac{A^{m+1}-A^m}{\tau}=-\Bigl(\mathfrak{F}(\mu^{m+1},\Gamma^m),1\Bigr)^h_{\Gamma^m},\qquad W^{m+1}\leq W^m\leq\cdots \leq W^0,\qquad\forall m\geq 0.
  \end{equation}
\end{theorem}

\begin{proof}
  For the area decay rate. Similar to \eqref{eqn:area decay rate formula}, we have \begin{equation}
    A^{m+1}-A^m=\Bigl(\boldsymbol{n}^{m+\frac{1}{2}}\cdot(\boldsymbol{X}^{m+1}-\boldsymbol{X}^m),1\Bigr)_{\Gamma^m}^h=-\tau\Bigl(\mathfrak{F}(\mu^{m+1},\Gamma^m),1\Bigr)_{\Gamma^m}^h.
  \end{equation}

  For the energy dissipation, \eqref{eqn:energy difference} still holds. Thus, we have \begin{equation}
    \begin{aligned}
      W^{m+1}-W^m&\leq \Bigl(\hat{\boldsymbol{G}}_k^\alpha(\theta^m)\partial_{s^m}\boldsymbol{X}^{m+1},\partial_{s^m}(\boldsymbol{X}^{m+1}-\boldsymbol{X}^m)\Bigr)^h_{\Gamma^m}\\
      &=\Bigl(\mu^{m+1}\boldsymbol{n}^{m+\frac{1}{2}},\boldsymbol{X}^{m+1}-\boldsymbol{X}^m\Bigr)^h_{\Gamma^m}=-\tau\Bigl(\mathfrak{F}(\mu^{m+1},\Gamma^m),\mu^{m+1}\Bigr)^h_{\Gamma^m}\leq 0.
    \end{aligned}
  \end{equation} The last inequality follows from the assumption \eqref{eqn:F property}.
\end{proof}

Here, we present three examples to illustrate our ideas. For additional examples and numerical discussions, we refer the reader to \cite[section 5]{barrett2020parametricbook} by Barrett et al. as well as the references therein.

\begin{example}
  By choosing $\mathfrak{F}(f,\Gamma)=-f+\int_{\Gamma}f\,\mathrm{d}s/|\Gamma|$, where $f$ is a function on $\Gamma$, we obtain $\mathfrak{F}(\mu^{m+1},\Gamma^{m})=-\mu^{m+1}+\int_{\Gamma^m}\mu^{m+1}\,\mathrm{d}s/|\Gamma^m|$. This leads to the SP-PFEM for the area-conserved anisotropic curvature flow, as shown in \eqref{eqn:SP-PFEM acmcf}.
\end{example}

\begin{example}
  For the surface diffusion, we take $\mathfrak{F}(f,\Gamma)=-\partial_{ss}f$, where $\partial_s$ denotes differentiation with respect to arc-length along $\Gamma$. Applying integration by parts yields $\left(\mathfrak{F}(\mu,\Gamma),\mu\right)_\Gamma=-(\partial_s\mu,\partial_s\mu)_\Gamma\leq 0$, which leads to recovery of the SP-PFEM \eqref{eqn:SP-PFEM sd} and Theorem~\ref{thm:SP-PFEM sd} for anisotropic surface diffusion.
\end{example}

\begin{example}
  An intermediate flow between the area-conserved anisotropic curvature flow and surface diffusion is \begin{equation}\label{eqn:intermediate flow}
    V_n=\partial_{ss}\eta,\qquad\text{where}\,\,\left(-\frac{1}{\xi}\partial_{ss}+\frac{1}{\nu}\right)\eta=\mu,
  \end{equation} and $\xi,\nu\in\mathbb{R}^+$. We define the mapping $\mathfrak{F}(\cdot,\Gamma^m)$ via \begin{equation}
    \Bigl(\mathfrak{F}(f,\Gamma^m),\varphi^h\Bigr)^h_{\Gamma^m}=\Bigl(\partial_{s^m}\eta(f),\partial_{s^m}\varphi^h\Bigr)^h_{\Gamma^m},\qquad\forall\varphi^h\in \mathbb{K}_p^h,
  \end{equation} where $\eta(f)$ represents the solution of \begin{equation}
  \frac{1}{\xi}\Bigl(\partial_{s^m}\eta(f),\partial_{s^m}\psi^h\Bigr)^h_{\Gamma^m}+\frac{1}{\nu}\Bigl(\eta(f),\psi^h\Bigr)^h_{\Gamma^m}=\Bigl(f,\psi^h\Bigr)^h_{\Gamma^m},\qquad\forall\psi^h\in \mathbb{K}_p^h.
  \end{equation} Taking $f=\varphi^h=\mu^{m+1},\psi^h=\mu^{m+1}-\frac{1}{\nu}\eta(\mu^{m+1})$, we obtain \begin{equation}
    \Bigl(\mathfrak{F}(\mu^{m+1},\Gamma^m),\mu^{m+1}\Bigr)^h_{\Gamma^m}=-\frac{1}{\nu}\Bigl(\eta,\eta\Bigr)^h_{\Gamma^m}-\xi\Bigl(\mu^{m+1}-\frac{1}{\nu}\eta,\mu^{m+1}-\frac{1}{\nu}\eta\Bigr)^h_{\Gamma^m}\leq 0.
  \end{equation} As a result, we will derive a structure-preserving discretization of the intermediate flow \eqref{eqn:intermediate flow} under the condition \eqref{eqn:energy cond} or \eqref{eqn:energy cond rlx}. For further discussion of the intermediate flow, we refer the reader to \cite{taylor1994linking,elliott1997diffusional,escher2001limiting} for theoretical aspects and to \cite{bao2022volume,barrett2007variational} for numerical methods.
\end{example}

\section{Numerical results}\label{sec:numer}

In this section, we will report extensive numerical experiments to demonstrate the high performance of the proposed SP-PFEMs. 

\subsection{Error and convergence rate}
To measure the distance between two closed curves $\Gamma_1$ and $\Gamma_2$, we introduce the manifold distance $M(\Gamma_1,\Gamma_2)$ as follows \cite{zhao2021energy,li2021energy}: \begin{equation}
  M(\Gamma_1,\Gamma_2)\coloneqq|(\Omega_1\backslash\Omega_2)\cup(\Omega_2\backslash\Omega_1)|=2|\Omega_1\cup\Omega_2|-|\Omega_1|-|\Omega_2|,
\end{equation} where $\Omega_i\,\,(i=1,2)$ is the interior region enclosed by $\Gamma_i$ and $|\Omega|$ denotes the area of $\Omega$. 

We define the intermediate curve $\Gamma_{h,\tau}(t)$ between $\Gamma^m,\Gamma^{m+1}$ as \begin{equation}
  \Gamma_{h,\tau}(t)\coloneqq\frac{t_{m+1}-t}{\tau}\,\Gamma^m+\frac{t-t_m}{\tau}\,\Gamma^{m+1},\qquad t\in[t_m,t_{m+1}],\,\,m\geq 0.
\end{equation} The numerical error $e^h(t)$ is formally given as \begin{equation}
  e^h(t)\coloneqq M(\Gamma_{h,\tau}(t),\Gamma(t)).
\end{equation} We approximate the exact solution $\Gamma(t)$ by $\Gamma_{h_e,\tau_e}(t)$ with fine meshes $h_e=2^{-8},\tau_e=h_e^2$. 

In the convergence tests, we mainly consider the following two types of surface energies: \begin{itemize}
  \item Case I: $\hat{\gamma}(\theta)=1+\beta\cos 3\theta$ with $|\beta|<1$.
  \item Case II: $\hat{\gamma}(\theta)=\sqrt{\left(\frac{5}{2}+\frac{3}{2}\text{sgn}(n_1)\right)n_1^2+n_2^2}$, where $\boldsymbol{n}=(n_1,n_2)^T=(-\sin\theta,\cos\theta)^T$.
\end{itemize} 

The minimal stabilizing function $k_\text{min}^a(\theta)$ is obtained as follows: for a given $\hat{\gamma}(\theta)$ and $a\in\mathbb{R}$, we solve the optimization problem \eqref{eqn:def of min stab func b} at $\theta_j=-\pi+\frac{j\pi}{10},\,\,\forall 0\leq j\leq 20$ to compute $k_0^a(\theta_j)$, then use linear interpolation to approximate $k_0^a(\theta)$ at the intermediate points and obtain $k_\text{min}^a(\theta)$ by \eqref{eqn:def of min stab func}. In the following convergence tests, the stabilizing functions are always chosen to be $k(\theta)=k_\text{min}^a(\theta)$, unless otherwise stated.

\begin{table}[htbp]
  \captionsetup{skip=12pt}
  \renewcommand{\arraystretch}{1.5}
  \centering
  \begin{tabular}{|c|c|cc|cc|cc|}
  \hline
   $\alpha$ & $(h,\tau)$ & $e^h(t=0.1)$ & order & $e^h(t=0.2)$ & order & $e^h(t=0.3)$ & order \\
   \hline
   \multirow{3}{*}{-1} & $(h_0,\tau_0)$ & 1.30e-1 & - & 1.63e-1 & - & 1.84e-1 & - \\
                       & $(\frac{h_0}{2},\frac{\tau_0}{4})$ & 3.12e-2 & 2.06 & 4.03e-2 & 2.01 & 4.64e-2 & 1.98 \\
                       & $(\frac{h_0}{2^2},\frac{\tau_0}{4^2})$ & 8.02e-3 & 1.96 & 1.05e-2 & 1.94 & 1.21e-2 & 1.94 \\
  \hline
  \multirow{3}{*}{0} & $(h_0,\tau_0)$ & 1.30e-1 & - & 1.56e-2 & - & 1.73e-1 & - \\
                     & $(\frac{h_0}{2},\frac{\tau_0}{4})$ & 3.14e-2 & 2.05 & 4.04e-2  & 1.95 & 4.65e-2 & 1.90 \\
                     & $(\frac{h_0}{2^2},\frac{\tau_0}{4^2})$ & 8.00e-3 & 1.97 & 1.04e-2 & 1.95 & 1.20e-2 &  1.94\\
\hline
\multirow{3}{*}{1} & $(h_0,\tau_0)$ & 1.31e-1 & - & 1.58e-1 & - & 1.76e-1 & - \\
                    & $(\frac{h_0}{2},\frac{\tau_0}{4})$ & 3.15e-2 & 2.05 & 4.05e-2 & 1.97 & 4.66e-2 & 1.92 \\
                    & $(\frac{h_0}{2^2},\frac{\tau_0}{4^2})$ & 7.98e-3 & 1.98 & 1.04e-2 & 1.97 & 1.20e-2 & 1.96 \\
  \hline
  \end{tabular}
  \caption{Error $e^h(t)$ and the convergence rates of SP-PFEM \eqref{eqn:SP-PFEM} with $\alpha=0,\pm 1$. The anisotropy is chosen from the energy density in Case I, with $\beta$ set to $1/9$. Other parameters are chosen as $h_0=2^{-4},\tau_0=h_0^2$.}
  \label{tab:conver rate}
\end{table}

\begin{figure}[htbp]
  \centering
  \includegraphics[width=0.333\textwidth]{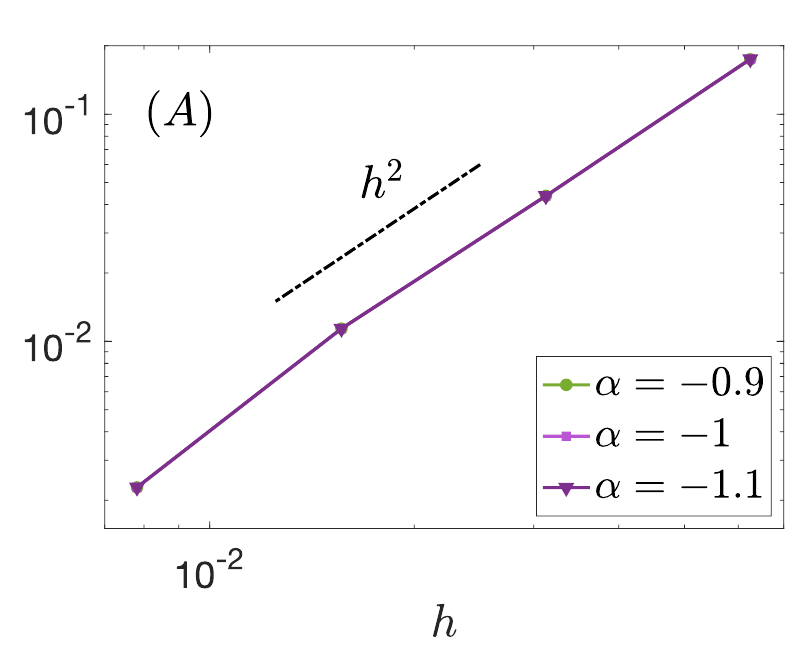}\includegraphics[width=0.333\textwidth]{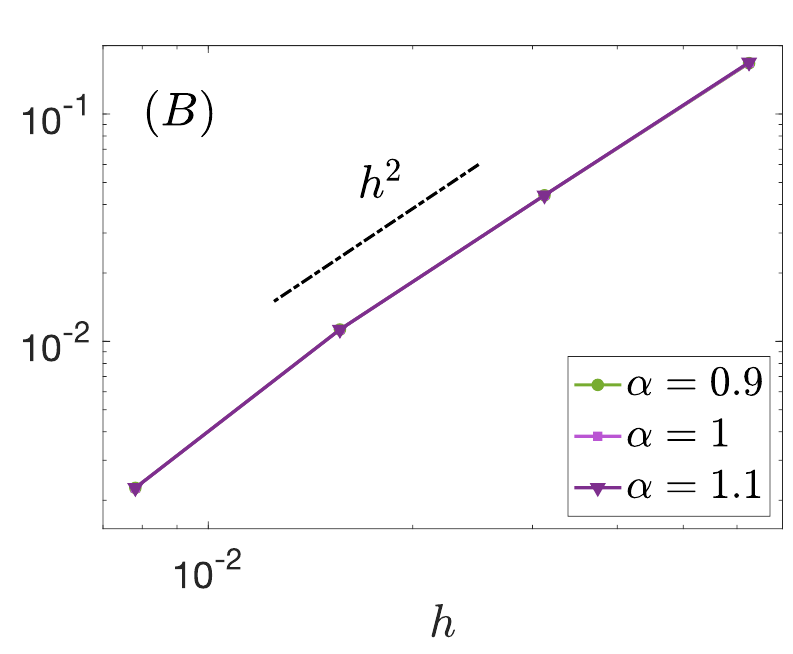}\includegraphics[width=0.333\textwidth]{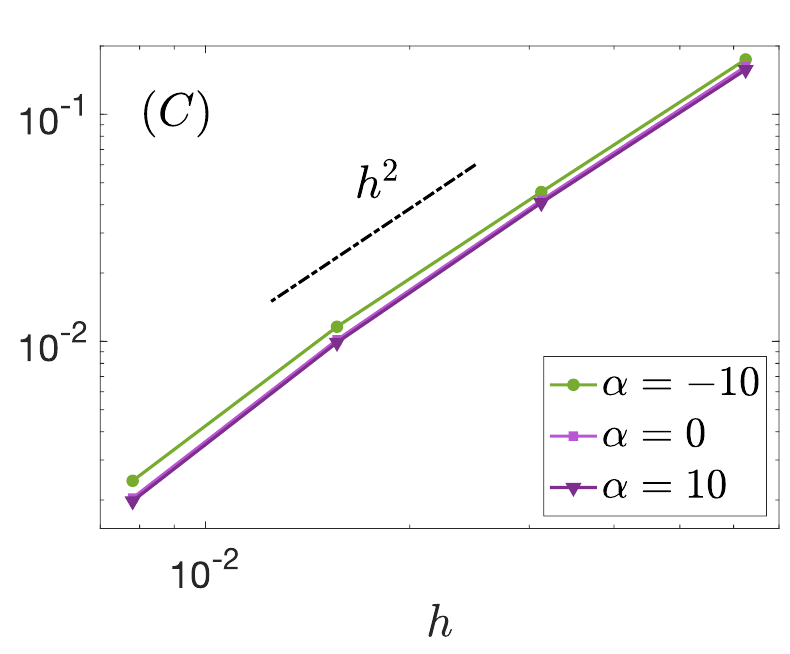}
  \caption{Convergence rates of SP-PFEM \eqref{eqn:SP-PFEM} with different $\alpha$ at $t=0.25$: (A)--(B) Case I with $\beta=1/9$; and (C) Case II.}
  \label{fig: conver rate 2}
\end{figure}

Given that the numerical scheme formally exhibits second-order spatial accuracy and first-order temporal accuracy, the time step $\tau$ is consistently chosen as $\tau=h^2$, unless otherwise specified. In this subsection, the initial curve is chosen as an ellipse with a major axis of 4 and a minor axis of 1. The tolerance value for the Newton's iteration is set to be $\text{tol}=10^{-11}$. 

Numerical errors are reported in Table~\ref{tab:conver rate} and Figure~\ref{fig: conver rate 2}. It can be observed that the SP-PFEM \eqref{eqn:SP-PFEM} exhibits second-order spatial accuracy and first-order temporal accuracy. 

\begin{figure}[htbp]
  \centering
  \includegraphics[width=1\textwidth]{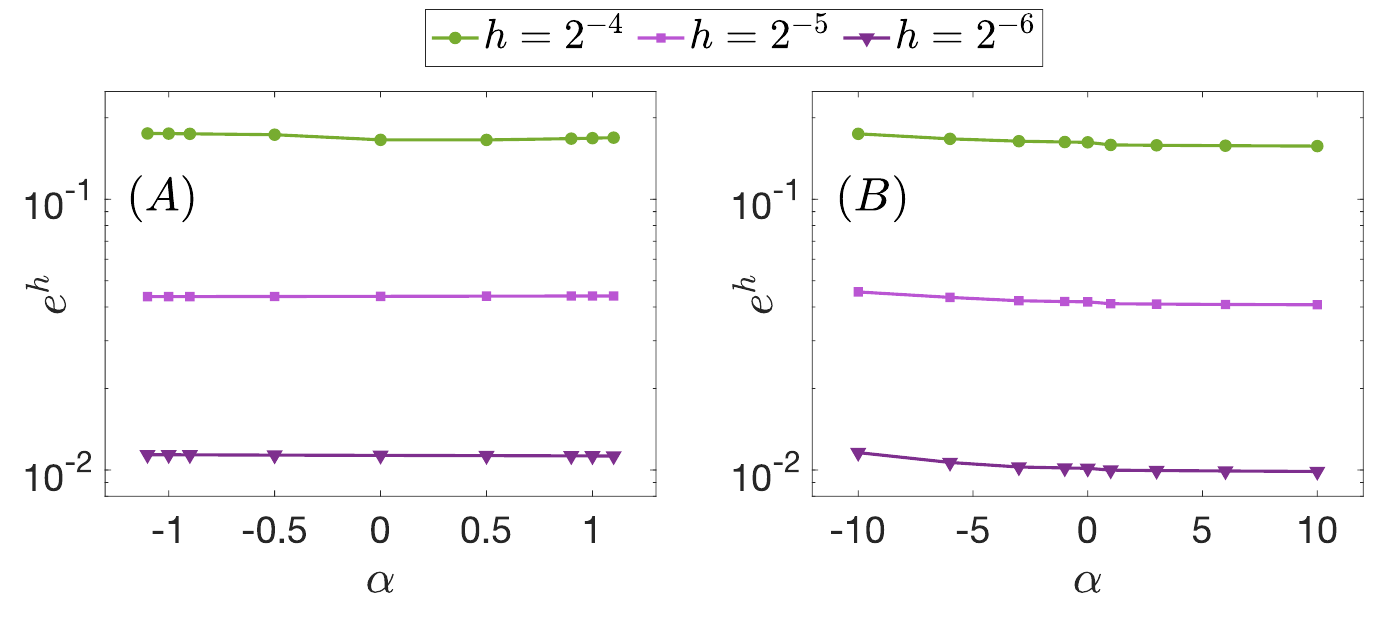}
  \caption{Error comparison of SP-PFEM \eqref{eqn:SP-PFEM} at $t=0.25$: (A) Case I with $\alpha=0,\pm 0.5, \pm 0.9, \pm 1, \pm 1.1$; and (B) Case II with $\alpha=0, \pm 1, \pm 3, \pm 6, \pm 10$.}
  \label{fig:errorcompare}
\end{figure}

We compared the errors corresponding to multiple different values of $\alpha$ at the same time $t=0.25$ for Case I and Case II, as shown in Figure \ref{fig:errorcompare}. Specifically, we tested the baseline values $\alpha=0, \pm 1$, values near the baseline $\alpha=\pm 0.9,\pm 1.1$,  fractional values $\alpha=\pm 0.5$, as well as larger values $\alpha=\pm 3,\pm 6,\pm 10$. The results indicate that the accuracy of the SP-PFEM \eqref{eqn:SP-PFEM} is robust with respect to the choice of $\alpha$. In particular, it was observed that the relative differences among these errors remained below $5\%$ in the majority of cases. 

\subsection{Interation count and CPU time}

\begin{table}
  \captionsetup{skip=12pt}
  \renewcommand{\arraystretch}{1.5}
  \centering
  \begin{tabular}{|c|c|c|c|c|}
  \hline
  & \multicolumn{2}{c|}{$(h,\tau)=(2^{-7},4^{-7})$} & \multicolumn{2}{c|}{$(h,\tau)=(2^{-8},4^{-8})$} \\
  \hline
   $\alpha$  & Iteration count & CPU time (s) & Iteration count & CPU time (s) \\
   \hline
   -10 & 3278 & 17.9999 & 15856 & 162.2915 \\
   -5 &  3278 & 18.0083 & 13265 & 148.6691 \\
   -1  & 3278 & 18.5907 & 13294 & 151.7395 \\
   0  & 3278 & 18.6436 & 13304 & 152.5232 \\
   1 & 3278 & 18.3653 & 13313 & 151.2127 \\
   5 & 3278 & 18.6596 & 13496 & 151.3283 \\
   10 & 3278 & 17.1996 & 18070 & 166.7016\\
   \hline
  \end{tabular}
  \caption{The total number of iterations and average CPU time of SP-PFEM \eqref{eqn:SP-PFEM} for Case I with $\beta=1/7$ at $t=0.1$.}
  \label{tab:iteration count and CPU time}
\end{table}

Table~\ref{tab:iteration count and CPU time} presents the total number of iterations and the average CPU time for the SP-PFEM \eqref{eqn:SP-PFEM} with different values of $\alpha$ at the fixed time $t=0.1$. The reported CPU time is computed as the mean over five independent runs. All computations were performed on a personal laptop equipped with an Apple M1 Pro chip and 16GB of RAM, running macOS Sequoia 15.3.2. The data presented in Table~\ref{tab:iteration count and CPU time} show that the average number of iterations per update step is about $2$, indicating that our scheme can be solved by Newton's method with high efficiency. The average CPU time is approximately 18 seconds for mesh $(h,\tau)=(2^{-7},4^{-7})$ and around 150 seconds for the fine mesh $(h_e,\tau_e)=(2^{-8},4^{-8})$. The results indicate that the SP-PFEM \eqref{eqn:SP-PFEM} exhibits a consistent computational efficiency across different values of $\alpha$.

\subsection{Minimal stabilizing function and mesh quality}

\subsubsection{Minimal stabilizing function}
\begin{figure}[htbp]
  \centering
  \includegraphics[width=1\textwidth]{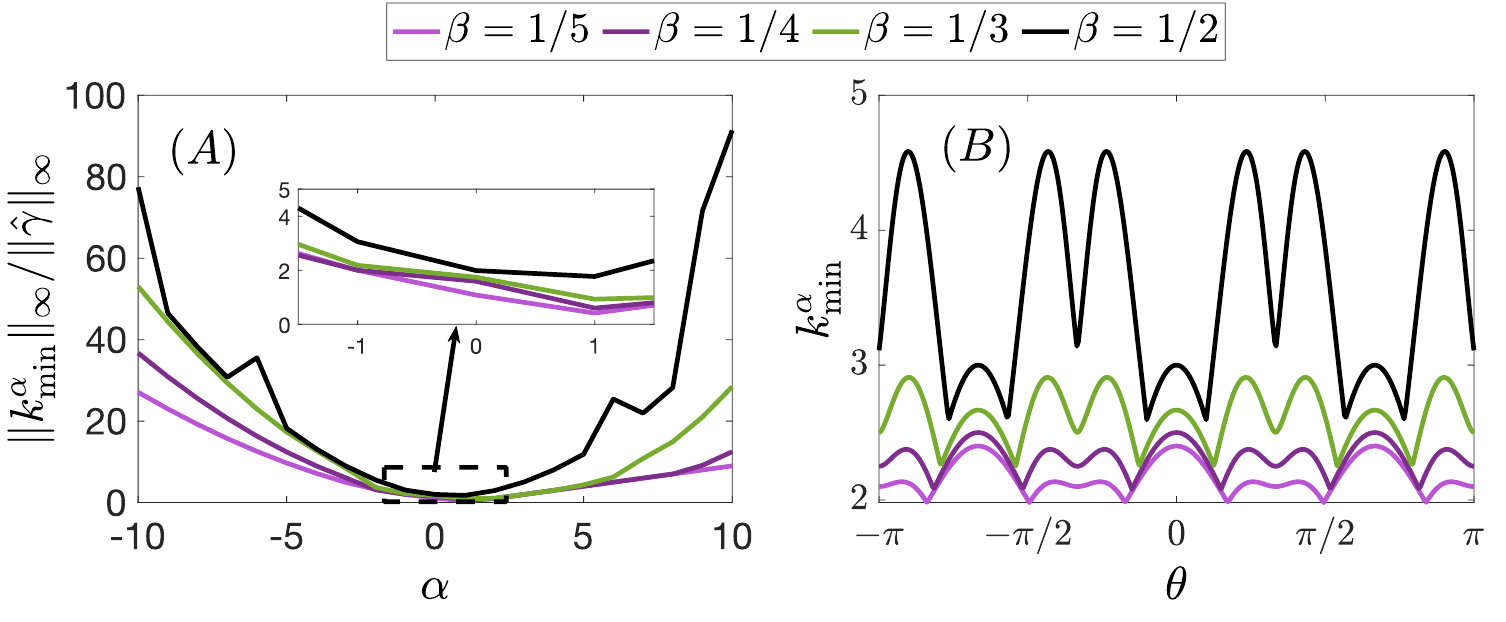}
  \caption{Minimal stabilizing functions and the ratios of their $L^\infty$-norms to the $L^\infty$-norm of the anisotropy in Case I.}
  \label{fig:stabfunccompare}
\end{figure}

Figure~\ref{fig:stabfunccompare} displays the minimal stabilizing functions for the 3-fold anisotropy $\hat{\gamma}(\theta)=1+\beta\cos 3\theta$ and the variation of the ratios of their $L^\infty$-norm to the $L^\infty$-norm of $\hat{\gamma}(\theta)$ with respect to parameter $\alpha$. In the right panel, the plot of $k_0^\alpha$ is provided for $\alpha=-1$. It can be observed that the value of $k^\alpha_{\text{min}}(\theta)$ increases as the anisotropy strength $\beta$ increases. Moreover, in the critical case $\beta=\frac{1}{2}$ of condition \eqref{eqn:energy cond} or \eqref{eqn:energy cond rlx}, $k_{\text{min}}^\alpha(\theta)$ remains finite, which is consistent with Theorem~\ref{thm:existence of k0 c=0}. From the left panel, it can be observed that the supremum of $k_{\text{min}}^\alpha$ increases rapidly as $|\alpha|$ becomes larger. When $-1\leq \alpha\leq 1$, the growth is moderate,  but once $|\alpha|>1$, the stabilizing term $k_{\text{min}}^\alpha(\theta)\boldsymbol{n}(\theta)\boldsymbol{n}(\theta)^T$ in the surface energy matrix $\hat{\boldsymbol{G}}_k^\alpha(\theta)$ becomes significantly larger in magnitude than the other components.

\begin{figure}[htbp]
  \centering
  \includegraphics[width=0.515\textwidth]{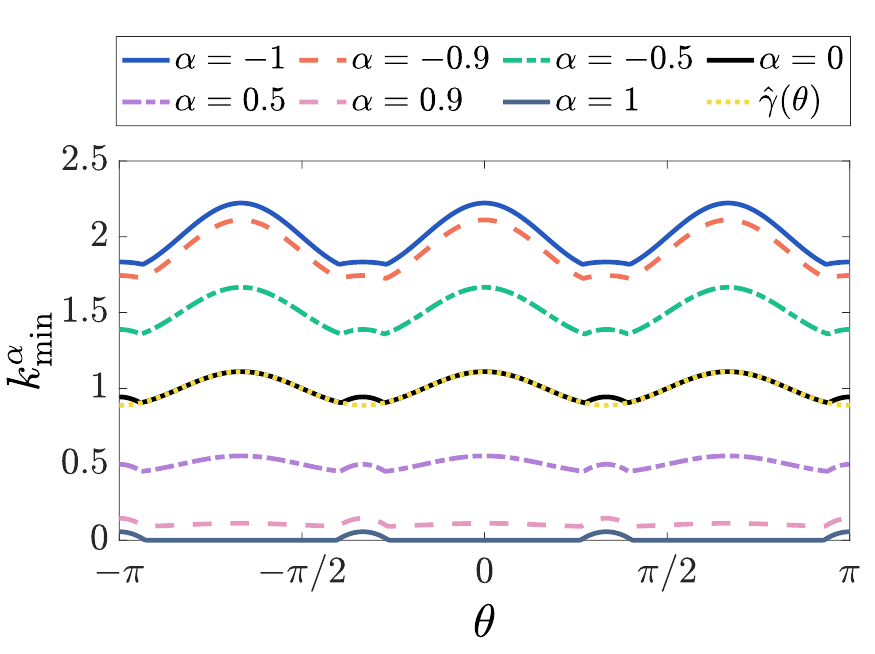}\includegraphics[width=0.485\textwidth]{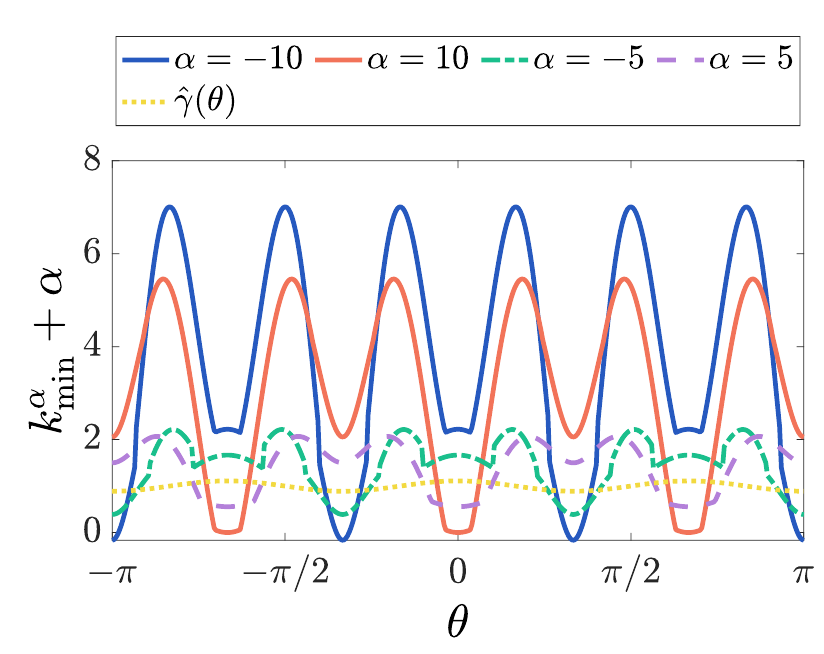}
  \caption{Minimal stabilizing functions for different values of $\alpha$ in Case I with $\beta=1/9$.}
  \label{fig:stab func}
\end{figure}

Plots of $k_{\text{min}}^\alpha(\theta)$ and $k_{\text{min}}^\alpha(\theta)+\alpha$ are given in Figure~\ref{fig:stab func} for different values of $\alpha$. The left panel shows that, for small $\alpha$, $k^\alpha_{\text{min}}(\theta)$ exhibits an approximately linear dependence on $\alpha$. For large values of $|\alpha|$, we use $k_{\text{min}}^\alpha+\alpha$ as an indicator of the magnitude of the regularization term in $\hat{\boldsymbol{G}}_k^\alpha(\theta)$. As shown in the right panel of Figure~\ref{fig:stab func}, the regularization term becomes significantly larger than $\hat{\gamma}(\theta)$, which is consistent with the conclusion drawn from Figure~\ref{fig:stabfunccompare}. 

\subsubsection{Weighted mesh ratio}

\begin{figure}[htbp]
  \centering
  \includegraphics[width=0.5\textwidth]{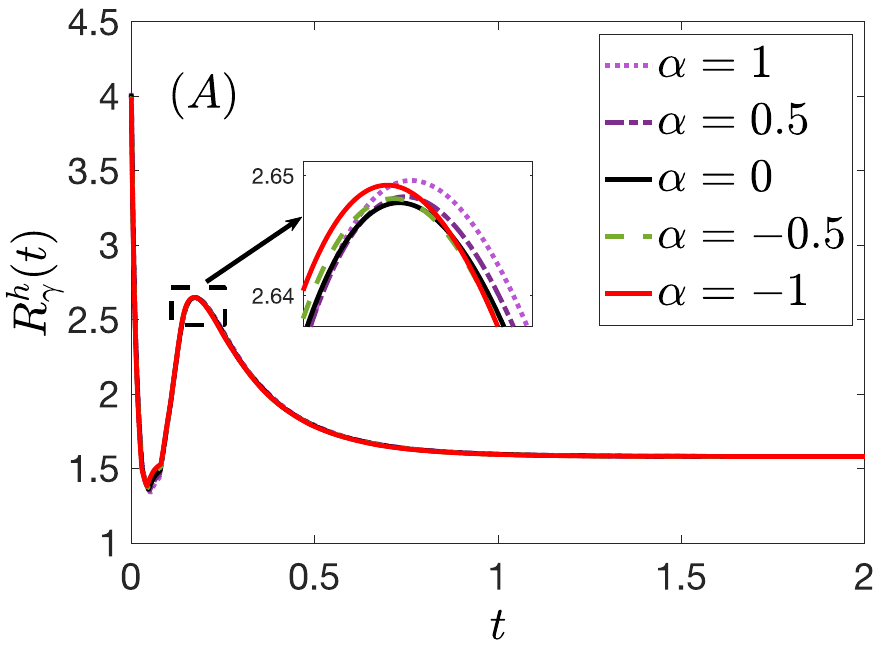}\includegraphics[width=0.5\textwidth]{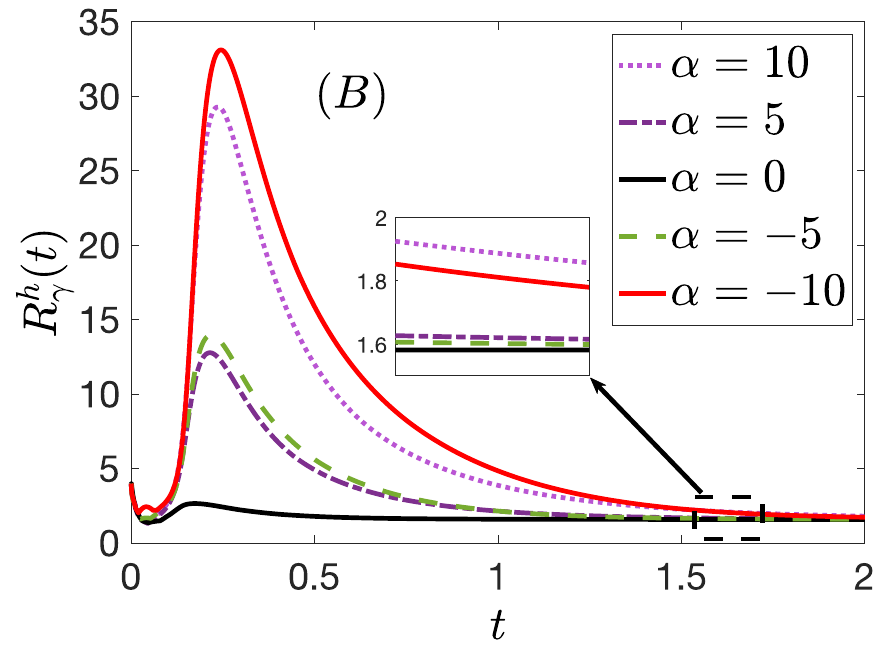}
  \caption{Weighted mesh ratio of the SP-PFEM \eqref{eqn:SP-PFEM sd} of an ellipse with major axis 4 and minor axis 1 under anisotropic surface diffusion with anisotropy $\hat{\gamma}(\theta)=1+\frac{1}{16}\cos 4\theta$ for different values of $\alpha$. The mesh sizes are chosen as $(h,\tau)=(2^{-6},4^{-6})$.}
  \label{fig:mesh}
\end{figure}

To test the mesh quality, we introduce the weighted mesh ratio $R^h_\gamma(t)$ as follows: \begin{equation}
  R^h_\gamma(t)\Bigg|_{t=t_m}\coloneqq\frac{\max\limits_{1\leq j\leq N}\hat{\gamma}(\theta_j)|\boldsymbol{h}_j^m|}{\min\limits_{1\leq j\leq N}\hat{\gamma}(\theta_j)|\boldsymbol{h}_j^m|},\qquad m\geq 0.
\end{equation} As the anisotropic curvature flow causes the curve to collapse into a singularity, we conduct experiments on anisotropic surface diffusion to examine the long-time mesh properties of the proposed method.

Figure~\ref{fig:mesh} shows the weighted mesh ratio $R^h_\gamma(t)$ of the SP-PFEM \eqref{eqn:SP-PFEM sd} for an ellipse with major axis 4 and minor axis 1 under anisotropic surface diffusion with anisotropy $\hat{\gamma}(\theta)=1+\frac{1}{16}\cos 4\theta$. The results indicate that the weighted mesh ratio $R^h_\gamma$ converges to a constant as $t\to\infty$, indicating that the SP-PFEM \eqref{eqn:SP-PFEM sd} achieves an asymptotically quasi-uniform mesh distribution. Moreover, it can be observed that for small values of $|\alpha|$,  the behavior of the weighted mesh ratio $R_\gamma^h$ remains nearly identical. However, when $|\alpha|$ is large, $R^h_\gamma$ can become significantly larger during the early stages of evolution. Therefore, to achieve better mesh quality in practical applications, smaller values of $|\alpha|$ are recommended.

\subsection{Area decay rate and energy dissipation}

\begin{figure}[htbp]
  \centering
  \includegraphics[width=1\textwidth]{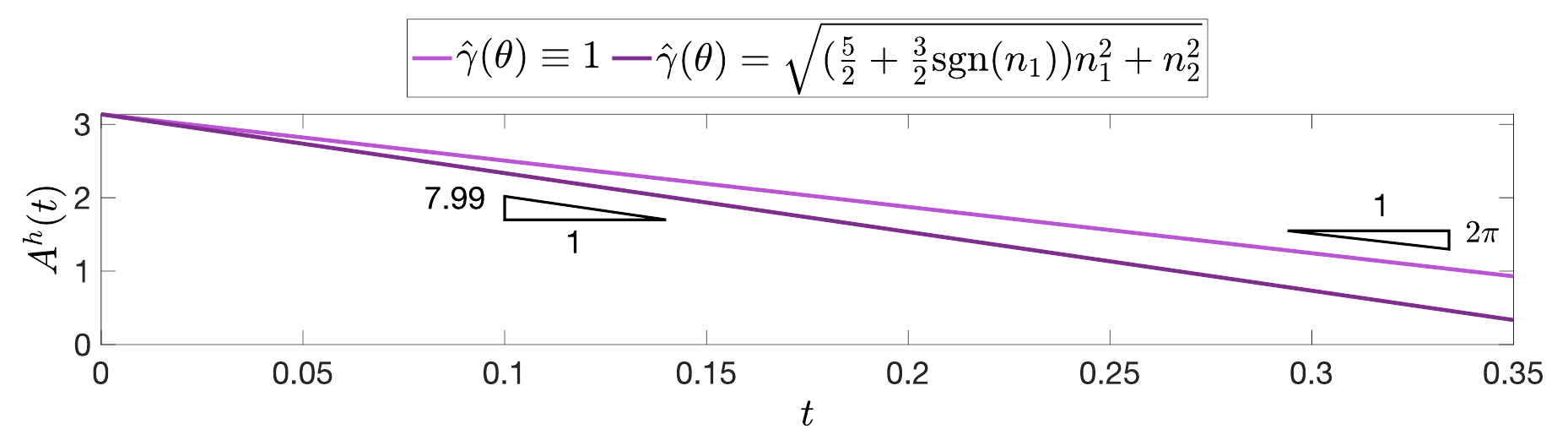}
  \caption{Area and area decay rate of an ellipse with major axis 4 and minor axis 1 under anisotropic curvature flow by SP-PFEM \eqref{eqn:SP-PFEM} with $\alpha=0$.}
  \label{fig:Area-rate}
\end{figure}

Figure~\ref{fig:Area-rate} illustrates the evolution of the area of an ellipse with major axis 4 and minor axis 1 under anisotropic curvature flows by SP-PFEM \eqref{eqn:SP-PFEM} with two surface energy densities: (i) the isotropic surface energy density $\hat{\gamma}(\theta)\equiv 1$; (ii) Case II. It can be observed that the area decreases approximately linearly.

For a simple closed curve evolving under the isotropic curvature flow, the area decay rate $\frac{\mathrm{d}A(t)}{\mathrm{d}t}=\int_{\Gamma}\kappa\,\mathrm{d}s=-2\pi$ remains constant. For anisotropic surface energies, the area decay rate tends to approach a constant value $-\int_{\Gamma}\mu\,\mathrm{d}s$. Here, $\Gamma$ represents the Wulff shape associated with the anisotropy $\hat{\gamma}(\theta)$ \cite{palmer1998stability}, scaled to have the same enclosed area as the initial curve. In particular, for the surface energy density of Case II, we have $-\int_{\Gamma}\mu\,\mathrm{d}s\approx -7.9858$ \cite{bao2024structure}.

\begin{figure}[htbp]
  \centering
  \includegraphics[width=0.5\textwidth]{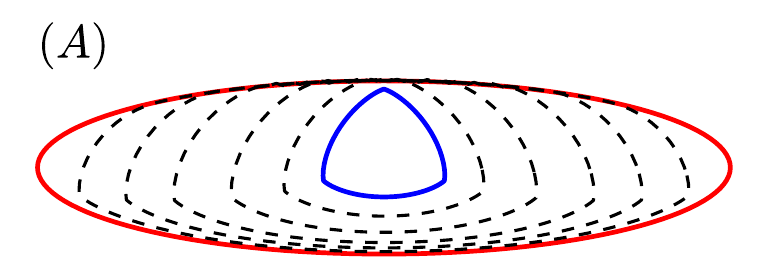}\includegraphics[width=0.5\textwidth]{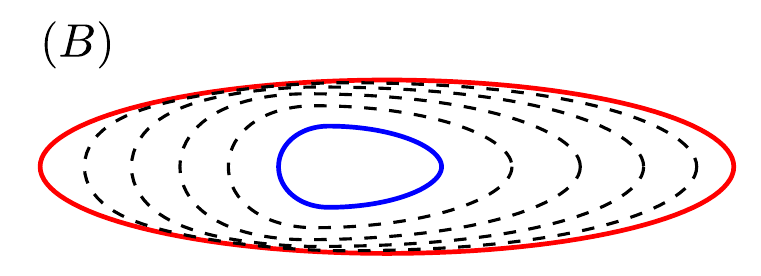}
  \caption{Morphological evolution of an ellipse with major axis 4 and minor axis 1 under anisotropic curvature flow by SP-PFEM \eqref{sec:SP-PFEM} for: (A) Case I with $\beta=1/7$; and (B) Case II. The red lines represent the initial shapes. The blue lines represent the numerical curves at (A) $t=0.45$ and (B) $t=0.35$. The black dashed lines represent the intermediate curves. Parameters are chosen as $\alpha=0,(h,\tau)=(2^{-7},4^{-7})$.}
  \label{fig:morph-evo}
\end{figure}

\begin{figure}[htbp]
  \centering
  \includegraphics[width=0.5\textwidth]{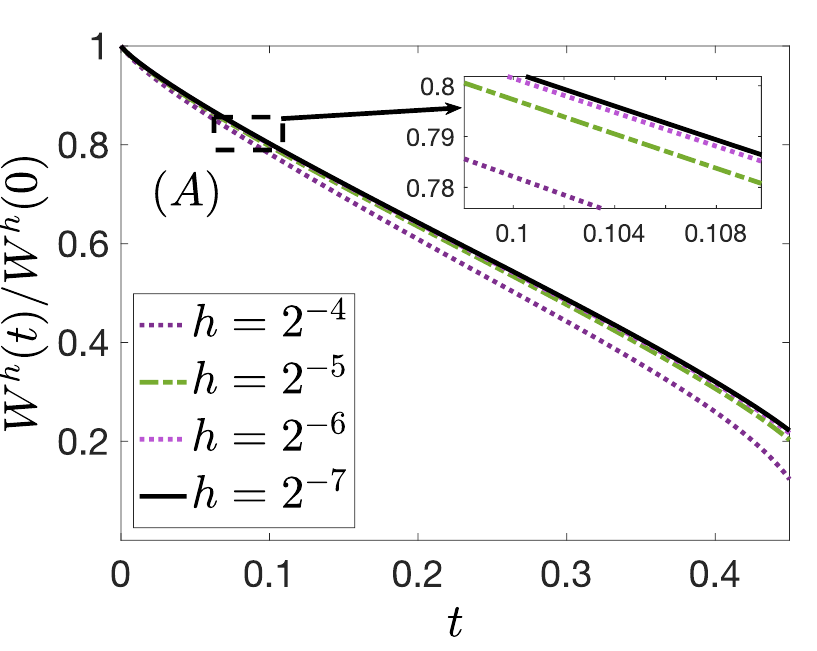}\includegraphics[width=0.5\textwidth]{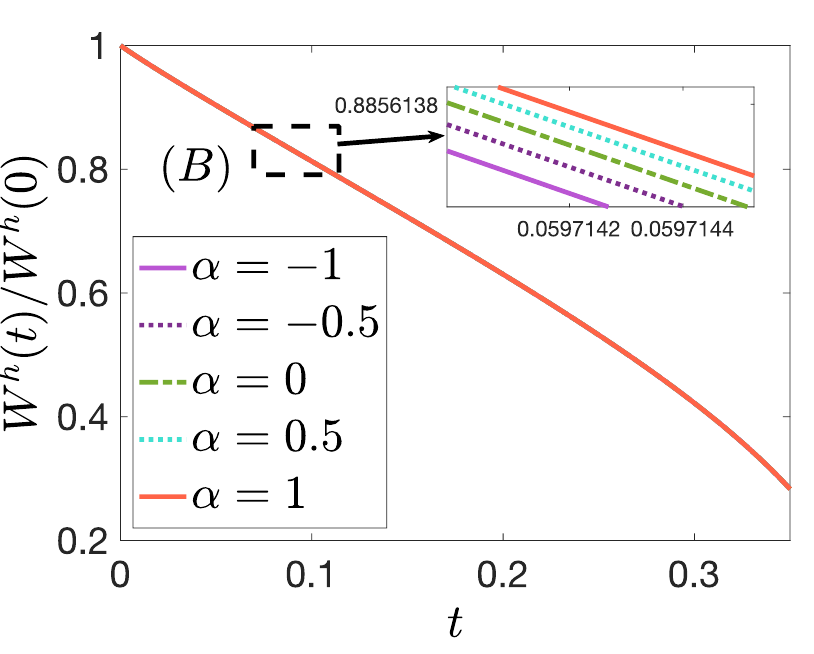}
  \caption{Normalized energy of SP-PFEM \eqref{eqn:SP-PFEM} for (A) Case I with $\beta=1/7$ and $\alpha=0$; and (B) Case II with different values of $\alpha$. The mesh sizes for figure (B) are chosen as $(h,\tau)=(2^{-7},4^{-7})$. }
  \label{fig:energy-evo}
\end{figure}

Figure~\ref{fig:morph-evo}--Figure~\ref{fig:energy-evo} plots the morphological evolution and normalized energy of an ellipse with major axis 4 and minor axis 1 under anisotropic curvature flow by SP-PFEM \eqref{eqn:SP-PFEM}. The anisotropy is chosen from the energy density in (A) Case I with $\beta=1/7$ and (B) Case II. The observation from Figure~\ref{fig:morph-evo}--Figure~\ref{fig:energy-evo} reveals that: \begin{itemize}
  \item The normalized energy is monotonically decreasing when $\hat{\gamma}(\theta)$ satifies the energy stability condition \eqref{eqn:energy cond}.
  \item The energy decay curves of SP-PFEM \eqref{eqn:SP-PFEM} for different values of $\alpha$ exhibit very similar shapes, indicating robustness of the method with respect to the parameter $\alpha$.
\end{itemize}

\subsection{Morphological evolution under anisotropic curvature flow}

In the following, we employ SP-PFEM \eqref{eqn:SP-PFEM} to simulate the morphological evolution of closed curves governed by anisotropic curvature flow. Unless otherwise specified, the mesh size is set to $(h,\tau)=(2^{-7},4^{-7})$.

\begin{figure}[htbp]
  \centering
  \includegraphics[width=0.458\textwidth]{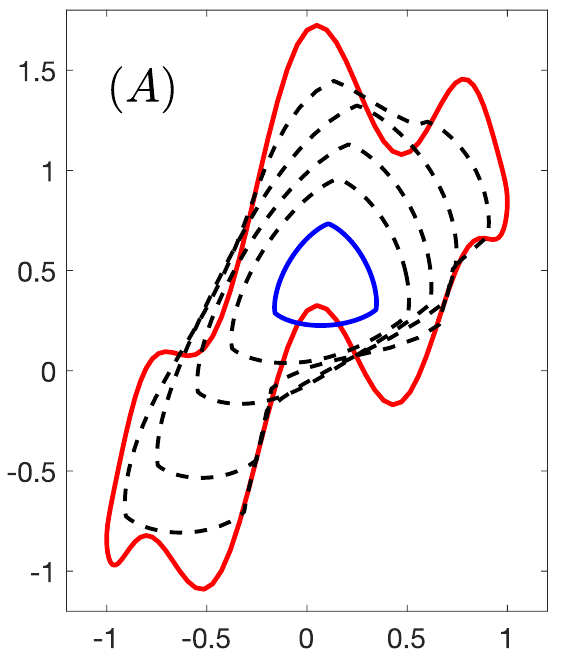}\includegraphics[width=0.45\textwidth]{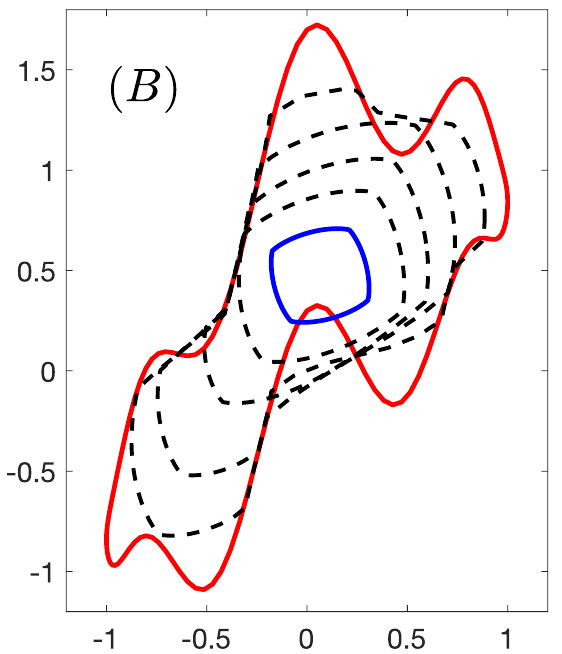}
  \caption{Morphological evolution of a non-convex initial curve with large curvature variations under anisotropic curvature flow by SP-PFEM \eqref{eqn:SP-PFEM} with $\alpha=2$. The red lines represent the initial shapes. The blue lines represent the numerical curves at $t=0.32$. The black dashed lines represent the intermediate curves.}
  \label{fig:morph-evo-1}
\end{figure}

\begin{figure}[htbp]
  \centering
  \includegraphics[width=0.5\textwidth]{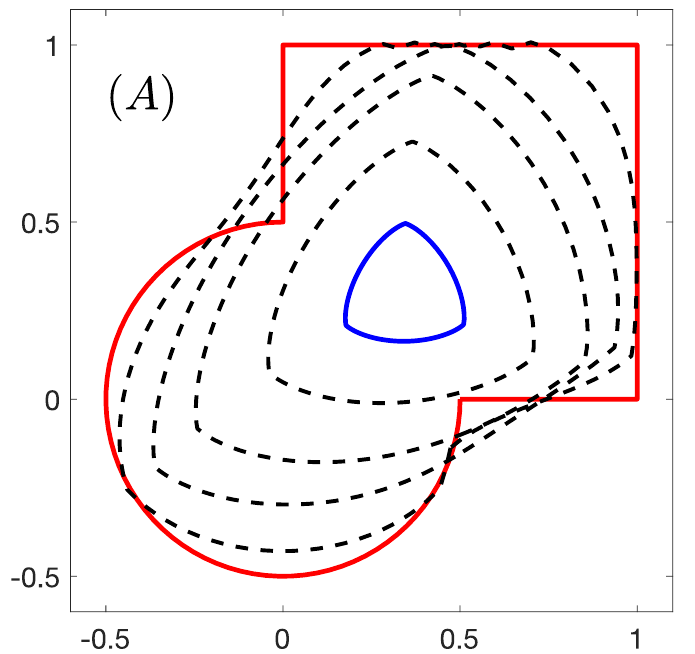}\includegraphics[width=0.5\textwidth]{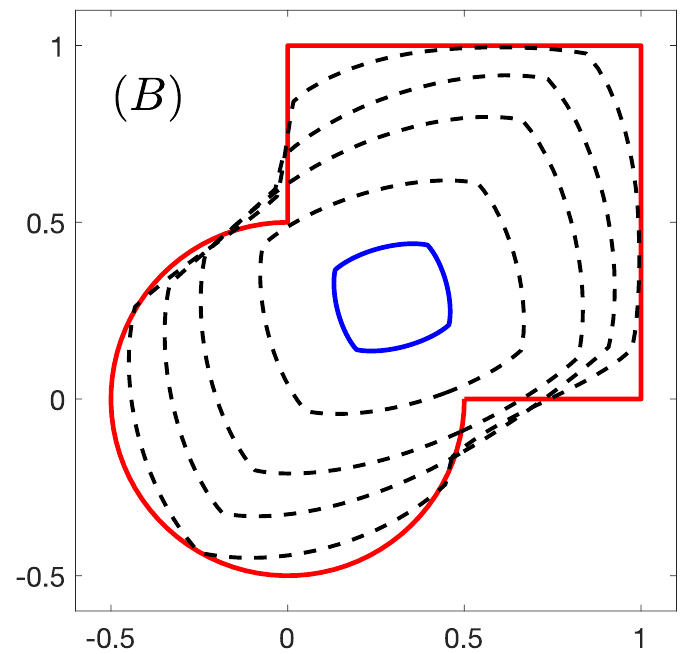}
  \caption{Morphological evolution of an initial curve with $C^0$-smoothness under anisotropic curvature flow by SP-PFEM \eqref{eqn:SP-PFEM} with $\alpha=-2$. The red lines represent the initial shapes. The blue lines represent the numerical curves at $t=0.24$. The black dashed lines represent the intermediate curves.}
  \label{fig:morph-evo-2}
\end{figure}

We consider the following non-convex initial curve with large variations in the curvature \cite{sevcovic2001evolution}: \begin{equation}
  \left\{\begin{array}{l}
    x=\cos (2\pi\rho),\qquad \rho\in [0,1],\\
    y=\frac{1}{2}\sin (2\pi\rho) +\sin(\cos (2\pi\rho))+\sin (2\pi\rho)\left(\frac{1}{5}+\sin (2\pi\rho) \sin^2 (6\pi\rho) \right),\\
  \end{array}\right.
\end{equation} and a $C^0$-smooth initial curve with sharp corners and concavities. Results are displayed in Figure~\ref{fig:morph-evo-1} and Figure~\ref{fig:morph-evo-2}, respectively.

In Figure~\ref{fig:morph-evo-1}--Figure~\ref{fig:morph-evo-2}, two types of surface energy densities are applied in our simulation: (A) $\hat{\gamma}(\theta)=1+\frac{1}{7}\cos 3\theta$; and (B) $\hat{\gamma}(\theta)=1+\frac{1}{12}\cos 4(\theta+\frac{\pi}{6})$. Observation from Figure~\ref{fig:morph-evo-1}--Figure~\ref{fig:morph-evo-2}  reveals that, under the anisotropic curvature flow, the curve will gradually shrink, while its shape tends to evolve toward a common form determined by the same anisotropy.

\subsubsection*{Evolution of a self-intersecting curve}

\begin{figure}[htbp]
  \centering
  \includegraphics[width=1\textwidth]{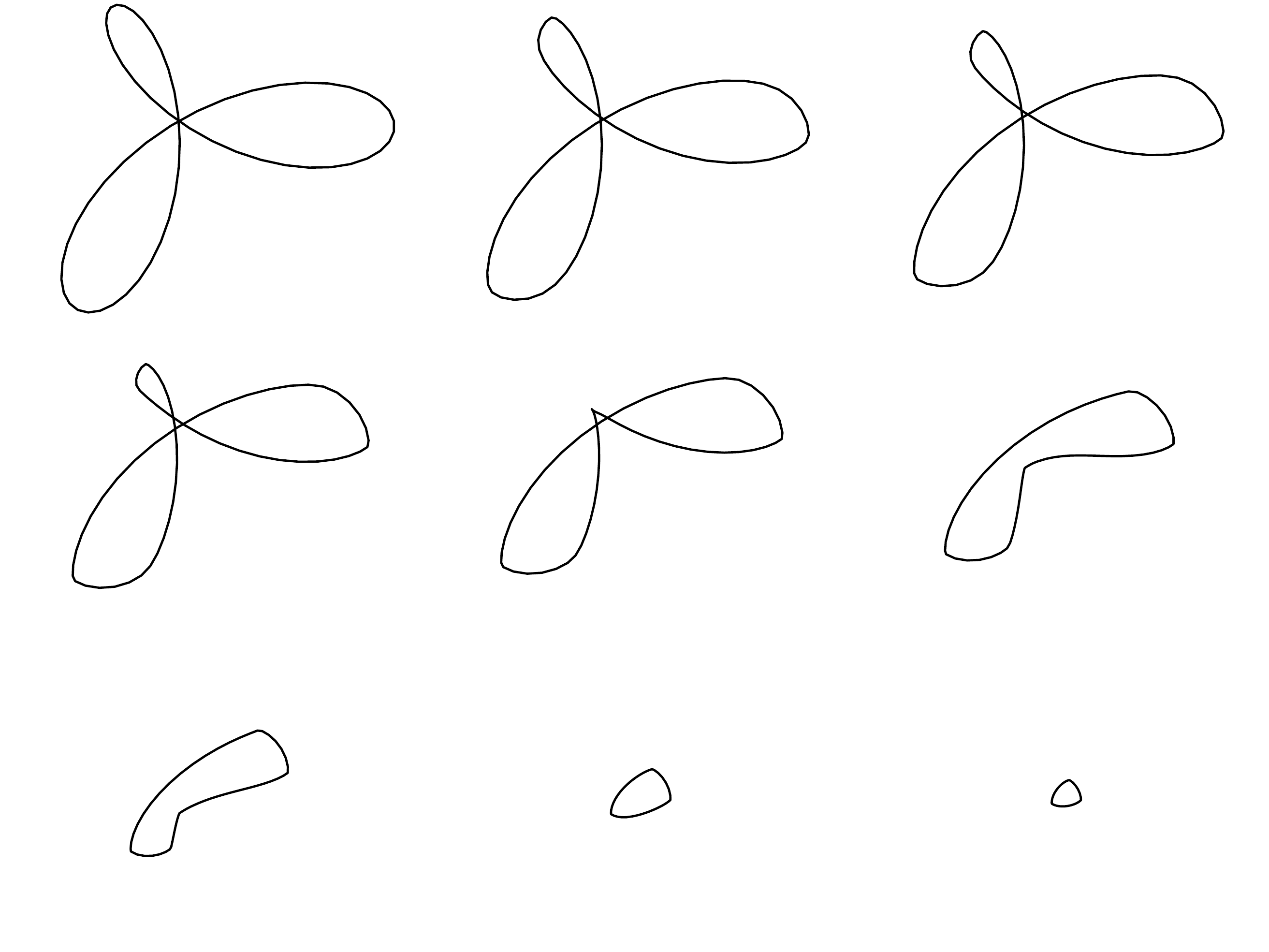}
  \caption{Plots of the anisotropic evolution of a self-intersecting initial curve at times $t=0,0.028,0.056,0.084,0.112,0.2,0.3,0.4,0.42$. A 3-fold anisotropy $\hat{\gamma}(\theta)=1+\frac{1}{9}\cos 3\theta$ is used for the surface energy density. Other parameters are set as $\alpha=-1, h=2^{-7}, \tau=0.005$.}
  \label{fig:morphrose}
\end{figure}

As is theoretically known, self-intersecting curves often develop singularities during curvature-driven evolution. The local geometry around these points becomes highly complex, posing significant challenges for numerical simulation. Figure~\ref{fig:morphrose} illustrates the anisotropic evolution of an initial curve with a triple self-intersection point. A similar example was presented in \cite[Fig 7.2]{dziuk1999discrete} for the isotropic case. The numerical results demonstrate that our method remains effective in capturing cusp singularities even in the presence of anisotropy.

\subsection{Numerical results for other anisotropic flows}

In this part, we present numerical experiments for two types of anisotropic flows that preserve area: the area-conserved anisotropic curvature flow and anisotropic surface diffusion. The following normalized area loss is introduced as an indicator to numerically demonstrate the area conservation property \cite{zhang2025stabilized}: \begin{equation}
 \left.\frac{\Delta A^h(t)}{A^h(0)}\right|_{t=t_m}\coloneqq\frac{A^{m+1}-A^0}{A^0},\qquad \forall m\geq 0.
\end{equation}

\subsubsection{Area-conserved anisotropic curvature flow}

\begin{figure}[htbp]
  \centering
  \includegraphics[width=0.8\textwidth]{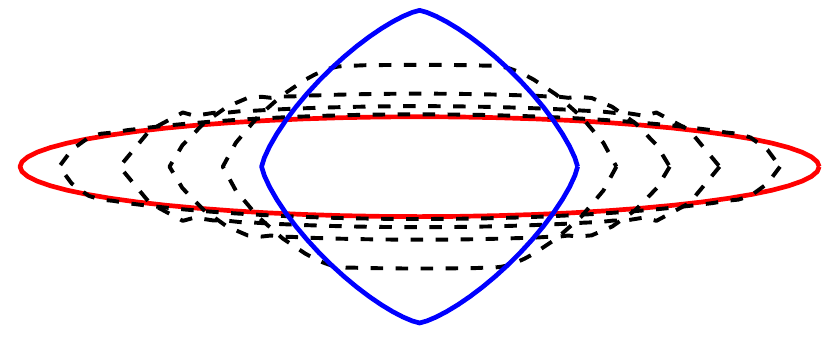}
  \caption{Morphological evolution of an ellipse with major axis 8 and minor axis 1 under area-conserved anisotropic curvature flow by SP-PFEM \eqref{eqn:SP-PFEM acmcf} with the $l^4$-norm metric anisotropy $\hat{\gamma}(\theta)=\sqrt[4]{n_1^4+n_2^4}$, where $\boldsymbol{n}=(n_1,n_2)^T=(-\sin\theta,\cos\theta)^T$. The red and blue lines represent the initial shape and the numerical equilibrium, respectively; and the black dash lines represent the intermediate curves. Parameters are set to $\alpha=2.5,(h,\tau)=(2^{-6},4^{-6})$.}
  \label{fig:acmcf}
\end{figure}

\begin{figure}[htbp]
  \centering
  \includegraphics[width=0.525\textwidth]{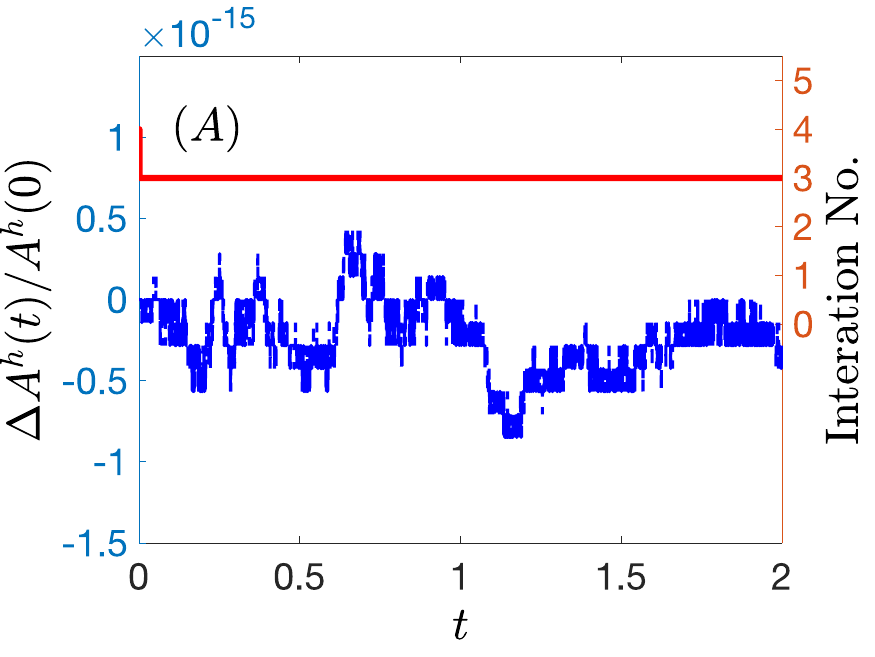}\includegraphics[width=0.475\textwidth]{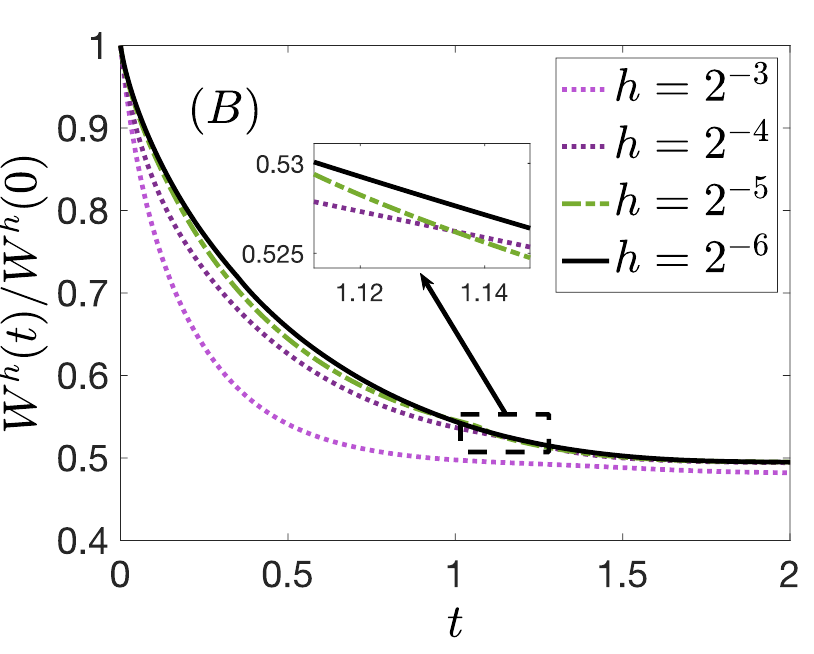}
  \caption{Temporal evolution of (A) normalized area loss (blue dash line) and iteration number (red line); and (B) normalized energy of SP-PFEM \eqref{eqn:SP-PFEM acmcf} for the anisotropic curvature flow in Figure~\ref{fig:acmcf}.}
  \label{fig:dA-energy-acmcf}
\end{figure}

Figure~\ref{fig:acmcf} illustrates the morphological evolution of an ellipse with major axis 8 and minor axis 1 under the area-conserved anisotropic curvature flow. The surface enegrgy density is the so-called $l^4$-norm metric anisotropy: $\hat{\gamma}(\theta)=\sqrt[4]{n_1^4+n_2^4}$, where $\boldsymbol{n}=(n_1,n_2)^T=(-\sin\theta,\cos\theta)^T$ \cite{bao2023symmetrized2D}. The normalized area loss and the number of Newton's iteration are plotted in Figure~\ref{fig:dA-energy-acmcf} (A). The normalized area loss is observed to be on the order of $10^{-15}$, which is very close to rounding error. This indicates that the area is conserved in the sense of machine precision. The number of Newton iterations quickly decreases from 4 at the beginning to 3. And the normalized energy is monotonically decreasing, as shown in Figure~\ref{fig:dA-energy-acmcf} (B). 

\begin{figure}[htbp]
  \centering
  \includegraphics[width=0.5\textwidth]{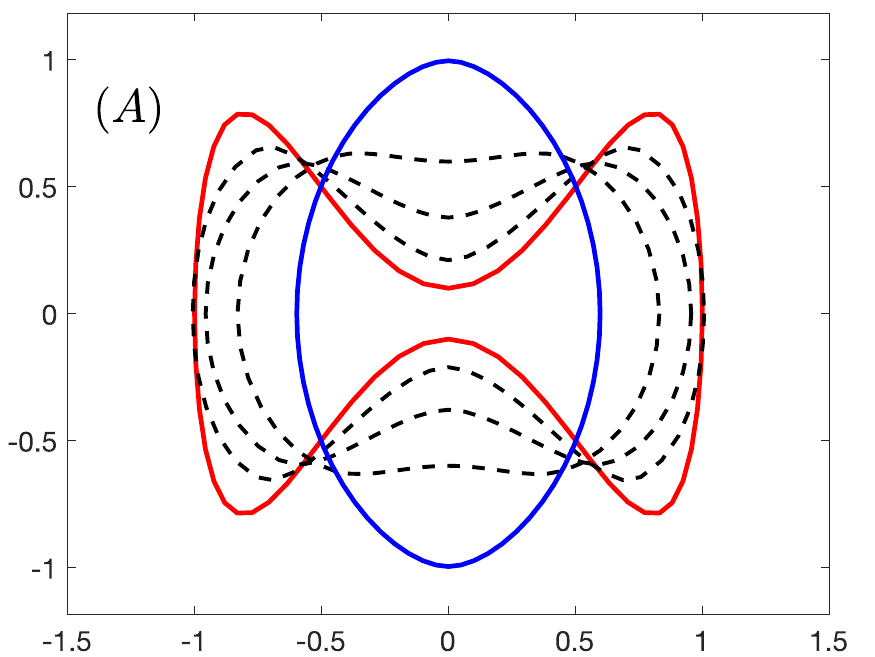}\includegraphics[width=0.5\textwidth]{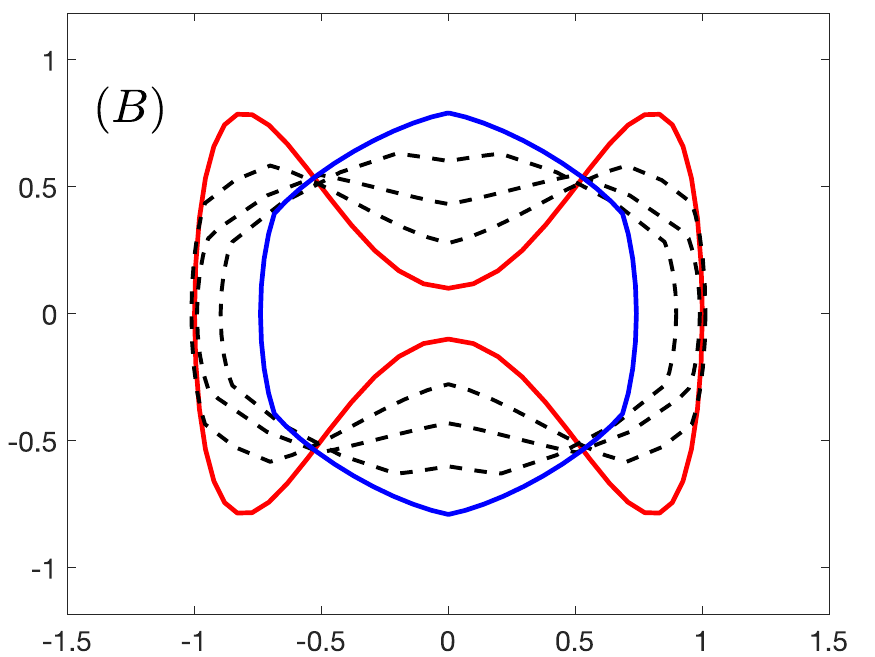}
  \caption{Morphological evolution of a bowtie-shaped curve under area-conserved anisotropic curvature flow with anisotropy (A) $\hat{\gamma}(\theta)=1+\frac{1}{28}\cos 2\theta$; and (B) $\hat{\gamma}(\theta)=1+\frac{1}{30}\cos 6\theta$. The red and blue lines represent the initial shape and the numerical equilibrium, respectively; and the black dash lines represent the intermediate curves. Other parameter are set as $\alpha=0.5, (h,\tau)=(2^{-6},4^{-6})$.}
  \label{fig:acmcf morph}
\end{figure}

\begin{figure}[htbp]
  \centering
  \includegraphics[width=1\textwidth]{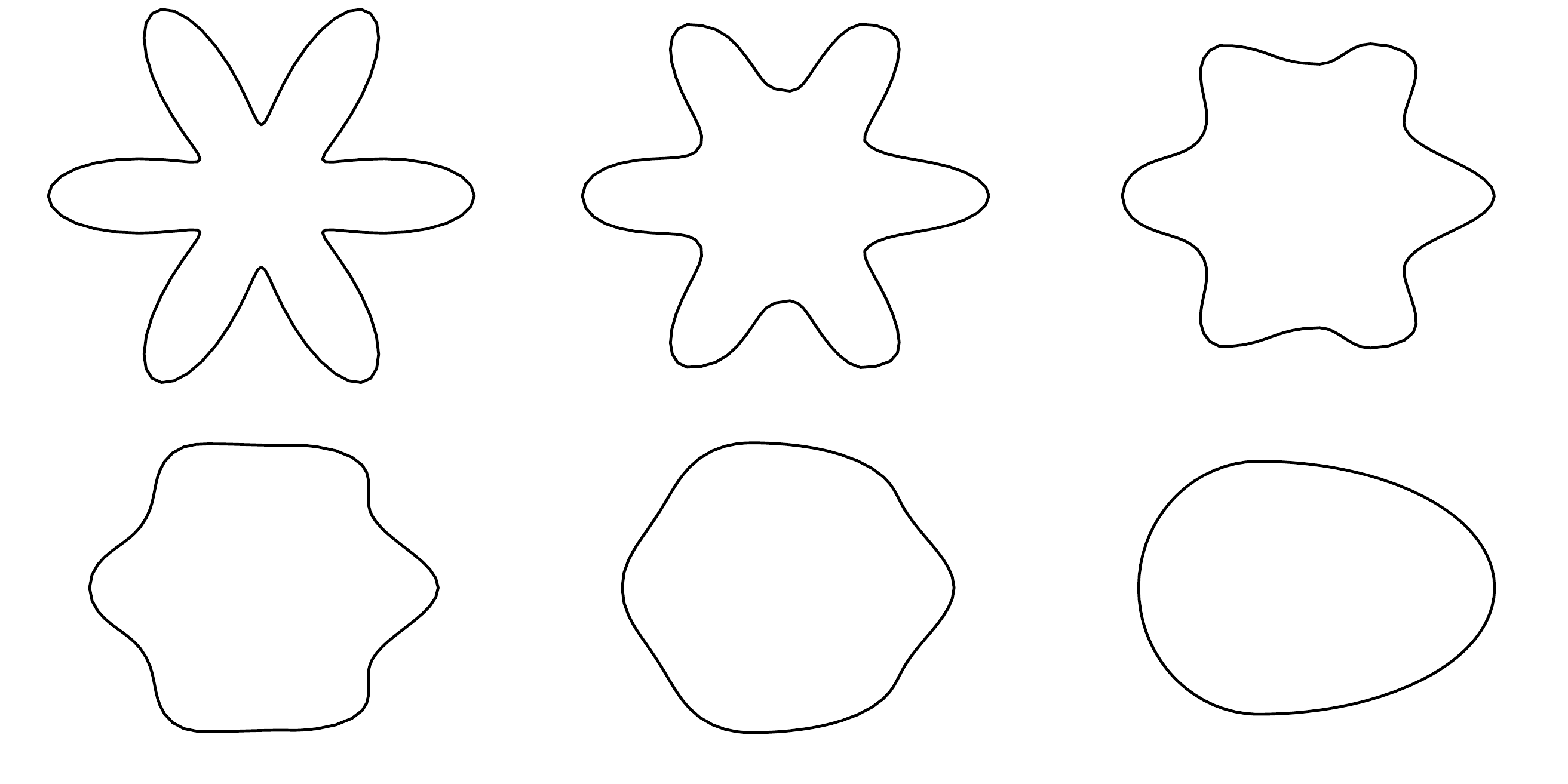}
  \caption{Snapshots of a flower initial curve under area-conserved anisotropic curvature flow with anisotropy Case II at times $t=0,0.05,0.15,0.25,0.4,2.5$. Other parameter are set as $\alpha=-0.5, h=2^{-7}, \tau=0.005$.}
  \label{fig:morphflower}
\end{figure}

We also apply SP-PFEM \eqref{eqn:SP-PFEM acmcf} to simulate the morphological evolution of some complex initial curves governed by area-conserved anisotropic curvature flow. The results are shown in Figure~\ref{fig:acmcf morph}--Figure~\ref{fig:morphflower}. Two types of curves are considered as follows: \begin{itemize}
  \item a bowtie-shaped curve: \begin{equation}
    \left\{\begin{array}{l}
      x=\cos(2\pi\rho),\\
      y=2\sin(2\pi\rho)-1.9\sin^3(2\pi\rho),
    \end{array}\right. \qquad \rho\in [0,1];
  \end{equation}
  \item a flower-shaped curve: \begin{equation}
    \left\{\begin{array}{l}
      x=\left(2+\cos(12\pi\rho)\right)\cos(2\pi\rho),\\
      y=\left(2+\cos(12\pi\rho)\right)\sin(2\pi\rho),
    \end{array}\right. \qquad \rho\in [0,1].
  \end{equation}
\end{itemize} It can be observed that, as theoretically predicted, both curves gradually evolve toward the Wulff shape corresponding to the given anisotropy \cite{yazaki2002area}.

\subsubsection{Anisotropic surface diffusion}

\begin{figure}[htbp]
  \centering
  \includegraphics[width=0.8\textwidth]{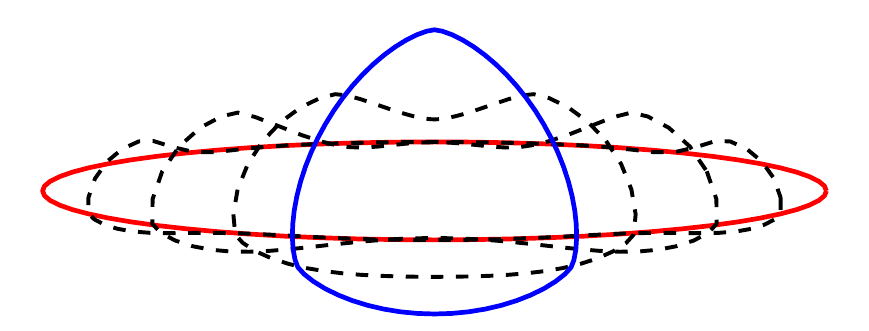}
  \caption{Morphological evolution of an ellipse with major axis 8 and minor axis 1 governed by anisotropic surface diffusion for Case I with $\beta=1/9$. The red and blue lines
    represent the initial shape and the numerical equilibrium, respectively; and the black dash lines represent the intermediate curves. Mesh sizes are set to $(h,\tau)=(2^{-6},4^{-6})$.}
  \label{fig:sd}
\end{figure}

\begin{figure}[htbp]
  \centering
  \includegraphics[width=0.515\textwidth]{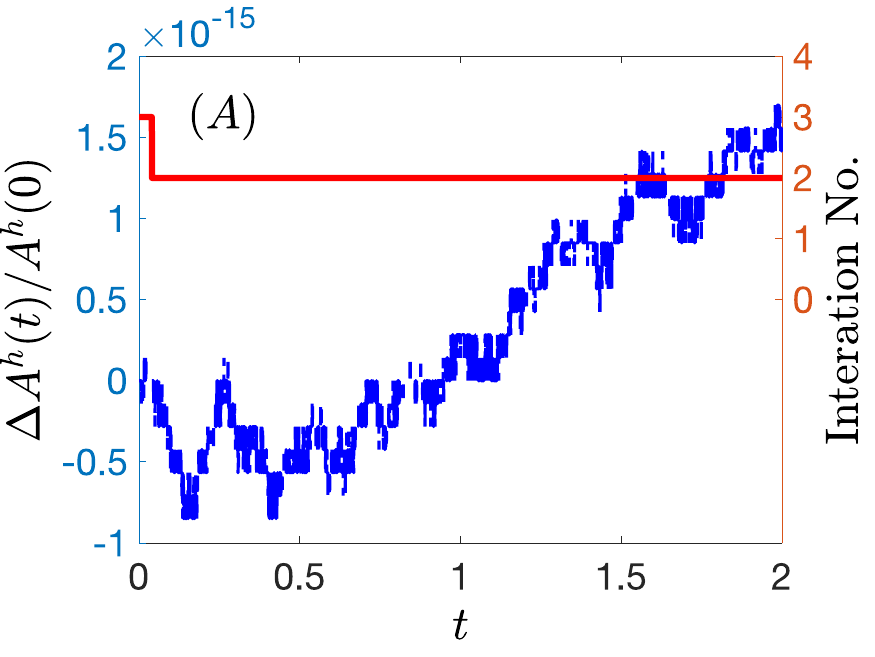}\includegraphics[width=0.485\textwidth]{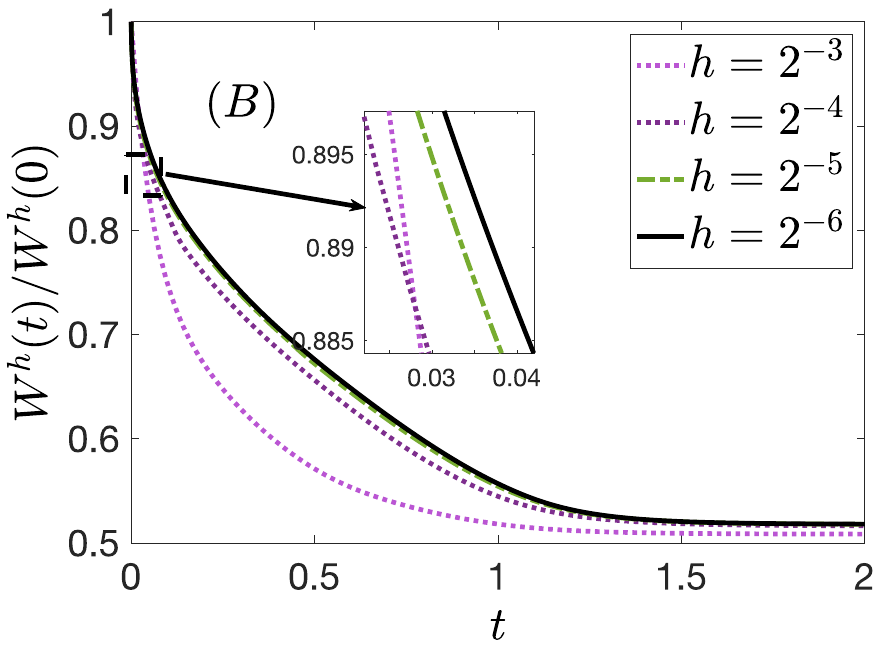}
  \caption{Temporal evolution of (A) normalized area loss (blue dash line) and iteration number (red line); and (B) normalized energy of SP-PFEM \eqref{eqn:SP-PFEM sd} for the anisotropic surface diffusion in Figure~\ref{fig:sd}.}
  \label{fig:dA-energy-sd}
\end{figure}

Figure~\ref{fig:sd}--Figure~\ref{fig:dA-energy-sd} illustrate the evolution of an 8:1 ellipse under anisotropic surface diffusion, along with the corresponding changes in the normalized area, the interation number and the normalized energy. The surface energy density is chosen as in Case I with $\beta=1/9$. Results in  Figure~\ref{fig:dA-energy-sd} confirm that our method is numerically area conservative and energy dissipative when condition \eqref{eqn:energy cond} is satisfied.
 
\begin{figure}[htbp]
  \centering
  \includegraphics[width=1\textwidth]{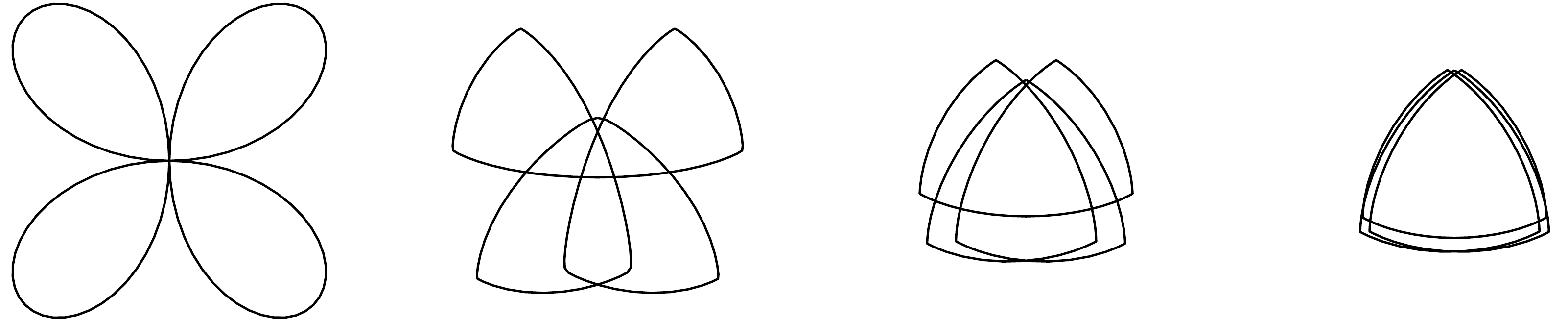}
  \caption{Anisotropic surface diffusion of a quadrifolium with anisotropy Case I at times $t=0,0.15,0.5,0.9$. Parameters are set as $\alpha=0.5, \beta=1/5, h=2^{-7}, \tau=0.001$.}
  \label{fig:morph4leaf}
\end{figure}

Figure~\ref{fig:morph4leaf} presents numerical experiments on a quadrifolium. Escher et al. \cite{escher1998surface} and Barrett et al. \cite{barrett2007parametric} have previously simulated the evolution of the quadrifolium under isotropic surface diffusion, demonstrating that the limiting curve is a triply covered circle. Our results show that the behavior under anisotropic surface diffusion is qualitatively similar: for positive time, the winding number of the curve with respect to the origin remains unchanged, and the limiting shape becomes a triply covered Wulff shape corresponding to the given anisotropy (cf. Figure~\ref{fig:morph4leaf}). 

\begin{figure}[htbp]
  \centering
  \includegraphics[width=1\textwidth]{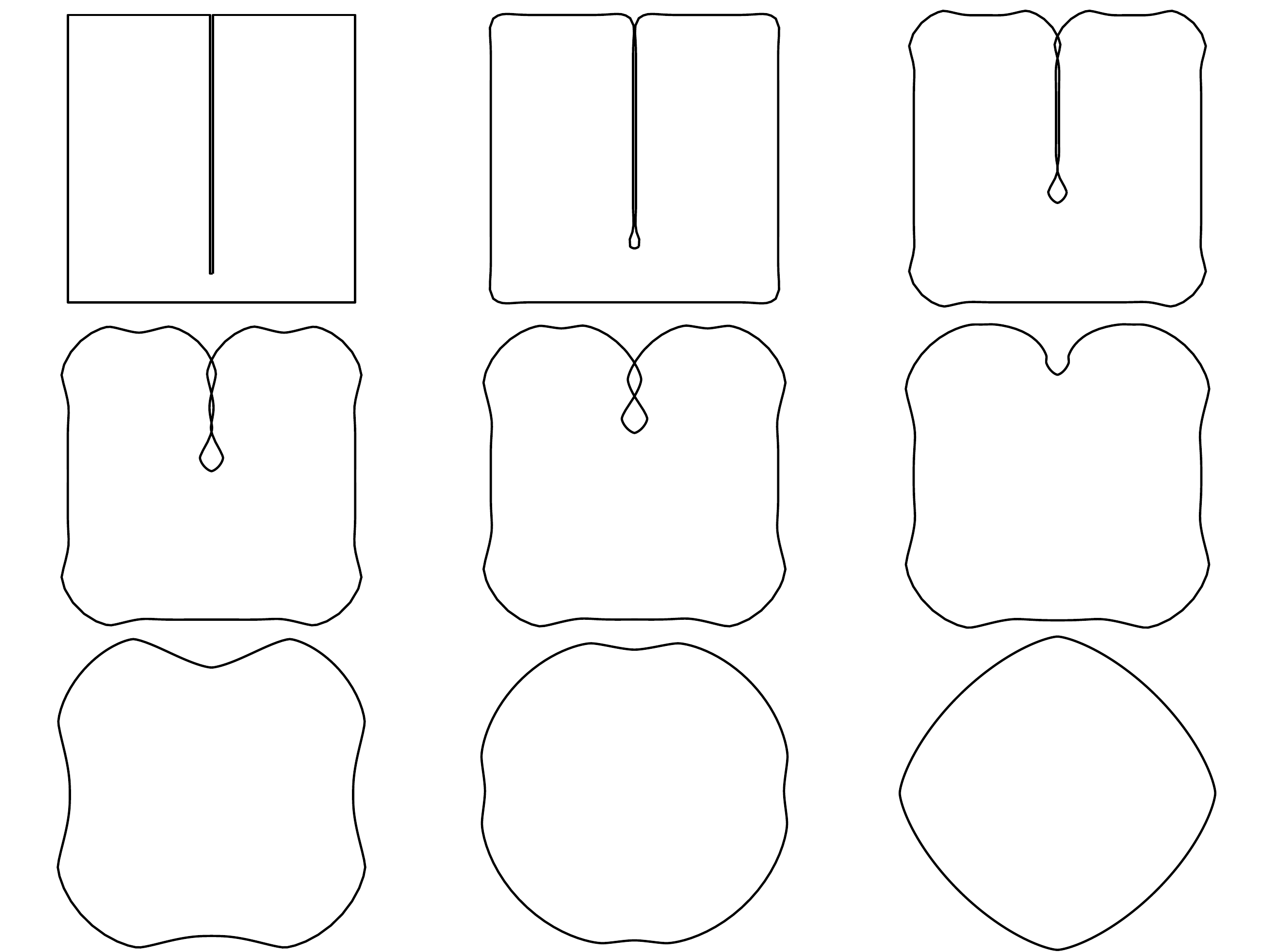}
  \caption{Anisotropic surface diffusion of an almost slit domain with 4-fold anisotropy $\hat{\gamma}(\theta)=1+\frac{1}{16}\cos 4\theta$ at times $t=0,5\times 10^{-6},7\times 10^{-5}, 2.5\times 10^{-4},5\times 10^{-4},7.5\times 10^{-4},0.002,0.02, 0.1$. Parameters are set as $a=1, h=1/288, \tau=1\times 10^{-6}$.}
  \label{fig:morphslit}
\end{figure}

The initial curve in Figure~\ref{fig:morphslit} is given by a $2\times 2$ square minus a thin rectangle ($0.02\times 1.8$). The shape was described by B{\"a}nsch \cite{bansch2005finite} as an \textit{almost slit domain}. Previous studies \cite{bansch2005finite, barrett2007parametric} have been shown that this curve undergoes a pinch-off phenomenon under isotropic surface diffusion due to a curve crossing. Our numerical simulations reveal that such behavior persists under anisotropic surface diffusion as well. Notably, compared to the methods of \cite{bansch2005finite,barrett2007parametric}, our approach guarantees exact area conservation during the entire evolution process.

\begin{figure}[htbp]
  \centering
  \includegraphics[width=1\textwidth]{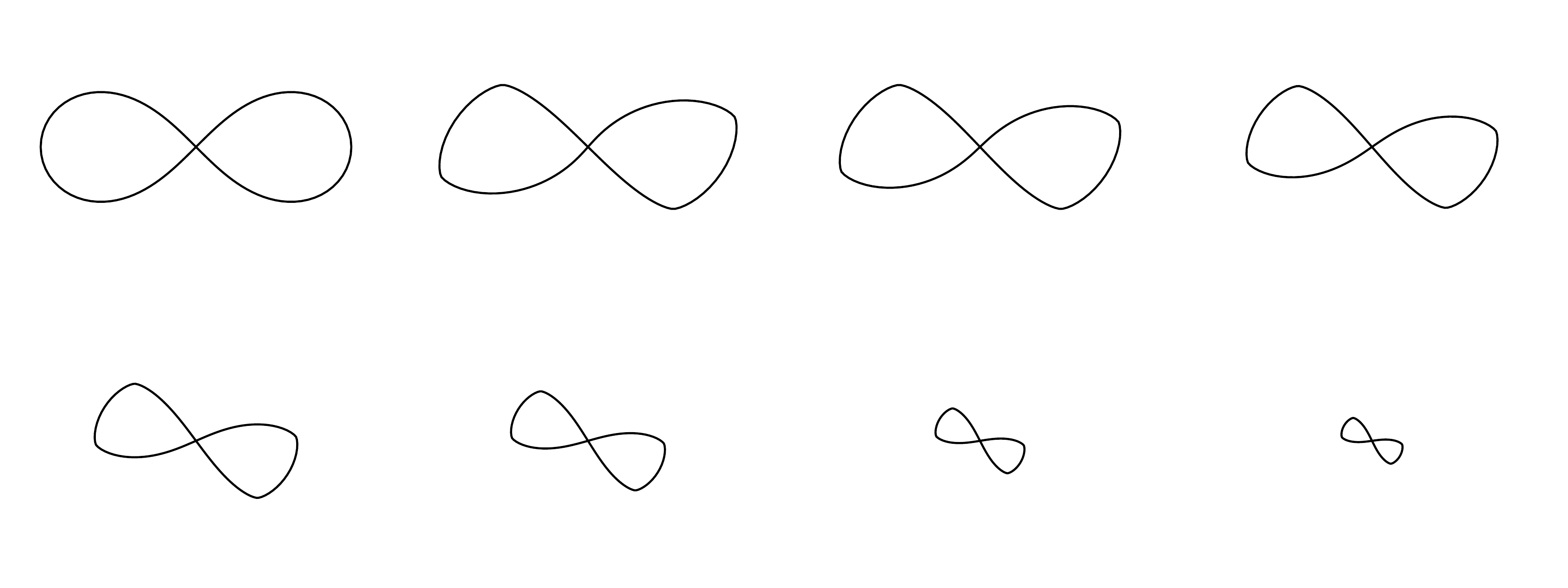}
  \includegraphics[width=1\textwidth]{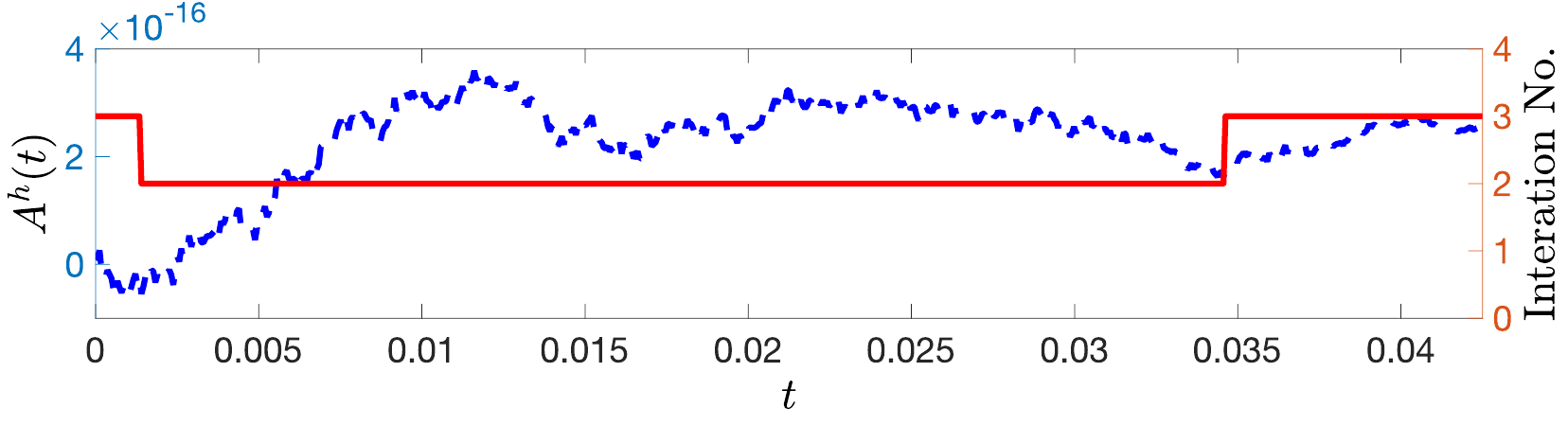}
  \caption{Evolution of the Bernoulli's lemniscate governed by anisotropic surface diffusion with anisotropy Case I at times $t=0,0.008,0.015,0.025,0.035,0.04,0.0425$. Parameters are chosen as $\beta=1/9, (h,\tau)=(2^{-7},4^{-7})$.}
  \label{fig:morph8curve}
\end{figure}

We also conducted experiments on the lemniscate of Bernoulli, the results can be found in Figure~\ref{fig:morph8curve}. As is known, the signed area enclosed by the Bernoulli's lemniscate is identically zero. Since anisotropic surface diffusion decreases the total energy (i.e. the weighted perimeter) while preserving area, it is reasonable to expect the curve to shrink to a point. Our numerical experiments support this conjecture: the observed enclosed area remains at the order of $10^{-16}$ throughout the simulation, effectively zero within machine precision. Meanwhile, each lobe of the curve gradually approaches the corresponding Wulff shape before ultimately collapsing into a single point. Interestingly, the entire curve exhibits a slow rotational motion during the evolution due to the anisotropic effects. Such a phenomenon does not occur in the isotropic setting, where the curve shrinks in place without changing orientation as shown in \cite{escher1998surface}.

\subsection{Evolution of long thin films under anisotropic surface diffusion}

The morphological evolution of crystal-grown thin films under anisotropic surface diffusion has attracted significant attention in materials science and solid-state physics \cite{thompson2012solid}, with profound applications in heterogeneous catalysis \cite{randolph2007controlling}, quantum dot manufacturing \cite{fonseca2014shapes} and solid-state dewetting \cite{xue2011pattern}. According to studies \cite{dornel2006surface,jiang2012phase}, once the island aspect ratio exceeds a certain critical threshold, the structure becomes unstable and undergoes pinch-off, resulting in the formation of multiple separate islands. In this section, we apply the SP-PFEM \eqref{eqn:SP-PFEM sd} to simulate the pinch-off phenomenon of long thin films under anisotropic surface diffusion. 

\begin{figure}[htbp]
  \centering
  \includegraphics[width=1\textwidth]{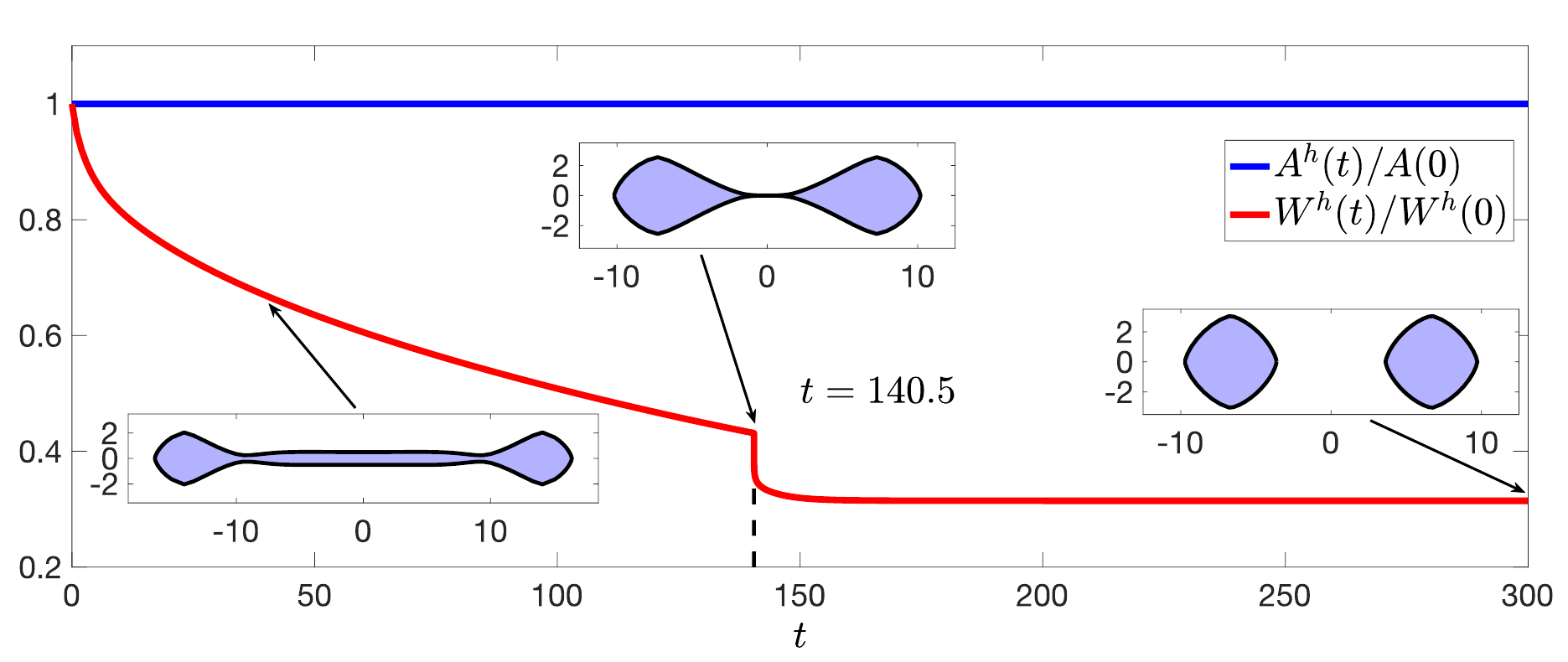}
  \caption{Normalized area and energy of a long thin film (aspect ratio of 50) governed by anisotropic surface diffusion with the 4-fold anisotropy $\hat{\gamma}(\theta)=1+\frac{1}{16}\cos 4\theta$.}
  \label{fig:EandApinchoff}
\end{figure}

First, we perform simulations of the evolution of a long thin film with an aspect ratio of 50 under 4-fold anisotropy $\hat{\gamma}(\theta)=1+\frac{1}{16}\cos 4\theta$. Figure~\ref{fig:EandApinchoff} illustrates the evolution of the normalized area and the normalized energy. In this case, the pinch-off occurs at $t=140.5$. It can be observed that,  when the pinch-off happens, the energy exhibits a sharp drop while the area reamins conserved.

\begin{figure}[htbp]
  \centering
  \includegraphics[width=1\textwidth]{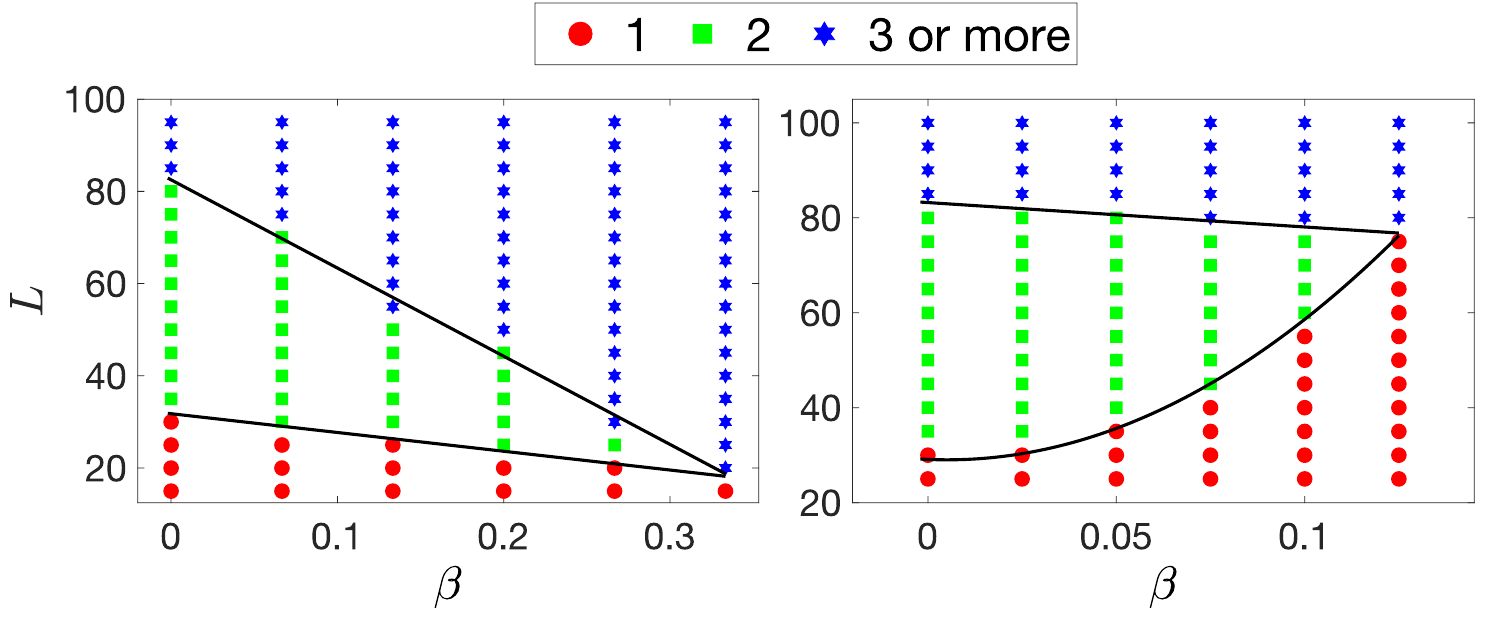}
  \caption{The number of islands formed from the pinch-off of a $2L\times 1$ long thin film as a function of $L$ and anisotropy strength $\beta$ for: (left) the 2-fold anisotropy $\hat{\gamma}(\theta)=1+\beta\cos 2\theta$; (right) the 3-fold anisotropy $\hat{\gamma}(\theta)=1+\beta\cos 4\theta$. The solid lines represent the 1-2 islands and 2-3 islands boundaries.}
  \label{fig:phasediagram1}
\end{figure}

\begin{figure}[htbp]
  \centering
  \includegraphics[width=1\textwidth]{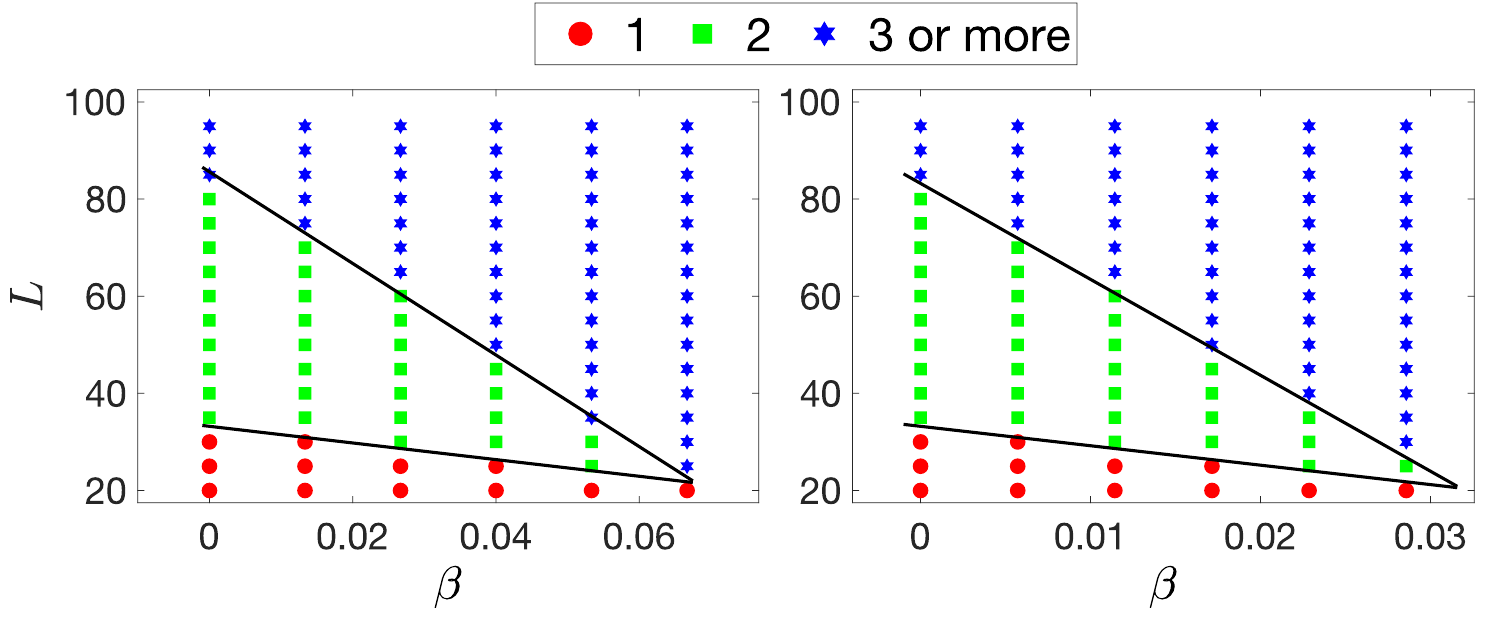}
  \caption{The number of islands formed from the pinch-off of a $2L\times 1$ long thin film as a function of $L$ and anisotropy strength $\beta$ for: (left) the 4-fold anisotropy $\hat{\gamma}(\theta)=1+\beta\cos 4\theta$; (right) the 6-fold anisotropy $\hat{\gamma}(\theta)=1+\beta\cos 6\theta$. The solid lines represent the 1-2 islands and 2-3 islands boundaries.}
  \label{fig:phasediagram2}
\end{figure}

In addition to the aspect ratio of the thin film, the anisotropy strength of the surface energy also plays a crucial role in the island formation process. We conducted experiments with varying aspect ratios and anisotropy strengths to investigate their influence on the formation of islands. The results are presented in Figure~\ref{fig:phasediagram1}-Figure~\ref{fig:phasediagram2} for commonly used $m$-fold anisotropies $\hat{\gamma}(\theta)=1+\beta \cos m\theta,\,\,m=2,3,4,6$.

For the symmetric anisotropy $\hat{\gamma}(\theta)=1+\beta\cos m\theta,\,\,m=2,4,6$, the boundaries between the 1-2 islands and 2-3 islands regions are approximated by the following linear curve fitting $L=p(m)\beta+q(m)$: \begin{itemize}
  \item 1-2 islands boundary: $p(m)=-12.23m^2+8.04m-7.86,q(m)=0.36m+31.31$;
  \item 2-3 islands boundary: $p(m)=-35.07m^2-165.40m+279.82,q(m)=0.18m+83.06$.
\end{itemize} Numerical results indicate that, for a given aspect ratio, increasing the anisotropy strength of symmetric surface energy promotes the formation of more islands. 

For the asymmetric anisotropy $\hat{\gamma}(\theta)=1+\beta\cos 3\theta$, we use a quadratic curve fitting $L=3285.70\beta^2-33.57\beta+29.11$ to determine the boundary between the 1-2 islands regions, and a linear curve fitting $L=-51.43\beta+83.21$ to identify the 2-3 islands boundary. Our numerical results show that, for smaller aspect ratios, increasing the anisotropy strength tends to produce fewer islands. This behavior is attributed to the asymmetry in surface energy. A comparison of the two sets of simulation results reveals that, for the evolution of long thin films, symmetric and asymmetric anisotropies may lead to markedly different morphological behaviors.

% \begin{figure}[htbp]
%   \centering
%   \includegraphics[width=0.333\textwidth]{pinchoff_3fold1.pdf}\includegraphics[width=0.333\textwidth]{pinchoff_3fold2.pdf}\includegraphics[width=0.333\textwidth]{pinchoff_3fold3.pdf}
%   \includegraphics[width=1\textwidth]{pinchoff_3fold4.pdf}
%   \caption{Snapshots in the evolution of a $140\times 1$ thin film under anisotropic surface diffusion with the 3-fold anisotropy $\hat{\gamma}(\theta)=1+\frac{1}{9}\cos 3\theta$ at different times.}
%   \label{fig:pinchoff_3fold}
% \end{figure}

% \begin{figure}[htbp]
%   \centering
%   \includegraphics[width=1\textwidth]{pinchoff_6fold.pdf}
%   \caption{Morphological evolution of a long thin film with aspect ratio of $80$ under anisotropic surface diffusion with the isotropic surface energy density (left) and the 6-fold anisotropy $\hat{\gamma}(\theta)=1+\frac{1}{35.1}\cos 6\theta$ (right).}
%   \label{fig:pinchoff_6fold}
% \end{figure}

% Finally, we present the morphological evolution of ultrathin films in Figures~\ref{fig:pinchoff_3fold} and Figure~\ref{fig:pinchoff_6fold}. As illustrated in Figure~\ref{fig:pinchoff_3fold}, due to the asymmetry in surface energy, the film exhibits distinct behaviors on its upper and lower sides: the upper surface undergoes the formation of pronounced ridges and valleys, while the lower surface remains relatively smooth and flat. Figure~\ref{fig:pinchoff_6fold} shows that, for the same aspect ratio, symmetric anisotropic surface energy promotes the formation of more islands compared to the isotropic case.

\section{Conclusion}\label{sec:conclusion}

We provided a detailed analysis and comparsion of the structure-preserving parametric finite element methods (SP-PFEM) for anisotropic geomertic flows. By introducing a hyperparameter $\alpha$, we are able to express all possible surface energy matrices in a unified form, thereby incorporating the SP-PFEM schemes into a single analytical framework. It is proven that $3\hat{\gamma}(\theta)-\hat{\gamma}(\theta-\pi)>0$ serves as a common energy stability condition for all surface energy matrices. In the special case where the surface energy matrix is symmetric, this condition can be improved to $3\hat{\gamma}(\theta)-\hat{\gamma}(\theta-\pi)\geq 0$, which is both necessary and sufficient for the energy stability. Apart from mesh quality, the proposed method exhibits consistent performance across different values of $\alpha$ in terms of accuracy, computational efficiency, and structure preservation, indicating a certain degree of robustness. The method presented in this paper can be easily extended to general anisotropic geometric flows, including area-conserved anisotropic curvature flow and anisotropic surface diffusion. It provides an effective framework for developing structure-preserving numerical algorithms for general anisotropic curvature-driven problems.

\section*{Acknowledgements}

We sincerely thank the reviewer for the valuable comments and suggestions.%, which have significantly improved the quality of the paper.} 
This work was partially supported by the Ministry of Education of Singapore under its AcRF Tier 1 funding A-8003584-00-00 (W. Bao), the Alexander von Humboldt Foundation (Y. Li), the National Natural Science Foundation of China in the Division of Mathematical Sciences Project No. 12471342 (W. Ying) and Zhiyuan Honors Program for Graduate Students in Shanghai Jiao Tong University (Y. Zhang).

% \clearpage

%\newpage

\printbibliography

@article{li2021energy,
  title={An energy-stable parametric finite element method for anisotropic surface diffusion},
  author={Li, Yifei and Bao, Weizhu},
  journal={Journal of Computational Physics},
  volume={446},
  pages={110658},
  year={2021},
  publisher={Elsevier}
}

@article{bao2023symmetrized2D,
  title={A symmetrized parametric finite element method for anisotropic surface diffusion of closed curves},
  author={Bao, Weizhu and Jiang, Wei and Li, Yifei},
  journal={SIAM Journal on Numerical Analysis},
  volume={61},
  number={2},
  pages={617--641},
  year={2023},
  publisher={SIAM}
}

@article{bao2021structure,
  title={A structure-preserving parametric finite element method for surface diffusion},
  author={Bao, Weizhu and Zhao, Quan},
  journal={SIAM Journal on Numerical Analysis},
  volume={59},
  number={5},
  pages={2775--2799},
  year={2021},
  publisher={SIAM}
}

@article{zhang2025stabilized,
  title={A stabilized parametric finite element method for surface diffusion with an arbitrary surface energy},
  author={Zhang, Yulin and Li, Yifei and Ying, Wenjun},
  journal={Journal of Computational Physics},
  volume={523},
  pages={113605},
  year={2025},
  publisher={Elsevier}
}

@article{ye2010mechanisms,
  title={Mechanisms of complex morphological evolution during solid-state dewetting of single-crystal nickel thin films},
  author={Ye, Jongpil and Thompson, Carl V},
  journal={Applied Physics Letters},
  volume={97},
  number={7},
  year={2010},
  publisher={AIP Publishing}
}

@article{gurtin2002interface,
  title={Interface Evolution in Three Dimensions with Curvature-Dependent Energy and Surface Diffusion: Interface-Controlled Evolution, Phase Transitions, Epitaxial Growth of Elastic Films},
  author={Gurtin, Morton E and Jabbour, Michel E},
  journal={Archive for rational mechanics and analysis},
  volume={163},
  pages={171--208},
  year={2002},
  publisher={Springer}
}

@article{fonseca2014shapes,
  title={Shapes of epitaxially grown quantum dots},
  author={Fonseca, Irene and Pratelli, Aldo and Zwicknagl, Barbara},
  journal={Archive for Rational Mechanics and Analysis},
  volume={214},
  pages={359--401},
  year={2014},
  publisher={Springer}
}

@article{randolph2007controlling,
  title={Controlling thin film structure for the dewetting of catalyst nanoparticle arrays for subsequent carbon nanofiber growth},
  author={Randolph, SJ and Fowlkes, JD and Melechko, AV and Klein, KL and Meyer, HM and Simpson, ML and Rack, PD},
  journal={Nanotechnology},
  volume={18},
  number={46},
  pages={465304},
  year={2007},
  publisher={IOP Publishing}
}

@book{clarenz2000anisotropic,
  title={Anisotropic geometric diffusion in surface processing},
  author={Clarenz, Ulrich and Diewald, Udo and Rumpf, Martin},
  year={2000},
  publisher={IEEE}
}

@article{jiang2012phase,
  title={Phase field approach for simulating solid-state dewetting problems},
  author={Jiang, Wei and Bao, Weizhu and Thompson, Carl V and Srolovitz, David J},
  journal={Acta materialia},
  volume={60},
  number={15},
  pages={5578--5592},
  year={2012},
  publisher={Elsevier}
}

@article{jiang2018solid,
  title={Solid-state dewetting on curved substrates},
  author={Jiang, Wei and Wang, Yan and Srolovitz, David J and Bao, Weizhu},
  journal={Physical Review Materials},
  volume={2},
  number={11},
  pages={113401},
  year={2018},
  publisher={APS}
}

@article{thompson2012solid,
  title={Solid-state dewetting of thin films},
  author={Thompson, Carl V},
  journal={Annual Review of Materials Research},
  volume={42},
  pages={399--434},
  year={2012},
  publisher={Annual Reviews}
}

@article{bao2017parametric,
  title={A parametric finite element method for solid-state dewetting problems with anisotropic surface energies},
  author={Bao, Weizhu and Jiang, Wei and Wang, Yan and Zhao, Quan},
  journal={Journal of Computational Physics},
  volume={330},
  pages={380--400},
  year={2017},
  publisher={Elsevier}
}

@incollection{einstein2015equilibrium,
  title={Equilibrium shape of crystals},
  author={Einstein, Theodore L},
  booktitle={Handbook of Crystal Growth},
  pages={215--264},
  year={2015},
  publisher={Elsevier}
}

@article{taylor1994linking,
  title={Linking anisotropic sharp and diffuse surface motion laws via gradient flows},
  author={Taylor, Jean E and Cahn, John W},
  journal={Journal of Statistical Physics},
  volume={77},
  pages={183--197},
  year={1994},
  publisher={Springer}
}

@article{barrett2008numerical,
  title={Numerical approximation of anisotropic geometric evolution equations in the plane},
  author={Barrett, John W and Garcke, Harald and N{\"u}rnberg, Robert},
  journal={IMA journal of numerical analysis},
  volume={28},
  number={2},
  pages={292--330},
  year={2008},
  publisher={Oxford University Press}
}

@article{barrett2008variational,
  title={A variational formulation of anisotropic geometric evolution equations in higher dimensions},
  author={Barrett, John W and Garcke, Harald and N{\"u}rnberg, Robert},
  journal={Numerische Mathematik},
  volume={109},
  number={1},
  pages={1--44},
  year={2008},
  publisher={Springer}
}

@article{deckelnick2005computation,
  title={Computation of geometric partial differential equations and mean curvature flow},
  author={Deckelnick, Klaus and Dziuk, Gerhard and Elliott, Charles M},
  journal={Acta numerica},
  volume={14},
  pages={139--232},
  year={2005},
  publisher={Cambridge University Press}
}

@article{du2010tangent,
  title={A tangent-plane marker-particle method for the computation of three-dimensional solid surfaces evolving by surface diffusion on a substrate},
  author={Du, Ping and Khenner, Mikhail and Wong, Harris},
  journal={Journal of Computational Physics},
  volume={229},
  number={3},
  pages={813--827},
  year={2010},
  publisher={Elsevier}
}

@article{du2020phase,
  title={The phase field method for geometric moving interfaces and their numerical approximations},
  author={Du, Qiang and Feng, Xiaobing},
  journal={Handbook of numerical analysis},
  volume={21},
  pages={425--508},
  year={2020},
  publisher={Elsevier}
}

@article{bansch2005finite,
  title={A finite element method for surface diffusion: the parametric case},
  author={B{\"a}nsch, Eberhard and Morin, Pedro and Nochetto, Ricardo H},
  journal={Journal of Computational Physics},
  volume={203},
  number={1},
  pages={321--343},
  year={2005},
  publisher={Elsevier}
}

@article{barrett2007parametric,
  title={A parametric finite element method for fourth order geometric evolution equations},
  author={Barrett, John W and Garcke, Harald and N{\"u}rnberg, Robert},
  journal={Journal of Computational Physics},
  volume={222},
  number={1},
  pages={441--467},
  year={2007},
  publisher={Elsevier}
}

@article{barrett2008parametric,
  title={On the parametric finite element approximation of evolving hypersurfaces in R3},
  author={Barrett, John W and Garcke, Harald and N{\"u}rnberg, Robert},
  journal={Journal of Computational Physics},
  volume={227},
  number={9},
  pages={4281--4307},
  year={2008},
  publisher={Elsevier}
}

@incollection{barrett2020parametricbook,
  title={Parametric finite element approximations of curvature-driven interface evolutions},
  author={Barrett, John W and Garcke, Harald and N{\"u}rnberg, Robert},
  booktitle={Handbook of numerical analysis},
  volume={21},
  pages={275--423},
  year={2020},
  publisher={Elsevier}
}

@article{jiang2016solid,
  title={Solid-state dewetting and island morphologies in strongly anisotropic materials},
  author={Jiang, Wei and Wang, Yan and Zhao, Quan and Srolovitz, David J and Bao, Weizhu},
  journal={Scripta Materialia},
  volume={115},
  pages={123--127},
  year={2016},
  publisher={Elsevier}
}

@article{bao2022volume,
  title={Volume-preserving parametric finite element methods for axisymmetric geometric evolution equations},
  author={Bao, Weizhu and Garcke, Harald and N{\"u}rnberg, Robert and Zhao, Quan},
  journal={Journal of Computational Physics},
  volume={460},
  pages={111180},
  year={2022},
  publisher={Elsevier}
}

@article{bao2024structure,
  title={A structure-preserving parametric finite element method for geometric flows with anisotropic surface energy},
  author={Bao, Weizhu and Li, Yifei},
  journal={Numerische Mathematik},
  volume={156},
  pages={609-639},
  year={2024},
  publisher={Springer}
}

@article{zhao2021energy,
  title={An energy-stable parametric finite element method for simulating solid-state dewetting},
  author={Zhao, Quan and Jiang, Wei and Bao, Weizhu},
  journal={IMA Journal of Numerical Analysis},
  volume={41},
  number={3},
  pages={2026--2055},
  year={2021},
  publisher={Oxford University Press}
}

@article{bao2025unified,
  title={A unified structure-preserving parametric finite element method for anisotropic surface diffusion},
  author={Bao, Weizhu and Li, Yifei},
  journal={Mathematics of Computation},
  volume={94},
  number={355},
  pages={2113--2149},
  year={2025}
}

@article{jiang2019sharp,
  title={Sharp-interface approach for simulating solid-state dewetting in two dimensions: A Cahn--Hoffman $\xi$-vector formulation},
  author={Jiang, Wei and Zhao, Quan},
  journal={Physica D: Nonlinear Phenomena},
  volume={390},
  pages={69--83},
  year={2019},
  publisher={Elsevier}
}

@book{reynolds1983papers,
  title={Papers on mechanical and physical subjects},
  author={Reynolds, Osborne},
  year={1983},
  publisher={CUP Archive}
}

@article{li2025structure,
  title={A structure-preserving parametric finite element method with optimal energy stability condition for anisotropic surface diffusion},
  author={Li, Yifei and Ying, Wenjun and Zhang, Yulin},
  journal={Journal of Scientific Computing},
  volume={104},
  number={3},
  pages={76},
  year={2025},
  publisher={Springer}
}

@article{barrett2007variational,
  title={On the variational approximation of combined second and fourth order geometric evolution equations},
  author={Barrett, John W and Garcke, Harald and N{\"u}rnberg, Robert},
  journal={SIAM Journal on Scientific Computing},
  volume={29},
  number={3},
  pages={1006--1041},
  year={2007},
  publisher={SIAM}
}

@article{bao2025structure,
  title={A structure-preserving parametric finite element method for solid-state dewetting on curved substrates},
  author={Bao, Weizhu and Li, Yifei and Zhao, Quan},
  journal={Communications in Nonlinear Science and Numerical Simulation},
  pages={108767},
  year={2025},
  publisher={Elsevier}
}

@article{alvarez1993axioms,
  title={Axioms and fundamental equations of image processing},
  author={Alvarez, Luis and Guichard, Fr{\'e}d{\'e}ric and Lions, Pierre -Louis and Morel, Jean -Michel},
  journal={Archive for rational mechanics and analysis},
  volume={123},
  pages={199--257},
  year={1993},
  publisher={Springer}
}

@article{sapiro1994affine,
  title={On affine plane curve evolution},
  author={Sapiro, Guillermo and Tannenbaum, Allen},
  journal={Journal of functional analysis},
  volume={119},
  number={1},
  pages={79--120},
  year={1994},
  publisher={Elsevier}
}

@article{sevcovic2001evolution,
  title={Evolution of plane curves driven by a nonlinear function of curvature and anisotropy},
  author={Sevcovic, Daniel and Mikula, Karol},
  journal={SIAM Journal on Applied Mathematics},
  volume={61},
  number={5},
  pages={1473--1501},
  year={2001},
  publisher={SIAM}
}

@article{dziuk1999discrete,
  title={Discrete anisotropic curve shortening flow},
  author={Dziuk, Gerhard},
  journal={SIAM journal on numerical analysis},
  volume={36},
  number={6},
  pages={1808--1830},
  year={1999},
  publisher={SIAM}
}

@article{yazaki2002area,
  title={On an area-preserving crystalline motion},
  author={Yazaki, Shigetoshi},
  journal={Calculus of Variations and Partial Differential Equations},
  volume={14},
  pages={85--105},
  year={2002},
  publisher={Springer}
}

@article{escher1998surface,
  title={The surface diffusion flow for immersed hypersurfaces},
  author={Escher, Joachim and Mayer, Uwe F and Simonett, Gieri},
  journal={SIAM journal on mathematical analysis},
  volume={29},
  number={6},
  pages={1419--1433},
  year={1998},
  publisher={SIAM}
}

@article{elliott1997diffusional,
  title={Diffusional phase transitions in multicomponent systems with a concentration dependent mobility matrix},
  author={Elliott, Charles M and Garcke, Harald},
  journal={Physica D: Nonlinear Phenomena},
  volume={109},
  number={3-4},
  pages={242--256},
  year={1997},
  publisher={Elsevier}
}

@article{escher2001limiting,
  title={On a limiting motion and self-intersections of curves moved by the intermediate surface diffusion flow},
  author={Escher, J and Giga, Y and Ito, K},
  journal={Nonlinear Analysis: Theory, Methods \& Applications},
  volume={47},
  number={6},
  pages={3717--3728},
  year={2001},
  publisher={Elsevier}
}

@article{burger2007level,
  title={A level set approach to anisotropic flows with curvature regularization},
  author={Burger, Martin and Hau{\ss}er, Frank and St{\"o}cker, Christina and Voigt, Axel},
  journal={Journal of computational physics},
  volume={225},
  number={1},
  pages={183--205},
  year={2007},
  publisher={Elsevier}
}

@article{osher2001level,
  title={Level set methods: an overview and some recent results},
  author={Osher, Stanley and Fedkiw, Ronald P},
  journal={Journal of Computational physics},
  volume={169},
  number={2},
  pages={463--502},
  year={2001},
  publisher={Elsevier}
}

@article{maxwell2025level,
  title={A level-set method for simulating solid-state dewetting in systems with strong crystalline anisotropy},
  author={Maxwell, AL and Thompson, Carl V and Carter, W Craig},
  journal={Acta Materialia},
  volume={282},
  pages={120368},
  year={2025},
  publisher={Elsevier}
}

@article{deckelnick2005fully,
  title={Fully discrete finite element approximation for anisotropic surface diffusion of graphs},
  author={Deckelnick, Klaus and Dziuk, Gerhard and Elliott, Charles M},
  journal={SIAM Journal on Numerical Analysis},
  volume={43},
  number={3},
  pages={1112--1138},
  year={2005},
  publisher={SIAM}
}

@phdthesis{wang2016modeling,
  title={Modeling and simulation for solid-state dewetting problems in two dimensions},
  type={PhD thesis},
  school={National University of Singapore},
  author={Wang, Yan},
  year={2016}
}

@article{brassel2011modified,
  title={A modified phase field approximation for mean curvature flow with conservation of the volume},
  author={Brassel, Morgan and Bretin, Elie},
  journal={Mathematical Methods in the Applied Sciences},
  volume={34},
  number={10},
  pages={1157--1180},
  year={2011}
}

@article{kovacs2019convergent,
  title={A convergent evolving finite element algorithm for mean curvature flow of closed surfaces},
  author={Kov{\'a}cs, Bal{\'a}zs and Li, Buyang and Lubich, Christian},
  journal={Numerische Mathematik},
  volume={143},
  pages={797--853},
  year={2019},
  publisher={Springer}
}

@article{kovacs2021convergent,
  title={A convergent evolving finite element algorithm for Willmore flow of closed surfaces},
  author={Kov{\'a}cs, Bal{\'a}zs and Li, Buyang and Lubich, Christian},
  journal={Numerische Mathematik},
  volume={149},
  number={3},
  pages={595--643},
  year={2021},
  publisher={Springer}
}

@article{hu2022evolving,
  title={Evolving finite element methods with an artificial tangential velocity for mean curvature flow and Willmore flow},
  author={Hu, Jiashun and Li, Buyang},
  journal={Numerische Mathematik},
  volume={152},
  number={1},
  pages={127--181},
  year={2022},
  publisher={Springer}
}

@article{dziuk1990algorithm,
  title={An algorithm for evolutionary surfaces},
  author={Dziuk, Gerhard},
  journal={Numerische Mathematik},
  volume={58},
  number={1},
  pages={603--611},
  year={1990},
  publisher={Springer}
}

@article{dziuk1994convergence,
  title={Convergence of a semi-discrete scheme for the curve shortening flow},
  author={Dziuk, Gerhard},
  journal={Mathematical Models and Methods in Applied Sciences},
  volume={4},
  number={04},
  pages={589--606},
  year={1994},
  publisher={World Scientific}
}

@article{barrett2010parametric,
  title={Parametric approximation of surface clusters driven by isotropic and anisotropic surface energies},
  author={Barrett, John W and Garcke, Harald and N{\"u}rnberg, Robert},
  journal={Interfaces and Free Boundaries},
  volume={12},
  number={2},
  pages={187--234},
  year={2010}
}

@article{garcke2023structure,
  title={Structure-preserving discretizations of two-phase Navier--Stokes flow using fitted and unfitted approaches},
  author={Garcke, Harald and N{\"u}rnberg, Robert and Zhao, Quan},
  journal={Journal of Computational Physics},
  volume={489},
  pages={112276},
  year={2023},
  publisher={Elsevier}
}

@article{barrett2010stable,
  title={On stable parametric finite element methods for the Stefan problem and the Mullins--Sekerka problem with applications to dendritic growth},
  author={Barrett, John W and Garcke, Harald and N{\"u}rnberg, Robert},
  journal={Journal of Computational Physics},
  volume={229},
  number={18},
  pages={6270--6299},
  year={2010},
  publisher={Elsevier}
}

@article{bao2023energy,
  title={An energy-stable parametric finite element method for simulating solid-state dewetting problems in three dimensions},
  author={Bao, Weizhu and Zhao, Quan},
  journal={Journal of Computational Mathematics},
  volume={41},
  number={4},
  pages={771--796},
  year={2023}
}

@article{garcke2025variational,
  title={A variational front-tracking method for multiphase flow with triple junctions},
  author={Garcke, Harald and N{\"u}rnberg, Robert and Zhao, Quan},
  journal={Mathematics of Computation},
  year={2025}
}

@article{eto2025parametric,
  title={A parametric finite element method for a degenerate multi-phase Stefan problem with triple junctions},
  author={Eto, Tokuhiro and Garcke, Harald and N{\"u}rnberg, Robert},
  journal={arXiv preprint arXiv:2505.13165},
  year={2025}
}

@article{taylor1992ii,
  title={II—mean curvature and weighted mean curvature},
  author={Taylor, Jean E},
  journal={Acta metallurgica et materialia},
  volume={40},
  number={7},
  pages={1475--1485},
  year={1992},
  publisher={Elsevier}
}

@article{palmer1998stability,
  title={Stability of the Wulff shape},
  author={Palmer, Bennett},
  journal={Proceedings of the American Mathematical Society},
  volume={126},
  number={12},
  pages={3661--3667},
  year={1998}
}

@article{xue2011pattern,
  title={Pattern formation by dewetting of polymer thin film},
  author={Xue, Longjian and Han, Yanchun},
  journal={Progress in Polymer Science},
  volume={36},
  number={2},
  pages={269--293},
  year={2011},
  publisher={Elsevier}
}

@article{dornel2006surface,
  title={Surface diffusion dewetting of thin solid films: Numerical method and application to Si/ SiO 2},
  author={Dornel, Erwan and Barbe, Jean-Charles and De Cr{\'e}cy, F and Lacolle, G and Eymery, Jo{\"e}l},
  journal={Physical Review B—Condensed Matter and Materials Physics},
  volume={73},
  number={11},
  pages={115427},
  year={2006},
  publisher={APS}
}

@article{garcke2025isoparametric,
  title={Isoparametric finite element methods for mean curvature flow and surface diffusion},
  author={Garcke, Harald and N{\"u}rnberg, Robert and Praetorius, Simon and Zhang, Ganghui},
  journal={Journal of Computational Physics},
  pages={114248},
  year={2025},
  publisher={Elsevier}
}

\end{document}